\documentclass[12pt]{amsart}
\usepackage{kantlipsum}
\usepackage{mathrsfs}
\usepackage{amssymb}
\usepackage{hyperref}
\usepackage{graphicx}
\usepackage{color}
\usepackage{amsthm}
\usepackage{amsmath}
\usepackage{amsfonts}
\usepackage{comment}
\usepackage{fancyhdr}
\usepackage{chngcntr}
\calclayout
\counterwithin{table}{section}
\def\Z{\mathbb{Z}}
\newtheorem{prop}{Proposition}[section]
\newtheorem{theorem}[prop]{Theorem}
\newtheorem{lemma}[prop]{Lemma}

\newtheorem{cor}[prop]{Corollary}
\newtheorem{conj}[prop]{Conjecture}
\theoremstyle{definition}
\newtheorem{definition}[prop]{Definition}
\newtheorem{ex}[prop]{Example}
\newtheorem{remark}[prop]{Remark}
\DeclareMathAlphabet{\mathpzc}{OT1}{pzc}{m}{it}
\usepackage{xcoffins}
\NewCoffin\tablecoffin

\title[The Exchange Graphs]{The Exchange Graphs of Weakly Separated Collections}
\author[Meena Jagadeesan]{Meena Jagadeesan}
\address{Department of Mathematics, Harvard University, Cambridge, MA 02138}
\email{\href{mailto:mjagadeesan@college.harvard.edu}{{\tt mjagadeesan@college.harvard.edu}}}
\begin{document}
\begin{abstract}
Weakly separated collections arise in the cluster algebra derived from the Pl\"ucker coordinates on the nonnegative Grassmannian. Oh, Postnikov, and Speyer studied weakly separated collections over a general Grassmann necklace $\mathcal{I}$ and proved the connectivity of every exchange graph. Oh and Speyer later introduced a generalization of exchange graphs that we call $\mathcal{C}$-constant graphs. They characterized these graphs in the smallest two cases. We prove an isomorphism between exchange graphs and a certain class of $\mathcal{C}$-constant graphs. We use this to extend Oh and Speyer's characterization of these graphs to the smallest four cases, and we present a conjecture on a bound on the maximal order of these graphs. In addition, we fully characterize certain classes of these graphs in the special cases of cycles and trees.
\end{abstract}
\maketitle
\section{Introduction}
The notion of weak separation was introduced by Leclerc and Zelevinsky \cite{Leclerc} in 1998 in the context of determining combinatorial criterion for the quasicommutativity of two quantum flag minors. The definition of weak separation requires the following definition of a cyclic ordering:
\begin{definition}
For an integer $n \ge 0$, the integers $a_1, a_2, \ldots, a_n$ are said to be \textbf{cyclically ordered} if there exists $1 \le j \le n$ such that $a_j < a_{j+1} < \ldots < a_n < a_1 < \ldots < a_{j-1}.$
\end{definition}
Leclerc and Zelevinsky defined weak separation as follows: 
\begin{definition}
Let $k$ and $n$ be nonnegative integers such that $k < n$. Two $k$-element subsets $I, J \subset [n]=\{1,2,\ldots,n\}$ are called \textbf{weakly separated} if there do not exist $a, b, a', b'$, cyclically ordered, with $a, a' \in I \setminus J$ and $b, b' \in J \setminus I$, where $\setminus$ denotes set-theoretic difference.
\end{definition}
\begin{remark}
This is equivalent to the condition that $I \setminus J = \left\{a_1\ldots,a_r\right\}$ and $J \setminus I = \left\{b_1, b_2,\ldots,b_r \right\}$ are separated by some chord in the circle, i.e. that $a_1, a_2, \ldots, a_r,$ $b_1, b_2, \ldots b_r$ are cyclically ordered.
\end{remark}

Leclerc and Zelevinsky \cite{Leclerc} used the notion of weak separation to define certain collections of $k$-elements subsets of $[n]$ for a fixed $k$ and $n$. Let ${[n] \choose k}$ denote the collection of all $k$-element subsets of $[n]$. For nonnegative integers $k < n$, a subset $\mathcal{W}$ of the collection ${[n] \choose k}$ is called \textbf{weakly separated} if every two $k$-elements subsets in $\mathcal{W}$ are weakly separated. The subset $\mathcal{W}$ is called a \textbf{maximal weakly separated collection} if $\mathcal{W}$ is not contained in any larger weakly separated collection of ${[n] \choose k}$.

Leclerc and Zelevinsky conjectured that any two maximal weakly separated collections of ${[n] \choose k}$ have the same cardinality and are linked by a sequence of cardinality-preserving operations called mutations. A mutation is defined as follows:
\begin{definition}
Consider a set $V \in {[n] \choose k-2}$ and let $a,b,c,d$ be cyclically ordered elements of $[n]\setminus V.$ Suppose that a maximal weakly separated collection $\mathcal{V}_1 \subset {[n] \choose k}$ contains $V \cup \{a,b\}, V \cup \{b,c\}, V \cup \{c,d\}, V \cup \{d,a\}$ and $V \cup \{a,c\}$. Then the collection $\mathcal{V}_2 = (\mathcal{V}_1 \setminus (V \cup \{a,c\})) \cup (V \cup \{b,d\})$ is said to be linked to $\mathcal{V}_1$ by a \textbf{mutation}.
\end{definition}

In 2005, Scott \cite{Scott} proved that any weakly separated collection $\mathcal{W} \subset {[n] \choose k}$ satisfies $|\mathcal{W}| \leq k(n-k) + 1$. This motivated him to make the following refined conjecture regarding the cardinality of maximal weakly separated collections: If $\mathcal{W} \subset {[n] \choose k}$ is a maximal weakly separated collection, then the following is true:
\begin{equation}
\label{scottconj}
|\mathcal{W}| = k(n-k) + 1.
\end{equation}
In 2011, Oh, Postnikov, and Speyer \cite{OhSpPost} proved a generalization of $(\ref{scottconj})$ by linking weak separation to certain graphs called plabic graphs through applying a previous work of Postnikov \cite{Post}.

In 2006, Postnikov \cite{Post} studied total positivity on the Grassmannian. He defined the nonnegative Grassmannian $Gr^{\ge 0} (k,n)$ as the part of the real Grassmannian $Gr(k,n)$ in which all Pl\"ucker coordinates are nonnegative. He described a stratification of $Gr^{\ge 0} (k,n)$ into positroid strata and combinatorially constructed their parametrization using plabic graphs. Namely, he proved that positroid cells can be parameterized through a certain type of planar bicolored graphs (plabic graphs). These type of graphs play an important role in the study of various mathematical objects, for example, appearing also in Kodama and Williams's \cite{Kodama} study on KP-solitons.

Postnikov, Oh, and Speyer \cite{OhSpPost} used plabic graphs to prove Leclerc and Zelevinsky's conjecture that any two maximal weakly separated collections in ${[n] \choose k}$ are linked by a sequence of mutations, from which $(\ref{scottconj})$ followed. The main ingredient of their proof was their bijection between weakly separated collections and reduced plabic graphs. In fact, using this bijection, they were able to prove a generalization of $(\ref{scottconj})$ based on extending the notion of a weakly separated collection to a general positroid. Their results had interesting consequences for cluster algebras. In 2003, Scott \cite{scott2} had proved that the coordinate ring of $Gr^{\ge 0} (k,n)$, in its Pl\"ucker embedding, is a cluster algebra where the Pl\"ucker coordinates are the cluster variables. Postnikov, Oh, and Speyer \cite{OhSpPost} proved that the aforementioned clusters are in bijection with the maximal weakly collections of ${[n] \choose k}$.

Our global aim is to continue the study of the combinatorial properties of the cluster algebra structure on $Gr^{\ge 0} (k,n)$. We specifically study weakly separated collections over a general positroid. In Section 2, we review important existing definitions and known results. In Section 3, we introduce some of our new definitions, state our main results, and outline the rest the paper.

\section{Existing Definitions and Known Results}
We review the relevant existing definitions, notation, and known results involving weakly separated collections. In Section 2.1, we review the definitions from \cite{OhSpPost} related to maximal weakly separated collections over a general positroid. In Section 2.2, we review the technology of plabic graphs from \cite{OhSpPost} and \cite{Post}. In Section 2.3, we review exchange graphs from \cite{OhSpPost} and the subgraphs from \cite{OhSp} that we call $\mathcal{C}$-constant graphs.
\subsection{Weakly Separated Collections over a General Positroid}
We recall the definitions and results from \cite{OhSpPost} relating to Grassmann necklaces, positroids, decorated permutations, and maximal weakly separated collections.

We define a connected Grassmann necklace.
\begin{definition}
A connected \textbf{Grassmann necklace} is a sequence $\mathcal{I} = (I_1,\ldots,I_n, I_{n+1} = I_1)$ of $k$-element subsets of $[n]$ such that, for $1 \le i \le n$, the set $I_{i+1}$ contains $\left\{i+1 \right\}$ and  $I_{i} \setminus \left\{i\right\}$ (where we take the indices modulo $n$).
\end{definition}
\begin{remark}
All Grassmann necklaces that we will work with will be connected, though we will omit the word connected. For the rest of the paper, we assume that $n \ge 3$ and $k \ge 2$.
\end{remark}
We use the following notion of a linear order $<_{i}$ on $[n]$.
\begin{definition}
Consider positive integers $i, r< n$. For $a_1, a_2, \ldots a_r \in [n]$, we say that $a_1 <_i a_2 <_i \ldots <_i a_r$ if $(a_1, a_2, \ldots a_r)$ is a subsequence of $(i, i+1, \ldots n, 1, 2 \ldots i-1).$ For $j, k \le r$, we say that $a_j \le_i a_k$ if and only if $a_j <_i a_k$ or $a_j = a_k$.
\end{definition} 
Let $V$ and $W $ be $k$-element subsets of $[n]$ such that $V= \left\{v_1,\ldots,v_k \right\}$ and $W = \left\{ w_1,\ldots,w_k \right\}$ where $v_1 <_{i} v_2 <_{i} \ldots v_k$ and $w_1 <_{i} w_2 \ldots <_{i} w_k$. Then we define the partial order
$$\displaystyle V \leq_{i} W \text{ if and only if } v_1 \leq_{i} w_1,\ldots, v_k \leq_{i} w_k.$$

We now define the positroid associated to each Grassmann necklace.
\begin{definition}
Given a Grassmann necklace $\mathcal{I} = (I_1,\ldots,I_n)$, we define the \textbf{positroid} $\mathcal{M}_{\mathcal{I}}$ to be
$$\displaystyle \mathcal{M}_{\mathcal{I}} = \left\{ J \in {[n] \choose k} \mid I_i \leq_i J \text{ for all } 1 \le i \le n\right\}.$$
\end{definition}

This allows us to extend the notion of a maximal weakly separated collection to a general Grassmann necklace.
\begin{definition}
For a Grassmann necklace $\mathcal{I}$, a weakly separated collection $\mathcal{W} \subset {[n] \choose k}$ is said be \textbf{over} $\mathcal{I}$ if $\mathcal{W}$ is a subset of $\mathcal{M}_{\mathcal{I}}$. The collection $\mathcal{W}$ is said to be \textbf{maximal} over $\mathcal{I}$ if $\mathcal{W}$ is not contained in any larger weakly separated collection over $\mathcal{I}$.
\end{definition}

Oh, Postnikov, and Speyer \cite{OhSpPost} proved a generalization of $(\ref{scottconj})$ (Scott's conjecture). Namely, they expressed the cardinality of every maximal weakly separated collection over a Grassmann necklace $\mathcal{I}$ as a function of the decorated permutation associated to $\mathcal{I}$. We review the definitions relating to decorated permutations in the case of connected Grassmann necklaces.
\begin{definition}
A \textbf{connected decorated permutation} is a permutation $\pi$ over $[n]$ such that there do not exist two circular intervals $[i,j)$ and $[j,i)$ such that $\pi([i,j)) = [i, j)$ and $\pi([j,i)) = [j, i)$.
\end{definition}
There is a simple relation between connected Grassmann necklaces and connected decorated permutations.
\begin{prop}[Oh, Postnikov, Speyer]
Connected Grassmann necklaces $\mathcal{I} \subset {[n] \choose k}$ are in bijection with connected decorated permutations over $[n]$.
\end{prop}
\begin{proof}
To go from a Grassmann necklace $\mathcal{I}$ to a decorated permutation $\pi$, we set
\[\pi(i) = j \text{ for each }I_{i+1} = (I_i \setminus \left\{i \right\}) \cup \left\{j \right\}.\]
To go from a decorated permutation $\pi$ to a Grassmann necklace $\mathcal{I}$, we set
\[ I_i =  \left\{j \in [n] \mid j <_i \pi^{-1}(j)  \right\}.\]
\end{proof}
Oh, Postnikov, and Speyer \cite{OhSpPost} defined the function $A: S_n \rightarrow  \mathbb{Z}$ (where $S_n$ is the set of permutations of $[n]$) as follows:
\begin{definition}
For $i, j \in [n]$, the set $\left\{i, j \right\}$ forms an \textbf{alignment} in $\pi$ if $i, \pi(i), \pi(j), j$ are cyclically ordered (and all distinct). Let $A(\pi)$ be the number of alignments in $\pi$. 
\end{definition}
Oh, Postnikov, and Speyer \cite{OhSpPost} used the notion of a decorated permutation and the function $A(\pi)$ to formulate their result involving cardinality. We recall the special case that corresponds to connected Grassmann necklaces:
\begin{theorem}[Oh, Postnikov, Speyer]
\label{card}
Let $\mathcal{W}$ be a maximal weakly separated collection over a connected Grassmann necklace $\mathcal{I} \subset {[n] \choose k}$. Suppose that $\mathcal{I}$ has associated decorated permutation $\pi$. Then the following is true:
\[|\mathcal{W}| =  k(n-k) + 1 - A(\pi).\]
\end{theorem}
\subsection{Plabic Graphs}
We review the technology of plabic graphs discussed in \cite{Post} and \cite{OhSpPost}.

We define a plabic graph.
\begin{definition}
A \textbf{plabic graph} (planar bicolored graph) is a planar undirected graph $\mathpzc{G}$ drawn inside a disk with vertices colored in black or white colors. The vertices on the boundary of the disk are called the \textbf{boundary vertices}. Suppose that there are $n$ boundary vertices. Then the boundary vertices are labeled in clockwise order by $[n]$.
\end{definition}
We define the strands in a plabic graph.
\begin{definition}
A \textbf{strand} in a plabic graph $\mathpzc{G}$ is a directed path that satisfies the rules of the road: At every black vertex, the strand turns right, and at every white vertex, the strand turns left.
\end{definition}
We define the criteria for a plabic graph to be reduced. 
\begin{definition}
A plabic graph $\mathpzc{G}$ is called \textbf{reduced} if the following holds:
\begin{itemize}
  \item A strand cannot form a closed loop in the interior of $\mathpzc{G}$.
  \item Any strand that passes through itself must be a simple loop that starts and ends at some boundary vertex.
  \item For any two strands that have two vertices $u$ and $v$ in common, one strand must be directed from $u$ to $v$, and the other strand must be directed from $v$ to $u$.
\end{itemize}
\end{definition}
Let $\mathpzc{G}$ be a reduced plabic graph. Then any strand in $\mathpzc{G}$ connects two boundary vertices. We associate a \textbf{strand permutation} $\pi_{\mathpzc{G}}$ with $\mathpzc{G}$ defined so that $\pi_{\mathpzc{G}}(j)=i$ if the strand that starts at a boundary vertex $j$ ends at a boundary vertex $i$. We label the strand that ends at boundary vertex $i$ by $i$.

There are three types of moves on a plabic graph:\\

(M1)Pick a square with vertices alternating in colors as in Figure~\ref{move1}. Then we can switch the colors of these $4$ vertices.\\
\begin{figure}[h!]
\centering
\includegraphics[height=0.6in]{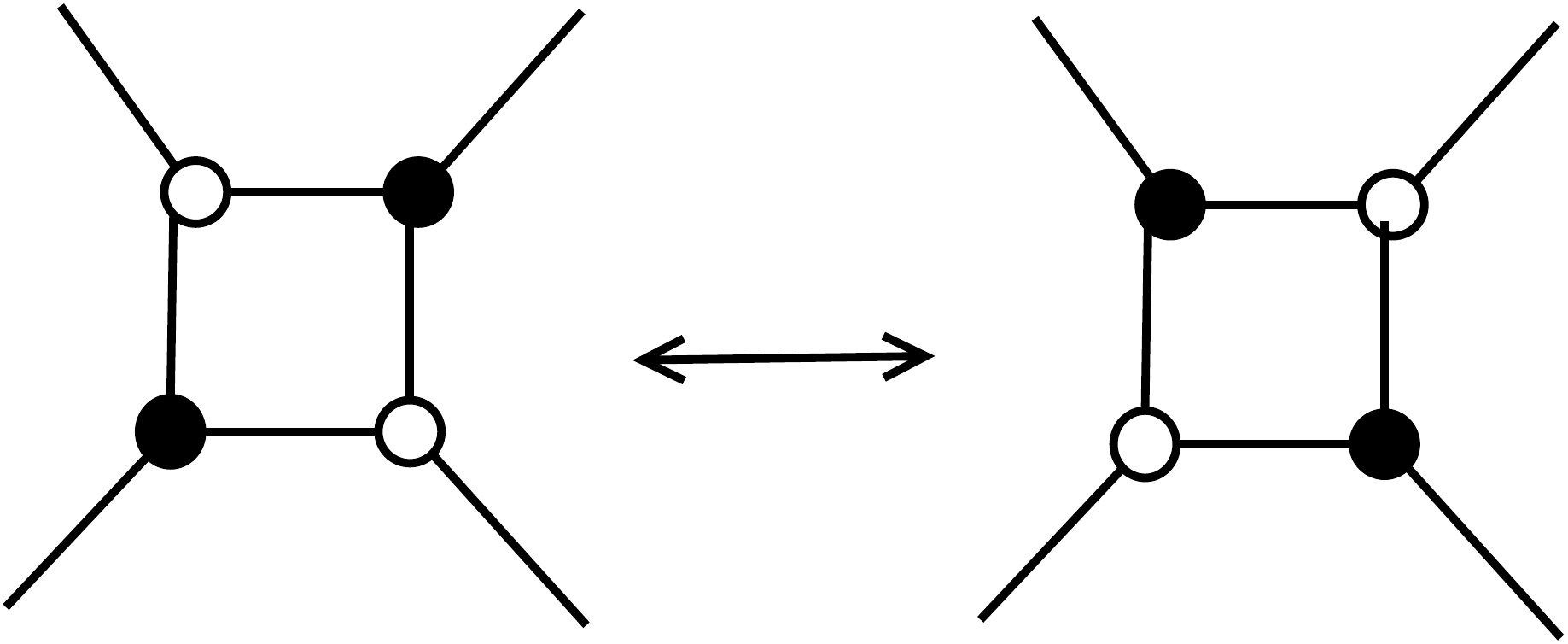}
\caption{(M1) Square move}
\label{move1}
\end{figure}

(M2)Two adjoint vertices of the same color can be contracted into one vertex as in Figure~\ref{move2}.\\
\begin{figure}[h!]
\centering
\includegraphics[height=0.4in]{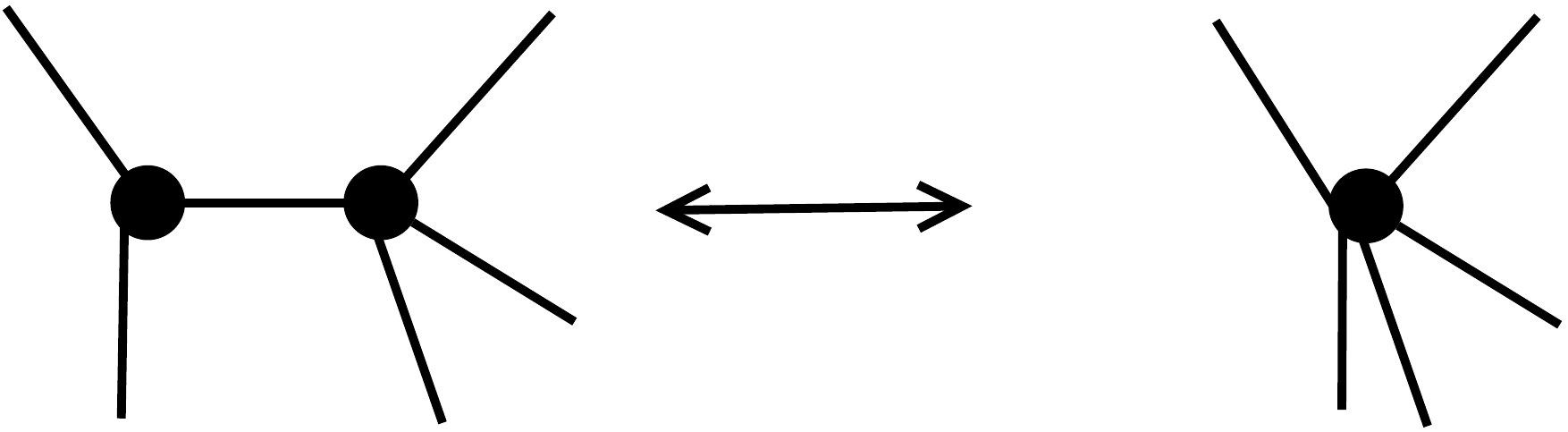}
\caption{(M2)}
\label{move2}
\end{figure}

(M3)We can insert or remove a vertex inside an edge as in Figure~\ref{move3}.
\begin{figure}[h!]
\centering
\includegraphics[height=0.08in]{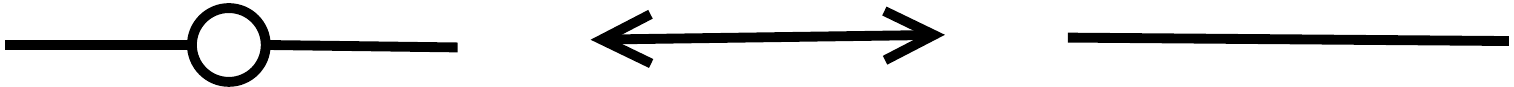}
\caption{(M3)}
\label{move3}
\end{figure}

Postnikov \cite{Post} proved the following relation between strand permutations and the moves (M1), (M2), and (M3). 
\begin{theorem}[Postnikov]
Let $\mathpzc{G}$ and $\mathpzc{G}'$ be two reduced plabic graphs with the same number of boundary vertices. Then $\pi_{\mathpzc{G}}=\pi_{\mathpzc{G}'}$ if and only if $\mathpzc{G}'$ can be obtained from $\mathpzc{G}$ by a sequence of moves (M1), (M2), and (M3).
\end{theorem}

We now describe the relation between reduced plabic graphs and weakly separated collections. We label the faces of a reduced plabic graph as follows: Let $\mathpzc{G}$ be a reduced plabic graph. For every $i \in [n]$, place $i$ inside every face $\mathcal{F}$ that appears to the left of the strand labeled $i$. Then the label of $\mathcal{F}$ is the set of all integers placed inside $\mathcal{F}$. Let $\mathcal{F}(\mathpzc{G})$ be the collection of sets of labels that occur on each face of the graph $\mathpzc{G}$. Postnikov \cite{Post} showed that all the sets in $\mathcal{F}(\mathpzc{G})$ have the same number of elements. We denote this number by $k$. Oh, Postnikov, and Speyer \cite{OhSpPost} proved the following:
\begin{theorem}[Oh, Postnikov, Speyer]
For a decorated permutation $\pi$ corresponding to a Grassmann necklace $\mathcal{I}$, a collection $\mathcal{V}$ is a maximal weakly separated collection over $\mathcal{I}$ if and only if it has the form $\mathcal{V}=\mathcal{F}(\mathpzc{G})$ for a reduced plabic graph
$\mathpzc{G}$ with strand permutation $\pi_{\mathpzc{G}} = \pi$.
\end{theorem}
We also consider a dual of plabic graphs called plabic tilings described in detail in \cite{OhSp}. The faces and vertices are flipped. In plabic tilings, the faces are colored black and white, and the vertices contain the sets of integer labels.
\subsection{Exchange Graphs and $\mathcal{C}$-Constant Graphs}
We review the definitions and known results regarding exchange graphs and $\mathcal{C}$-constant graphs.

For a given Grassmann necklace $\mathcal{I}$, the exchange graph is defined by the maximal weakly separated collections over $\mathcal{I}$ under the mutation operation.
\begin{definition}
Let $\mathcal{I}$ be a Grassmann necklace and $\mathscr{V} = \left\{\mathcal{V}_1, {V}_2,\ldots,{V}_m \right\}$ be the set of maximal weakly separated collections over $\mathcal{I}$. The \textbf{exchange graph} $\mathpzc{G}^{\mathcal{I}}$ is defined as follows:
\begin{enumerate}
\item{The vertices of $\mathpzc{G}^{\mathcal{I}}$ are $\mathscr{V}$.}
\item{The vertices $\mathcal{V}_1$ and $\mathcal{V}_2$ are connected by an edge if and only if $\mathcal{V}_1$ can be mutated into $\mathcal{V}_2$ in one mutation.}
\end{enumerate}
\end{definition}
In 2011, Oh, Postnikov, and Speyer \cite{OhSpPost} proved the following result:
\begin{theorem}[Oh, Postnikov, Speyer]
Any exchange graph is connected.
\end{theorem}

Given an exchange graph $\mathpzc{G}^{\mathcal{I}}$ and a weakly separated collection $\mathcal{C}$ over $\mathpzc{I}$, we call the $\mathcal{C}$-constant graph  the subgraph defined by maximal weakly separated collections over $\mathcal{I}$ that contain $\mathcal{C}$ with edges defined by mutations of subsets not in $\mathcal{C}$. 
\begin{definition}
Let $\mathcal{I}$ be a Grassmann necklace and $\mathscr{V}_{\mathcal{C}} = \left\{\mathcal{V}_1, \mathcal{V}_2,\ldots,\mathcal{V}_m \right\}$ be the set of maximal weakly separated collections over $\mathcal{I}$ that contain $\mathcal{C}$. We call the $\mathbf{\mathcal{C}}$\textbf{-constant graph} $\mathpzc{G}^{\mathcal{I}} (\mathcal{C})$ the vertex-induced subgraph of $\mathpzc{G}^{\mathcal{I}}$ generated by $\mathscr{V}_C$. The co-dimension of $\mathpzc{G}^{\mathcal{I}}(\mathcal{C})$ is defined to be $|\mathcal{W}| - |C|$ for $\mathcal{W} \in \mathpzc{G}^{\mathcal{I}}.$
\end{definition}
\begin{remark}
Notice that $\mathcal{C}$-constant graphs are a generalization of exchange graphs. In fact, every exchange graph is isomorphic to a $\mathcal{C}$-constant graph: for any Grassmann necklace $\mathcal{I}$, we see that $\mathpzc{G}^{\mathcal{I}}(\mathcal{I})$ is isomorphic to $\mathpzc{G}^{\mathcal{I}}.$
\end{remark}

In 2014, Oh and Speyer \cite{OhSp} proved the following results:
\begin{theorem}[Oh, Speyer]
Any $\mathcal{C}$-constant graph is connected.
\end{theorem}
\begin{theorem}[Oh, Speyer]
\label{char01}
The only $\mathcal{C}$-constant graph with co-dimension $0$ is a path with $1$ vertex. The only $\mathcal{C}$-constant graphs with co-dimension $1$ are a path with $1$ vertex and a path with $2$ vertices.
\end{theorem}
\section{Main Results}
In Section 3.1, we present some of our new notions that are critical to understanding our main results. In Section 3.2, we present our main results along with an outline for the rest of the paper.
\subsection{Some New Definitions}
In Section 3.1.1, we present some basic definitions. In Section 3.1.2, we define special classes of exchange graphs and $\mathcal{C}$-constant graphs. In Section 3.1.3, we define equivalence classes of decorated permutations.
\subsubsection{Basic Definitions}
We define the following definitions and notation involving adjacency, interior size, and mutations.

We define two $k$-elements subsets to be quasi-adjacent as follows:
\begin{definition}
We call sets $V_1, V_2 \in {[n] \choose k}$ \textbf{quasi-adjacent} if $|V_1 \cap V_2| = k-1$.
\end{definition}
\begin{remark}
Suppose that $V_1$ and $V_2$ are contained in a maximal weakly separated collection $\mathcal{V}$. This definition is equivalent to the condition that $V_1$ and $V_2$ are on the same face in the plabic tiling of $\mathcal{V}$.
\end{remark}

We now define a stronger notion of adjacency that is dependent on the choice of maximal weakly separated collection:
\begin{definition}
Given a maximal weakly separated collection $\mathcal{V} \in \mathpzc{G}^{\mathcal{I}}$, we call two $k$-element subsets $V_1, V_2 \in \mathcal{V}$ \textbf{adjacent} if and only if in the plabic tiling of $\mathcal{V}$, there exists an edge between $V_1$ and $V_2$ that border a black face and a white face.
\end{definition}
\begin{remark}
Notice that all adjacent subsets are also quasi-adjacent.
\end{remark}

We define the following notion with interior size which is closely related to cardinality:
\begin{definition}
We define the $\mathbf{interior}$ $\mathbf{size}$ of a maximal weakly separated collection $\mathcal{V}$ over a Grassmann Necklace $\mathcal{I}$ to be $|\mathcal{V}| - |\mathcal{I}|.$
\end{definition}
\begin{remark}
Notice that the interior size of $\mathcal{V}$ is the number of sets in the interior of the plabic tiling of $\mathcal{V}$.
\end{remark}
In fact, interior size is a property of the Grassmann necklace $\mathcal{I}$ (and thus the exchange graph $\mathpzc{G}^{\mathcal{I}}$). Let $\pi$ be the decorated permutation associated to $\mathcal{I}$. By Theorem~$\ref{card}$, we know that the interior size of any maximal weakly separated collection in $\mathpzc{G}^{\mathcal{I}}$ is $k(n-k) + 1 - A(\pi) - |\mathcal{I}|$. Thus, we let the \textbf{interior size} of both the Grassmann necklace $\mathcal{I}$ and the exchange graph $\mathpzc{G}^{\mathcal{I}}$ be this value. We denote this value by $i(\mathcal{I})$ or $i(\mathpzc{G}^{\mathcal{I}})$.

We use the following notation and definitions to discuss mutations:
\begin{definition}
Given a maximal weakly separated collection $\mathcal{V}$ over a Grassmann Necklace $\mathcal{I}$, we say that a $k$-element subset $V_1 \in \mathcal{V}$ is $\mathbf{mutatable}$ in $\mathcal{V}$ if in the plabic tiling of $\mathcal{V}$, the subset $V_1$ is surrounded by exactly 2 black faces and 2 white faces.
\end{definition}
This means that we can mutate $\mathcal{V}$ into a maximal weakly separated collection $\mathcal{W}$ over $\mathcal{I}$ that contains $\mathcal{V} \setminus \left\{V_1\right\}$. Suppose that $\left\{V_2 \right\} = \mathcal{W} \setminus \left\{V_1 \right\}$. We then say that $V_1$ can be mutated in $\mathcal{V}$ into $V_2$ in $\mathcal{W}$. If $V_1$ is not surrounded by 2 black faces and 2 white faces, then we say that $V_1$ is $\mathbf{immutatable}$ in $\mathcal{V}$.

\subsubsection{Special Classes of Exchange Graphs and $\mathcal{C}$-Constant Graphs}
We define the applicable, mutation-friendly, and very-mutation-friendly conditions.

Roughly speaking, the applicable condition for a $\mathcal{C}$-constant graph requires that the sets in the weakly separated collection $\mathcal{C}$ are connected to the boundary of the plabic tiling of $\mathcal{V}$ for any maximal weakly separated collection $\mathcal{V} \in \mathpzc{G}^{\mathcal{I}}(\mathcal{C})$. 
\begin{definition}
Consider a $\mathcal{C}$-constant graph $\mathpzc{G}^{\mathcal{I}}(\mathcal{C})$. We say that $\mathpzc{G}^{\mathcal{I}}(\mathcal{C})$ is $\mathbf{applicable}$ if for each set $V \in \mathcal{C}$, there exists an integer $m$ and a weakly separated collection $\mathcal{W} = (V = W_1, W_2,..W_m)$ over $\mathcal{I}$ satisfying the following properties:
\begin{enumerate}
\item{$W_i$ is quasi-adjacent to $W_{i+1}$ for $1 \le i \le m-1$,}
\item{$W_m \in \mathcal{I}$}
\item{and $\mathcal{W} \subset \mathcal{C}$.}
\end{enumerate}
\end{definition}
\begin{figure}
\caption{}
\label{appl}
\includegraphics[scale=0.7]{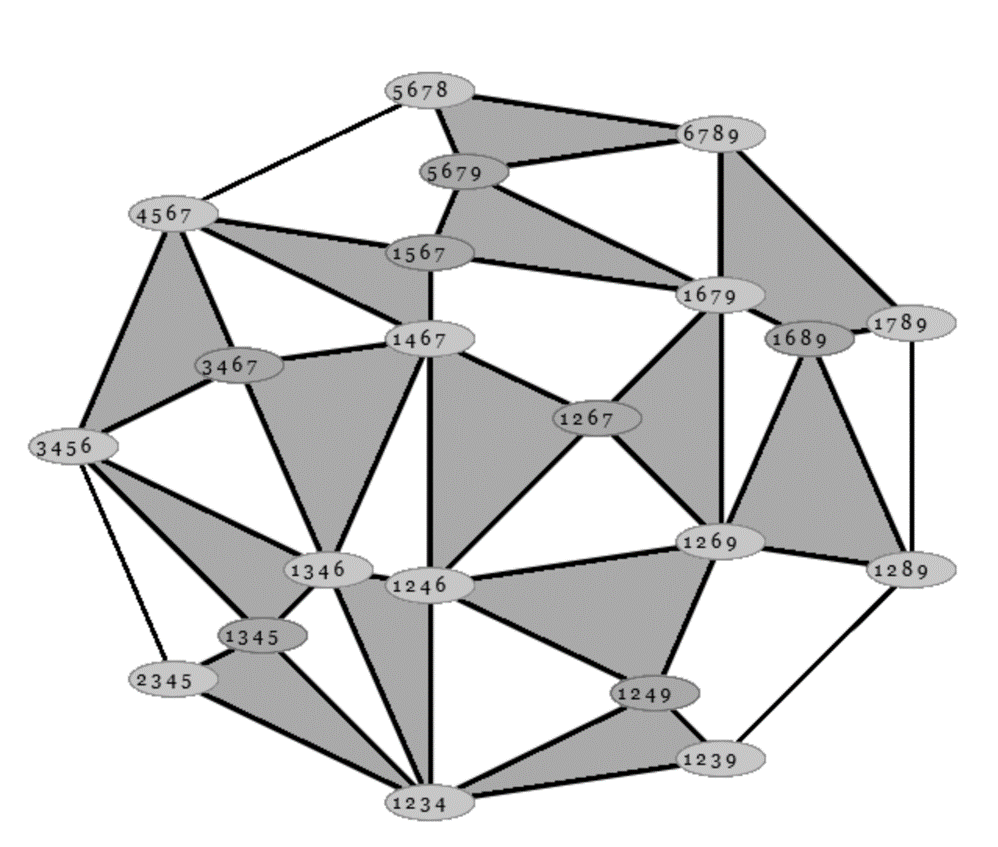}
\end{figure}
We use the following maximal weakly separated collection and Grassmann necklace as an example throughout the paper:
\begin{ex}
\label{mainexth}
Consider the Grassmann necklace
\[\mathcal{I} =
(\left\{1,2,3, 4\right\}, \left\{2, 3, 4, 5\right\},\left\{3,4,5, 6\right\}, \left\{4,5,6, 7\right\}, \left\{5, 6, 7, 8\right\}, \left\{6, 7, 8, 9 \right\}, \left\{1,7, 8, 9\right\}, \]
\[\left\{1,2,8, 9\right\}, \left\{1,2,3, 9\right\}).\]
Consider the maximal weakly separated collection $\mathcal{V}$ that contains $\mathcal{I}$ and the following subsets:
$\left\{5, 6, 7, 9 \right\}$, $\left\{1, 5, 6, 7\right\}$, $\left\{1, 6, 7, 9 \right\}$, $\left\{1, 6, 8, 9 \right\}$, $\left\{1, 2, 6, 9 \right\}$, $\left\{1, 2, 4, 9 \right\}$, $\left\{1,2, 4, 6 \right\}$, $\left\{1, 3, 4, 6 \right\}$, $\left\{1, 3, 4, 5 \right\}$, $\left\{3, 4, 6, 7 \right\}$, $\left\{1, 2, 6, 7 \right\}$, $\left\{1, 4, 6, 7 \right\}$.
Figure~$\ref{appl}$ shows the plabic tiling of $V$.
\end{ex}
\begin{ex}
\label{appandnot}
Consider $\mathcal{I}$ defined as in Example~$\ref{mainexth}$. Then we define the following weakly separated collections over $\mathcal{I}$:
\begin{align*}
\mathcal{C}_1 &= \left\{\left\{1, 2, 6, 7 \right\}\right\} \cup \mathcal{I} \\
\mathcal{C}_2 &= \left\{\left\{1, 6, 7, 9 \right\}, \left\{1, 2, 6, 7 \right\} \right\}\cup \mathcal{I}
\end{align*}
The $\mathcal{C}$-constant graph $\mathpzc{G}^{\mathcal{I}}(\mathcal{C}_1)$ is not applicable and the $\mathcal{C}$-constant graph $\mathpzc{G}^{\mathcal{I}}(\mathcal{C}_2)$ is applicable.
\end{ex}

Roughly speaking, the mutation-friendly condition for a Grassmann necklace $\mathcal{I}$ requires that every subset in $\mathcal{M}_{\mathcal{I}}$ that is present in at least one maximal weakly separated collection in $\mathpzc{G}^{\mathcal{I}}$ is mutatable in some maximal weakly separated collection in $\mathpzc{G}^{\mathcal{I}}$.
\begin{definition}
A Grassmann necklace $\mathcal{I}$ and its associated decorated permutation $\pi$ are $\mathbf{mutation-friendly}$ if the intersection of all of the maximal weakly separated collections in $\mathpzc{G}^{\mathcal{I}}$ is $\mathcal{I}$. We call an exchange graph $\mathpzc{G}^{\mathcal{I}}$ mutation-friendly if and only if the Grassmann necklace $\mathcal{I}$ is mutation-friendly.
\end{definition}
\begin{remark}
It follows from the definition that for a mutation-friendly exchange graph $\mathpzc{G}^{\mathcal{I}}$, we know that
$$\displaystyle |\mathpzc{G}^{\mathcal{I}}| > i(\mathpzc{G}^{\mathcal{I}}).$$
\end{remark}
We extend a similar notion to $\mathcal{C}$-constant graphs:
\begin{definition}
A $\mathcal{C}$-constant graph $\mathpzc{G}^{\mathcal{I}}(\mathcal{C})$ is said to be $\mathbf{mutation-friendly}$ if the intersection of all of the maximal weakly separated collections in $\mathpzc{G}^{\mathcal{I}}(\mathcal{C})$ is $\mathcal{C}$.
\end{definition}

We define a stronger notion of the mutation-friendly condition in the case of exchange graphs. This condition, the very-mutation-friendly condition, requires that a certain sequence of $\mathcal{C}$-constant graphs of the exchange graph satisfy both the applicable and mutation-friendly conditions.
\begin{definition}
\label{vmf}
A mutation-friendly Grassmann necklace $\mathcal{I}$ with interior size $i$ is \\
$\mathbf{very-mutation-friendly}$ if there exists a set of weakly separated collections $\mathscr{W} = (\mathcal{W}_1, \mathcal{W}_2,...\mathcal{W}_{i-1})$ satisfying the following properties:
\begin{enumerate}
\item{$\mathcal{W}_j$ is a weakly separated collection over $\mathcal{I}$ containing $\mathcal{I}$ for $1 \le j \le i-1$,}
\item{$\mathcal{W}_{j+1} \subset \mathcal{W}_{j}$ for $1 \le j \le i-2$,}
\item{$|\mathcal{W}_{j}| = i + |\mathcal{I}|-j$ for $1 \le j \le i-1$,}
\item{and $\mathpzc{G}^{\mathcal{I}}(\mathcal{W}_j)$ is applicable and mutation-friendly for $1 \le j \le i-1.$}
\end{enumerate}
We call $\mathpzc{G}^{\mathcal{I}}$ very-mutation-friendly if and only if $\mathcal{I}$ is very-mutation-friendly.
\end{definition}

\subsubsection{Equivalence Classes}
In order to simplify our discussion of Grassmann necklaces, we construct an equivalence class of decorated permutations that yield the same exchange graph. We show that exchange graphs are invariant under the following four operations on associated decorated permutations.

We first consider the inverse operation.
\begin{prop}
\label{inv}
Let $\mathcal{I}$ be the Grassmann necklace with decorated permutation $\pi$ and $\mathcal{J}$ be the Grassmann necklace with decorated permutation $\pi^{-1}.$ Then the following is true:
\[\mathpzc{G}^{\mathcal{I}} \cong \mathpzc{G}^{\mathcal{J}}. \]
\end{prop}
\begin{proof}
We must prove that there is a bijective mapping $\phi$ from the maximal weakly separated collections over $\mathcal{I}$ to the maximal weakly separated collections over $\mathcal{J}$ that preserves the mutation operation. Let $\phi$ take $\mathcal{V}_1 \in \mathpzc{G}^{\mathcal{I}}$ to $\mathcal{V}_2 \in \mathpzc{G}^{\mathcal{J}}$ where $\mathcal{V}_2$ is obtained by flipping the colors of the vertices in the plabic graph of $\mathcal{V}_1$.
\end{proof}
\begin{figure}
\caption{}
\label{first}
\includegraphics[scale=0.5]{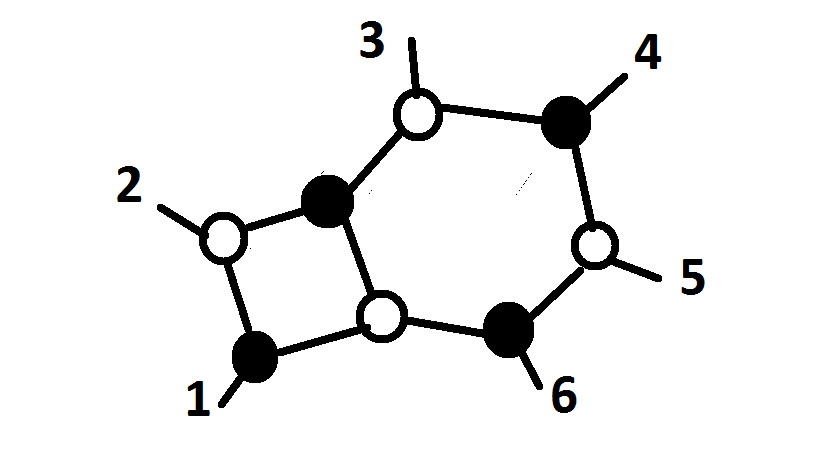}
\end{figure}
\begin{figure}
\caption{}
\label{second}
\includegraphics[scale=0.5]{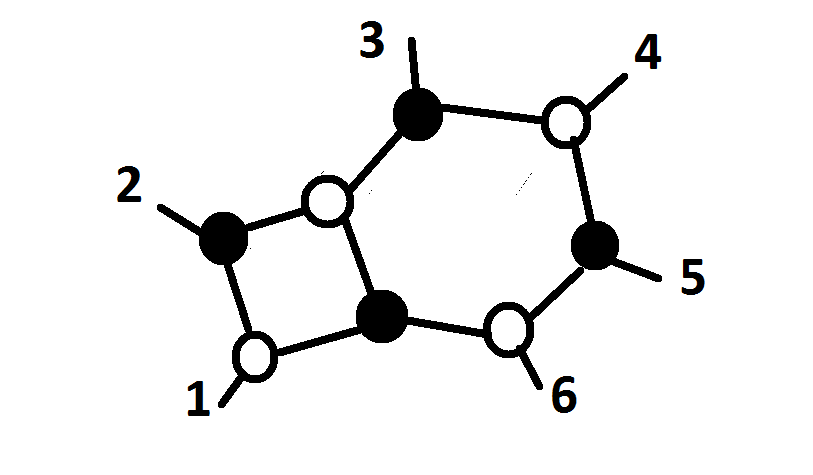}
\end{figure}
\begin{ex}
Let $\pi$ be the permutation $365124$. Then $\pi^{-1} = 451632$. Let $\mathcal{V}_1$ be the maximal weakly separated collection corresponding to the plabic graph in Figure~$\ref{first}$ and $\mathcal{V}_2$ be the maximal weakly separated collection corresponding to the plabic graph in Figure~$\ref{second}$. We see that $\phi$ takes $\mathcal{V}_1$ to $\mathcal{V}_2$ as expected.
\end{ex}

We now consider label-reflection operations.
\begin{definition}
For a positive integer $i$, we define the \textbf{label-reflection operation} $LR^{2i}$ as follows. For an integer $n \ge i$ and a permutation $\pi$ of $[n]$, we let  $LR^{2i}[\pi]$ be the permutation on $[n]$ defined so that $LR^{2i}[\pi](j) = 2i - \pi^{-1}(2i-j)$ for $1 \le j \le n$ (where all numbers are considered modulo $n$).
\end{definition}
\begin{prop}
\label{lr}
Given integers $1 \le i \le n$, let $\mathcal{I}$ be the Grassmann necklace with decorated permutation $\pi$ and $\mathcal{J}$ be the Grassmann necklace with decorated permutation $LR^{2i}[\pi].$ Then the following is true:
\[\mathpzc{G}^{\mathcal{I}} \cong \mathpzc{G}^{\mathcal{J}}. \]
\end{prop}
The proof is analogous to the proof of Proposition~$\ref{inv}$, except that $\phi$ is defined by reflecting the plabic graph of $\mathcal{V}_1$ about the vertex labeled $i$.
\begin{figure}
\caption{}
\label{third}
\includegraphics[scale=0.5]{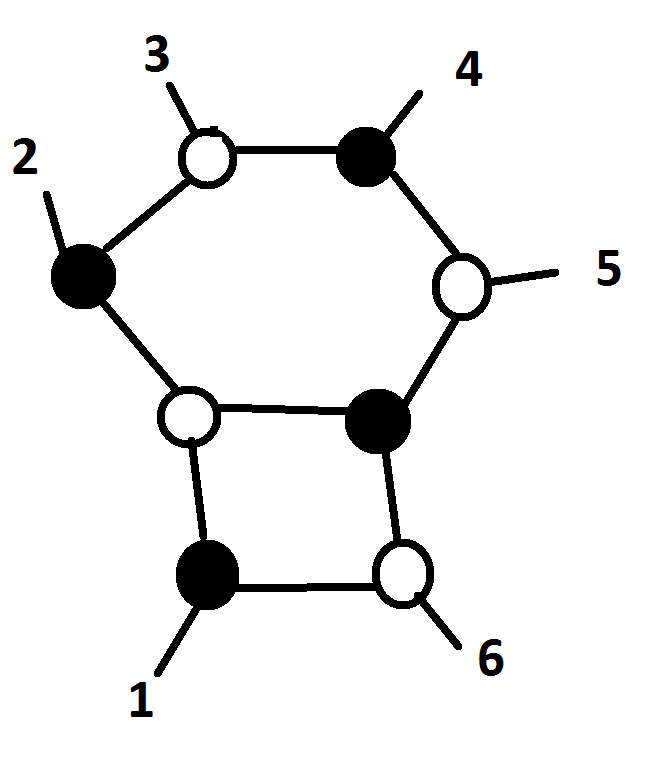}
\end{figure}
\begin{ex}
Let $\pi$ be the permutation $365124$, and let $i = 4$. Then $LR^{8}[\pi] = 465213$. Let $\mathcal{V}_1$ be the maximal weakly separated collection corresponding to the plabic graph in Figure~$\ref{first}$ and $\mathcal{V}_2$ be the maximal weakly separated collection corresponding to the plabic graph in Figure~$\ref{third}$. We see that $\phi$ takes $\mathcal{V}_1$ to $\mathcal{V}_2$ as expected.
\end{ex}

We now consider between-label-reflection operations.
\begin{definition}
For a positive integer $i$, we define the \textbf{between-label-reflection operation} $BLR^{2i-1}$ as follows. For an even integer $n \ge i$ and a permutation $\pi$ of $[n]$, we let  $BLR^{2i-1}[\pi]$ be the permutation on $[n]$ defined so that $BLR^{2i-1}[\pi](j) = 2i-1 - \pi^{-1}(2i-1-j)$ for $1 \le j \le n$ (where all numbers are considered modulo $n$).
\end{definition}
\begin{prop}
\label{blr}
Given integers $1 \le i \le n$ such that $n$ is even, let $\mathcal{I}$ be the Grassmann necklace with decorated permutation $\pi$ and $\mathcal{J}$ be the Grassmann necklace with decorated permutation $BLR^{2i-1}[\pi].$ Then the following is true:
\[\mathpzc{G}^{\mathcal{I}} \cong \mathpzc{G}^{\mathcal{J}}. \]
\end{prop}
The proof is analogous to the proof of Proposition~$\ref{inv}$, except that $\phi$ is defined by reflecting the plabic graph of $\mathcal{V}_1$ about the perpendicular bisector of the edge between the vertex labeled $i-1$ begins and the vertex labeled $i$.
\begin{figure}
\caption{}
\label{fourth}
\includegraphics[scale=0.5]{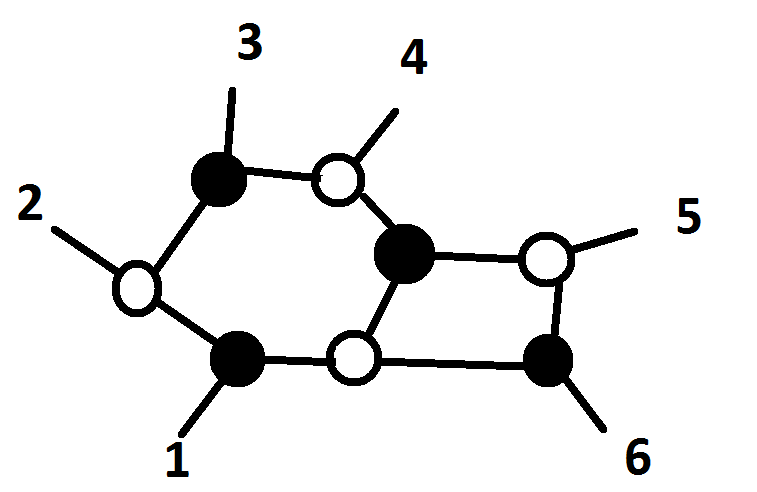}
\end{figure}
\begin{ex}
Let $\pi$ be the permutation $365124$, and let $i = 4$. Then $BLR^{7}[\pi] = 541623$. Let $\mathcal{V}_1$ be the maximal weakly separated collection corresponding to the plabic graph in Figure~$\ref{first}$ and $\mathcal{V}_2$ be the maximal weakly separated collection corresponding to the plabic graph in Figure~$\ref{fourth}$. We see that $\phi$ takes $\mathcal{V}_1$ to $\mathcal{V}_2$ as expected.
\end{ex}

We now consider rotation operations.
\begin{definition}
For a positive integer $i$, we define the \textbf{rotation operation} $R^{i}$ as follows. For an even integer $n \ge i$ and a permutation $\pi$ of $[n]$, we let $R^{i}[\pi]$ be the permutation on $[n]$ defined so that $R^{i}[\pi](j) =  \pi(j-i)+i$ for $1 \le j \le n$ (where all numbers are considered modulo $n$).
\end{definition}
\begin{prop}
\label{rot}
Given integers $1 \le i \le n$, let $\mathcal{I}$ be the Grassmann necklace with decorated permutation $\pi$ and $\mathcal{J}$ be the Grassmann necklace with decorated permutation $R^{i}[\pi].$ Then the following is true:
\[\mathpzc{G}^{\mathcal{I}} \cong \mathpzc{G}^{\mathcal{J}}. \]
\end{prop}
The proof is analogous to the proof of Proposition~$\ref{inv}$, except that $\phi$ is defined by rotating the vertex labels on the plabic graph of $\mathcal{V}_1$ so that the vertex $1$ is now labeled $i+1$.
\begin{figure}
\caption{}
\label{fifth}
\includegraphics[scale=0.5]{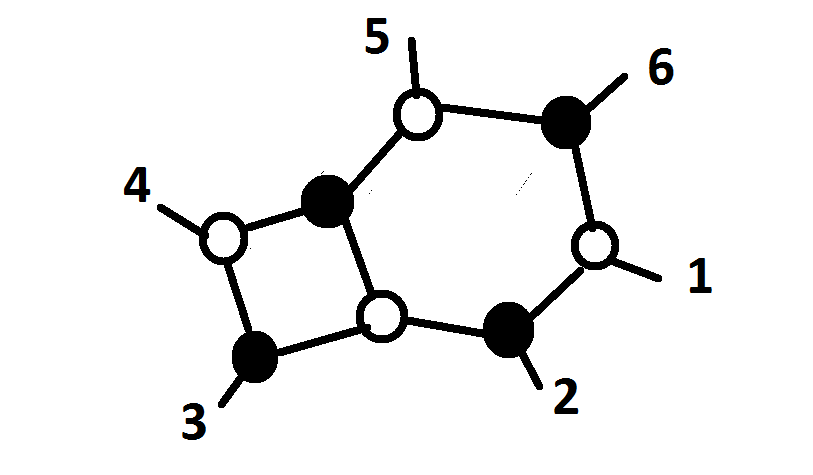}
\end{figure}
\begin{ex}
Let $\pi$ be the permutation $365124$, and let $i = 2$. Then $R^{2}[\pi] = 465213$. Let $\mathcal{V}_1$ be the maximal weakly separated collection corresponding to the plabic graph in Figure~$\ref{first}$ and $\mathcal{V}_2$ be the maximal weakly separated collection corresponding to the plabic graph in Figure~$\ref{fifth}$. We see that $\phi$ takes $\mathcal{V}_1$ to $\mathcal{V}_2$ as expected.
\end{ex}
This motivates the following definition of an equivalence class of decorated permutations:
\begin{definition}
For a given decorated permutation $\pi$, we define the $\mathbf{equivalence}$ $\mathbf{class}$ $\mathscr{C}$ to contain all permutations that can be obtained from one another by a sequence of inverse operations, label-reflection operations, between-label reflection operations, and rotation operations. We also let $\mathscr{C}$ denote the class of corresponding Grassmann necklaces.
\end{definition}
We prove that any two decorated permutations in the same equivalence class yield the same exchange graph.
\begin{lemma}
Let $\mathscr{C}$ be an equivalence class, and consider $\pi_1, \pi_2 \in \mathscr{C}$. Let $\mathcal{I}$ be the Grassmann necklace with decorated permutation $\pi_1$ and $\mathcal{J}$ be the Grassmann necklace with decorated permutation $\pi_2.$ Then the following is true:
\[\mathpzc{G}^{\mathcal{I}} \cong \mathpzc{G}^{\mathcal{J}}. \]
\end{lemma}
\begin{proof}
This follows from Proposition~$\ref{inv}$, Proposition~$\ref{lr}$, Proposition~$\ref{blr}$, and Proposition~$\ref{rot}$.
\end{proof}
\subsection{Statement of Main Results}
We continue the study of $\mathcal{C}$-constant graphs and exchange graphs for maximal weakly separated collections over a general positroid with connected Grassmann necklaces.

In Section 4, we prove our main result: an isomorphism between exchange graphs and applicable $\mathcal{C}$-constant graphs.
\begin{theorem}
\label{link}
For any $c > 0$, the set of possible applicable $\mathcal{C}$-constant graphs of co-dimension $\mathcal{C}$ is isomorphic to the set of the possible exchange graphs $\mathpzc{G}^{\mathcal{I}}$ with interior size $c$. An applicable $\mathcal{C}$-constant graph $\mathpzc{G}^{\mathcal{I}}(\mathcal{C})$ of co-dimension $c$ that is mutation-friendly is isomorphic to a mutation-friendly exchange graph with interior size $c$.
\end{theorem}
A consequence of this result is that all properties of the exchange graphs apply to applicable $\mathcal{C}$-constant graphs and vice versa. For $c < 4$, we show that we can eliminate the applicable condition:
\begin{cor}
\label{stronglink}
For any $0 < c < 4$, the set of possible $\mathcal{C}$ constant graphs of co-dimension $c$ is isomorphic to the set of the possible exchange graphs $\mathpzc{G}^{\mathcal{I}}$ with interior size $c$. A $\mathcal{C}$-constant graph $\mathpzc{G}^{\mathcal{I}}(\mathcal{C})$ of co-dimension $c$ that is mutation-friendly is isomorphic to a mutation-friendly exchange graph with interior size $c$.
\end{cor}

In Section 5, we characterize all exchange graphs with interior size $\le 4$. We also generalize Theorem~$\ref{char01}$ (Oh and Speyer's characterization result for $\mathcal{C}$-constant graphs) by characterizing all $\mathcal{C}$-constant graphs of co-dimension $\le 3$. These results involve the information in Table~$\ref{mainchare}$ and Table~$\ref{maincharc}$. These tables require the graphs $A, B, \ldots Y, Z_1, Z_2, \ldots Z_6$ which are defined in Table~$\ref{allfigures}$ and the direct product of graphs which is defined as follows:
\begin{definition}
Given graphs $\mathpzc{G}_1$ and $\mathpzc{G}_2$, we define the direct product $\mathpzc{G}_1 \square \mathpzc{G}_2$ as follows:
\begin{enumerate}
\item{The vertices of $\mathpzc{G}_1 \square \mathpzc{G}_2$ are $\left\{(v_1, v_2) \mid v_1 \in \mathpzc{G}_1, v_2 \in \mathpzc{G}_2 \right\}.$}
\item{There is an edge between $(v_1, v_2)$ and $(w_1, w_2)$ if and only if either}
\begin{itemize}
\item{$v_1 = w_1$ and $v_2$ is adjacent to $w_2$ in $\mathpzc{G}_2$,}
\item{or $v_1$ is adjacent to $w_1$ in $\mathpzc{G}_1$ and $v_2 = w_2.$}
\end{itemize}
\end{enumerate}
\end{definition} 
\begin{table}
\caption{}
\label{mainchare}
\begin{tabular}{||c | c| c||}
\hline
Interior Size & Exchange Graph Orders & Exchange Graphs \\ [0.5 ex]
\hline
0 & 1 & A\\
\hline
1 & 1, 2 & A, B\\
\hline
2 & 1, 2, 3, 4, 5 & A, B, C, D, B $\square$ B\\
\hline
3 & 1, 2, 3, 4, 5, 6, 7, 8, 10, 14 & A, B, C, D, E, F, G, H, I,\\
& & B $\square$ C, B $\square$ B $\square$ B, B $\square$ D\\
\hline
4 & 1, 2, 3, 4, 5, 6, 7, 8, 9, & A, B, C, D, E, F, G, H, I, J, K, L, M, N,\\
& 10, 11, 12, 13, 14, 15, 16, 17 & O, P, Q, R, S, T, U, V, W, X, Y, Z1, Z2,\\
& 19, 20, 25, 26, 28, 34, 42  & Z3, Z4, Z5, Z6, B $\square$ C, B $\square$ B $\square$ B, B $\square$ D, \\
& & B $\square$ E, B $\square$ G, B $\square$ F, B $\square$ H, B $\square$ I, \\
& & B $\square$ B $\square$ B $\square$ B, B $\square$ B $\square$ C, B $\square$ B $\square$ D \\
& & C $\square$ C, D $\square$ D \\
\hline
\end{tabular}
\end{table}
We have the following two results:
\begin{theorem}
\label{chare}
The information is Table~$\ref{mainchare}$ is true.
\end{theorem}
\begin{theorem}
\label{charc}
The information in Table~$\ref{maincharc}$ is true.
\end{theorem}
\begin{table}
\caption{}
\label{maincharc}
\begin{tabular}{||c | c | c||}
\hline
Co-dimension & $\mathcal{C}$-Constant Graph Orders & $\mathcal{C}$-Constant Graphs \\ [0.5 ex]
\hline
0 & 1 & A \\
\hline
1 & 1, 2 & A, B \\
\hline
2 & 1, 2, 3, 4, 5 & A, B, C, D, B $\square$ B \\
\hline
3 & 1, 2, 3, 4, 5, 6, 7, 8, 10, 14 & A, B, C, D, E, F, G, H, I, \\
& & B $\square$ C, B $\square$ B $\square$ B, B $\square$ D\\
\hline
\end{tabular}
\end{table}

In Section 6, we present the following conjecture on a bound for the maximal order of an exchange graph with a given interior size:
\begin{conj}
\label{cat}
For any $i \ge 0$, the maximum possible order of an exchange graph with interior size $i$ is the Catalan number $C_{i+1}$ (using the convention that $C_1 = 1$, $C_2 = 2$, and $C_3 = 5$).
\end{conj}
\begin{remark}
Theorem~$\ref{chare}$ proves Conjecture~$\ref{cat}$ for $i \le 4$. 
\end{remark}

In Section 7, we consider very-mutation-friendly exchange graphs in the special cases of cycles and trees. We fully characterize the very-mutation-friendly exchange graphs that are trees. We require the following equivalence classes:
\begin{definition}
We define $\mathscr{C}_i$  to be the equivalence class with the permutation $\pi_i$ which is defined as follows:
\begin{itemize}
\item{For $i =1$, we have $\pi_i = 312$.}
\item{For $i \ge 2$, we let $\pi_i$ be a permutation on $[2i]$ defined as follows:
\[\pi_c(a) = \begin{cases}
3 & \textrm{ if } a = 1 \\
2i+1-j & \textrm{ if } 1 < j \le i \\
1 & \textrm{ if } j =i+1 \\
2 & \textrm{ if } j=i+2 \\
2i+3-j & \textrm{ if } i+2 < j \le 2i.
\end{cases} \]}
\end{itemize}
\end{definition}
\begin{theorem}
\label{tree}
Any very-mutation-friendly exchange graph that is a tree must be a path. For an integer $i \ge 0$, a very-mutation-friendly exchange graph with interior size $i$ is a path if and only if it is a part of the equivalence class $\mathscr{C}_{i+1}$ (path with $i+1$ vertices). 
\end{theorem}
We also fully characterize the equivalence classes of very-mutation-friendly exchange graph with interior size $i$ that are single cycles.
\begin{theorem}
\label{cycle}
If a very-mutation-friendly exchange graph $\mathpzc{G}^{\mathcal{I}}$ is a single cycle, then it must have $1$, $2$, $4$, or $5$ vertices. If $\mathpzc{G}^{\mathcal{I}}$ is prime and very-mutation friendly, then $\mathcal{I}$ must be part of the equivalence class with $312$ (cycle with $1$ vertex), with $3412$ (cycle with $2$ vertices), or with $34512$ (cycle with $5$ vertices). If $\mathpzc{G}^{\mathcal{I}}$ is very-mutation-friendly and not prime, then $\mathcal{I}$ must be a part of the equivalence class with $351624$ (cycle with $4$ vertices).
\end{theorem}

\section{Proof of Theorem~$\ref{link}$ and Corollary~$\ref{stronglink}$}
Theorem~$\ref{link}$ is an isomorphism between applicable $\mathcal{C}$-constant graphs and exchange graphs. Corollary~$\ref{stronglink}$ is a special case of Theorem~$\ref{link}$ in the case of $c < 4$ where we show that the applicable condition is no longer necessary.

One direction of the proof of Theorem~$\ref{link}$ is easy. It is clear that an exchange graph $\mathpzc{G}^{\mathcal{I}}$ with interior size $c$ is isomorphic to the applicable $\mathcal{C}$-constant graph $\mathpzc{G}^{\mathcal{I}}(\mathcal{I})$ which has co-dimension $c$. This proves that the set of possible applicable $\mathcal{C}$-constant graphs of co-dimension $c$ contains an isomorphic copy of the set of possible exchange graphs with interior size $c$. In this section, we prove the other direction of Theorem~$\ref{link}$: namely, that every applicable $\mathcal{C}$-constant graph is isomorphic to an exchange graph with appropriate interior size. We also prove Corollary~$\ref{stronglink}$.

In Section 4.1, we introduce with some modifications the approaches of Oh, Postnikov, and Speyer \cite[Section 9]{OhSpPost} as well as Danilov, Karzanov, and Koshevoy \cite{DKK14} involving domains inside special weakly separated collections. In Section 4.2, we construct the machinery of interior-reduced plabic graphs and plabic tilings. In Section 4.3, we construct a decomposition set of a Grassmann necklace. In Section 4.4, we construct adjacency graphs and clusters. In Section 4.5, we prove a special case of Theorem~$\ref{link}$. In Section 4.6, we use this to prove Theorem~$\ref{link}$ in the general case. In Section 4.7, we show that all $\mathcal{C}$-constant graphs with co-dimension $<$ 4 are applicable, which when applied to Theorem~$\ref{link}$, proves Corollary~$\ref{stronglink}$.

\subsection{Domains inside and outside of cyclic patterns}
We introduce with some modifications the approaches in \cite{OhSpPost} and \cite{DKK14} regarding domains inside non self-intersecting closed curves. Fix $n$ unit vectors $\xi_1,\xi_2, \ldots, \xi_n$ in the upper half-plane, so that $\xi_1,\xi_2,\ldots,\xi_n$ go in this order clockwise around the origin $(0, 0)$ and are $\Z$-independent. Define:
$$\mathcal{Z}=\{\lambda_1\xi_1+\ldots+\lambda_n\xi_n \mid 0 \leq \lambda_i \leq 1, i=1,2,\ldots,n\}.$$
A subset $I \subset [n]$ is identified with the point $\sum_{i \in I} \xi_i$ in $\mathcal{Z}$.

Now, suppose we have a weakly separated collection $\mathcal{S} = (S_1, S_2,...,S_n, S_{n+1} = S_0)$ of subsets of equal size such that $|S_i \setminus S_{i+1}| = 1$ (where indices are considered modulo $n$). Then we call $\mathcal{S}$ a \textbf{sequence}. Note that $\mathcal{S}$ might have repeated subsets. We can associate with $\mathcal{S}$ a clockwise-oriented piecewise linear closed curve $\zeta_{\mathcal{S}}$ obtained by concatenating the line-segments connecting consecutive points $S_{i}$ and $S_{i+1}$ for $i=1,2\ldots,n$.

We consider sequences $\mathcal{S}$ with the following additional constraint on $\zeta_{\mathcal{S}}$:
\begin{definition}
\label{prongclosed}
We call a closed curve $\zeta$ $\textbf{prong-closed}$ if $\zeta$ is the concatenation of curves $\zeta^{F1}$, $\zeta^{P}$, and $\zeta^{F2}$ that satisfy the following conditions:
\begin{enumerate}
\item{The curve $\zeta^{F1} \cup \zeta^{F2}$ is a clockwise-oriented non self-intersecting closed curve.}
\item{The intersection $\zeta^{F1} \cap \zeta^{F2}$ is a single point $P$.}
\item{$\zeta^{P}$ satisfies one of the following:}
\begin{itemize}
\item{$\zeta^{P} = \emptyset$}
\item{or the following conditions hold:}
\begin{enumerate}
\item{$\zeta^{P}$ is a non self-intersecting open curve,}
\item{$\zeta^{P}$ has an endpoint at $P$,}
\item{$\zeta^{P}$ has an endpoint at a point $Q$ in the interior of $\zeta^{F1} \cup \zeta^{F2}$,}
\item{and $(\zeta^{F1} \cup \zeta^{F2}) \cap \zeta^{P} = P$.}
\end{enumerate}
\end{itemize}
\end{enumerate}
\end{definition}

We also impose the following constraints on the sequence $\mathcal{S}$. Sequences that satsify certain constraints involving weak separation are called generalized cyclic patterns \cite{DKK14}. We consider the following variant of a generalized cyclic pattern:
\begin{definition}
A \textbf{quasi-generalized cyclic pattern} is a sequence $\mathcal{S}=(S_1,S_2,\ldots,S_n, S_{n+1}=S_1)$ of subsets of $[m]$ where $m \ge n$ such that
\begin{enumerate}
 \item $\mathcal{S}$ is weakly separated,
 \item the sets in $\mathcal{S}$ all have the same size,
 \item and $|S_i \setminus S_{i+1}| = 1.$
\end{enumerate}
\end{definition}
\begin{figure}
\caption{}
\label{firstexample}
\includegraphics[scale=0.7]{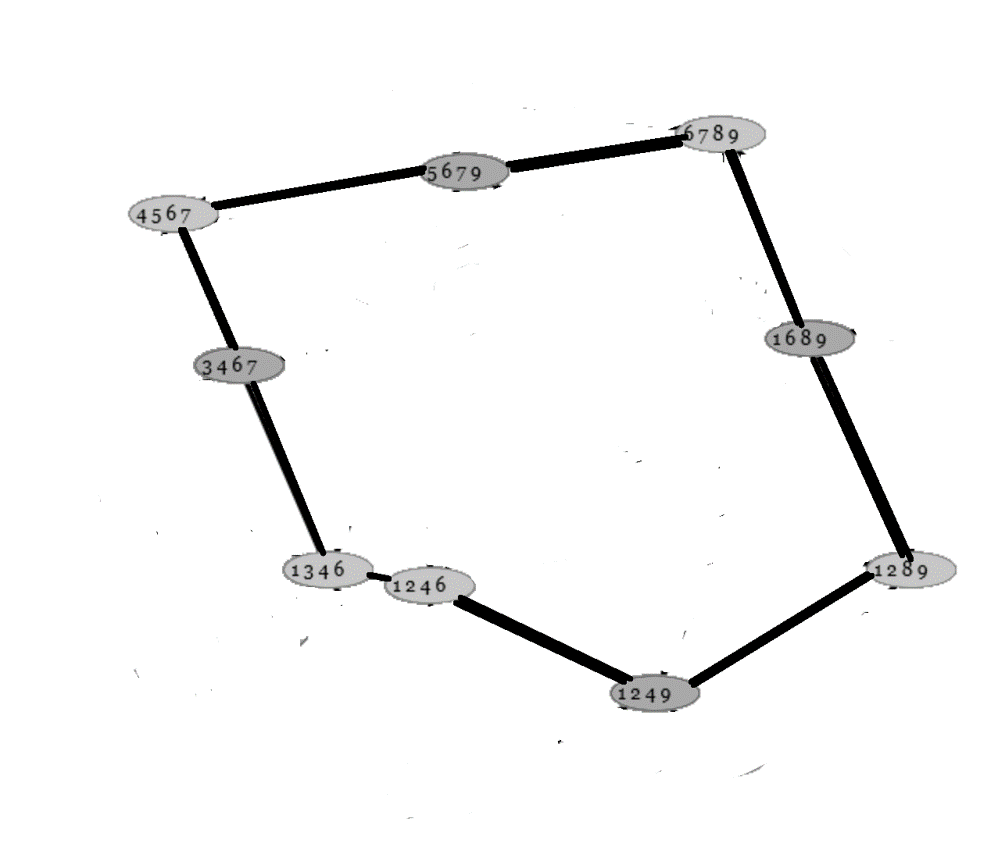}
\end{figure}
\begin{figure}
\caption{}
\label{secondexample}
\includegraphics[scale=0.7]{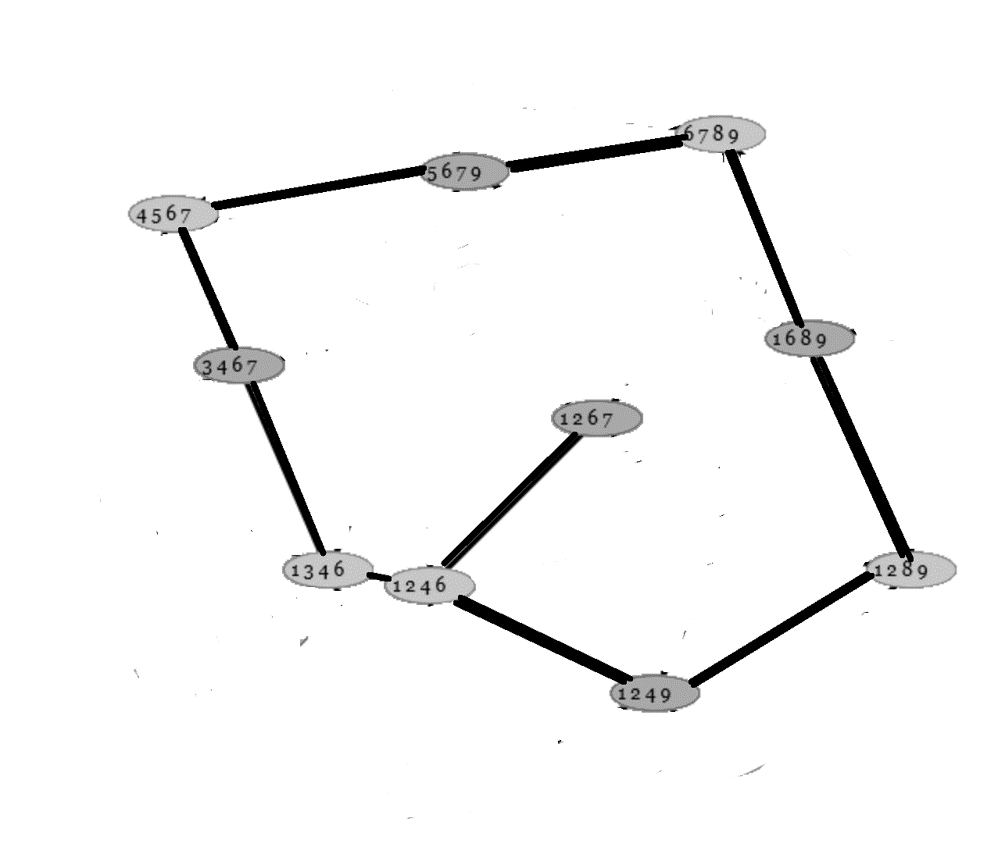}
\end{figure}
We show two examples of quasi-generalized cyclic patterns $S$ such that $\zeta_{\mathcal{S}}$ is prong-closed.
\begin{ex}
\label{firstex}
Suppose that
\[\mathcal{S} = (\left\{1, 3, 4, 6\right\}, \left\{3, 4, 6, 7 \right\}, \left\{4, 5, 6, 7 \right\}, \left\{5, 6, 7, 9 \right\}, \left\{6, 7, 8 ,9 \right\}, \left\{1, 6, 8 ,9 \right\}, \left\{1, 2, 8 ,9 \right\},\]
\[ \left\{1, 2, 4, 9 \right\}, \left\{1, 2, 4, 6 \right\}).\]
Then we have that $\zeta_{\mathcal{S}}$, shown in Figure~$\ref{firstexample}$, is prong-closed.
\end{ex}
\begin{ex}
\label{secondex}
\[\mathcal{S} = (\left\{1, 3, 4, 6\right\}, \left\{3, 4, 6, 7 \right\}, \left\{4, 5, 6, 7 \right\}, \left\{5, 6, 7, 9 \right\}, \left\{6, 7, 8 ,9 \right\}, \left\{1, 6, 8 ,9 \right\}, \left\{1, 2, 8 ,9 \right\}, \left\{1, 2, 4, 9 \right\},\]
\[\left\{1, 2, 4, 6 \right\}, \left\{1, 2, 6, 7 \right\}, \left\{1, 2, 4, 6 \right\}).\]
Then we have that $\zeta_{\mathcal{S}}$, shown in Figure~$\ref{secondexample}$, is prong-closed.
\end{ex}

For an  quasi-generalized cyclic pattern $S$ such that the curve $\zeta_{\mathcal{S}}$ is prong-closed, let
\[\mathcal{D}_{\mathcal{S}}=\left\{X \subset [n] \mid X \text{ is weakly separated from } \mathcal{S}\right\} \setminus \mathcal{S}.\] 
\begin{remark}
Our definition of $\mathcal{D}_{\mathcal{S}}$ differs slightly from the definitions in \cite{DKK14}. The definition in \cite{DKK14} includes $\mathcal{S}$ in $\mathcal{D}_{\mathcal{S}}$, while we do not include $\mathcal{S}$.
\end{remark}
We describe how to decompose $\mathcal{D}_{\mathcal{S}}$ into two pure domains: $\mathcal{D}_{\mathcal{S}}^{in}$ and $\mathcal{D}_{\mathcal{S}}^{out}$. Let $\zeta^{F}_{\mathcal{S}}$ be the curve obtained by concatenating $\zeta^{F1}_{\mathcal{S}}$ and $\zeta^{F2}_{\mathcal{S}}$. Then, by definition, we know that $\zeta^{F}_{\mathcal{S}}$ subdivides $\mathbb{Z}_n$ into two closed regions $\mathcal{R}_{\mathcal{S}}^{in}$ and $\mathcal{R}_{\mathcal{S}}^{out}$ such that
\[\mathcal{R}_{\mathcal{S}}^{in} \cap \mathcal{R}_{\mathcal{S}}^{out}= \zeta_\mathcal{S} \text{ and } \mathcal{R}_{\mathcal{S}}^{in} \cup \mathcal{R}_{\mathcal{S}}^{out}= \mathbb{Z}_n.\]
Then we let
\[\mathcal{D}_{\mathcal{S}}^{in}=\mathcal{D}_{\mathcal{S}} \cap \mathcal{R}_{\mathcal{S}}^{in}\]
and
\[\mathcal{D}_{\mathcal{S}}^{out}=\mathcal{D}_{\mathcal{S}} \cap \mathcal{R}_{\mathcal{S}}^{out}.\]

\subsection{Interior-Reduced Plabic Graphs}
We construct the interior-reduced plabic tiling machinery that is helpful in proving Theorem~$\ref{link}$.

Consider a plabic tiling of a maximal weakly separated collection $\mathcal{V}$. Take a sequence of subsets $\mathcal{I}^{*} = (I_1, I_2, I_3,\ldots,I_i, I_{i+1} = I_1)$. We call $\zeta_{\mathcal{I}^{*}}$ the \textbf{boundary curve}. Consider the sub-plabic tiling consisting of the sets and faces in the plabic tiling of $\mathcal{V}$ on and within $\zeta^F_{\mathcal{I}^{*}}$. Suppose that 
\begin{enumerate}
\item The boundary curve $\zeta_{\mathcal{I}^{*}}$ is prong-closed,
\item and there do not exist $i, j$ with $|i-j| > 1$ such that there is a face containing both $I_i$ and $I_j$ in the aformentioned sub-plabic tiling.
\end{enumerate}
Then we call the sub-plabic tiling together with the boundary curve $\zeta_{\mathcal{I}^{*}}$ an \textbf{interior-reduced plabic tiling}.

Suppose that there is an interior-reduced plabic tiling that consists of sets of the weakly separated multi-collection $\mathcal{W} \subset \mathcal{V}$ (that may have repeated subsets as permitted by $\mathcal{I}^{*}$). We omit the phrase ``multi'' for the remainder of the paper. Then, we denote the interior-reduced plabic tiling of $\mathcal{W}$ as $\mathpzc{PT}_{\mathcal{V}}(\mathcal{W})$. Let the sequence of subsets on the boundary be $\mathcal{I}^{*}$. We let $\mathpzc{PT}_{\mathcal{V}}^{*}(\mathcal{W})$ be $\mathpzc{PT}_{\mathcal{V}}(\mathcal{W})$ with the set labels removed.

We now define the \textbf{interior-reduced plabic graph} $\mathpzc{P}_{\mathcal{V}}^{*}(\mathcal{W})$ to be the dual plabic graph (without strand labels) of $\mathpzc{PT}_{\mathcal{V}}^{*}(\mathcal{W})$ with the following alteration: Suppose that $\mathcal{I}^{*}$ has repeated elements so that $\zeta^{P}_{\mathcal{I}^{*}}$ is nonempty. For every two quasi-adjacent sets $W_1$ and $W_2$ connected by $\zeta^{P}_{\mathcal{I}^{*}}$, consider the faces in $\mathpzc{PT}_{\mathcal{V}}^{*}(\mathcal{W})$ that contain both $W_1$ and $W_2$. Consider the dual vertices of these faces in the dual plabic graph. If there is only one such vertex, then we split this vertex into two vertices by the move (M2). Now, we let the two vertices be $u_1$ and $u_2$. We make $u_1$ and $u_2$ into boundary vertices as follows: For $i \neq j$, consider the strand that goes from $u_i$ to $u_j$. We break this strand into two strands so that the first strand ends at $u_i$ and the second strand begins at $u_j$. 

In the case that $V=W$, we omit the subscripts.

\begin{figure}
\caption{}
\label{PT1}
\includegraphics[scale=0.6]{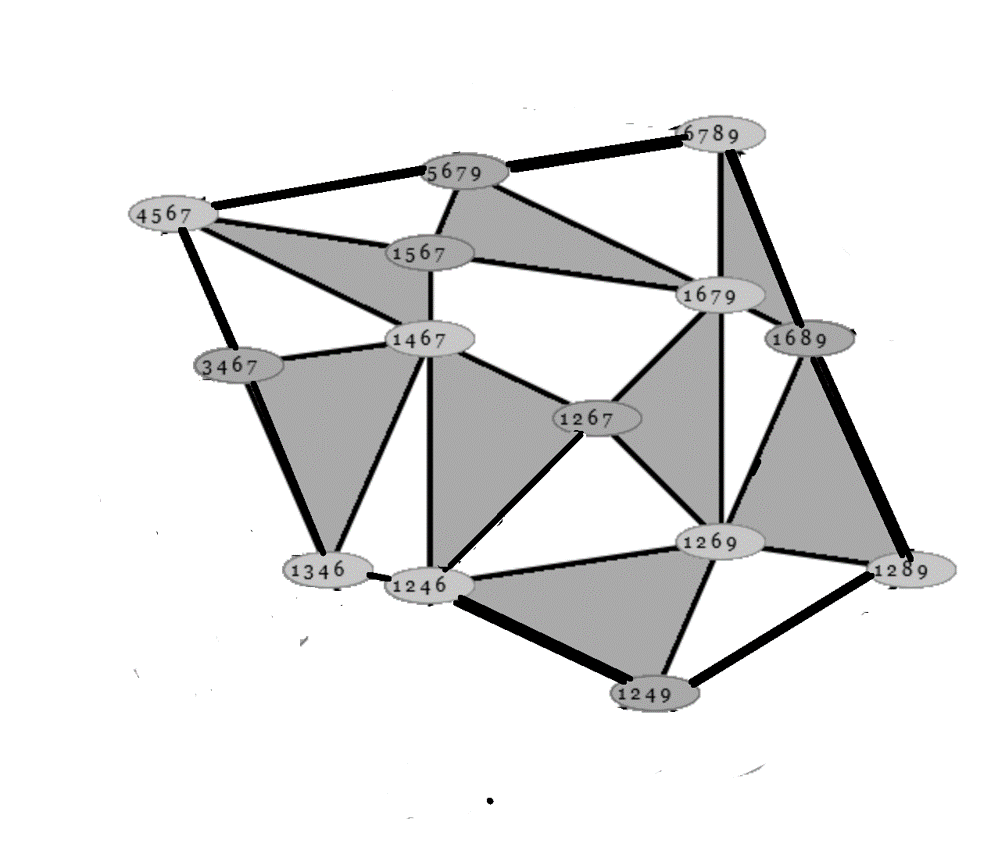}
\includegraphics[scale=0.4]{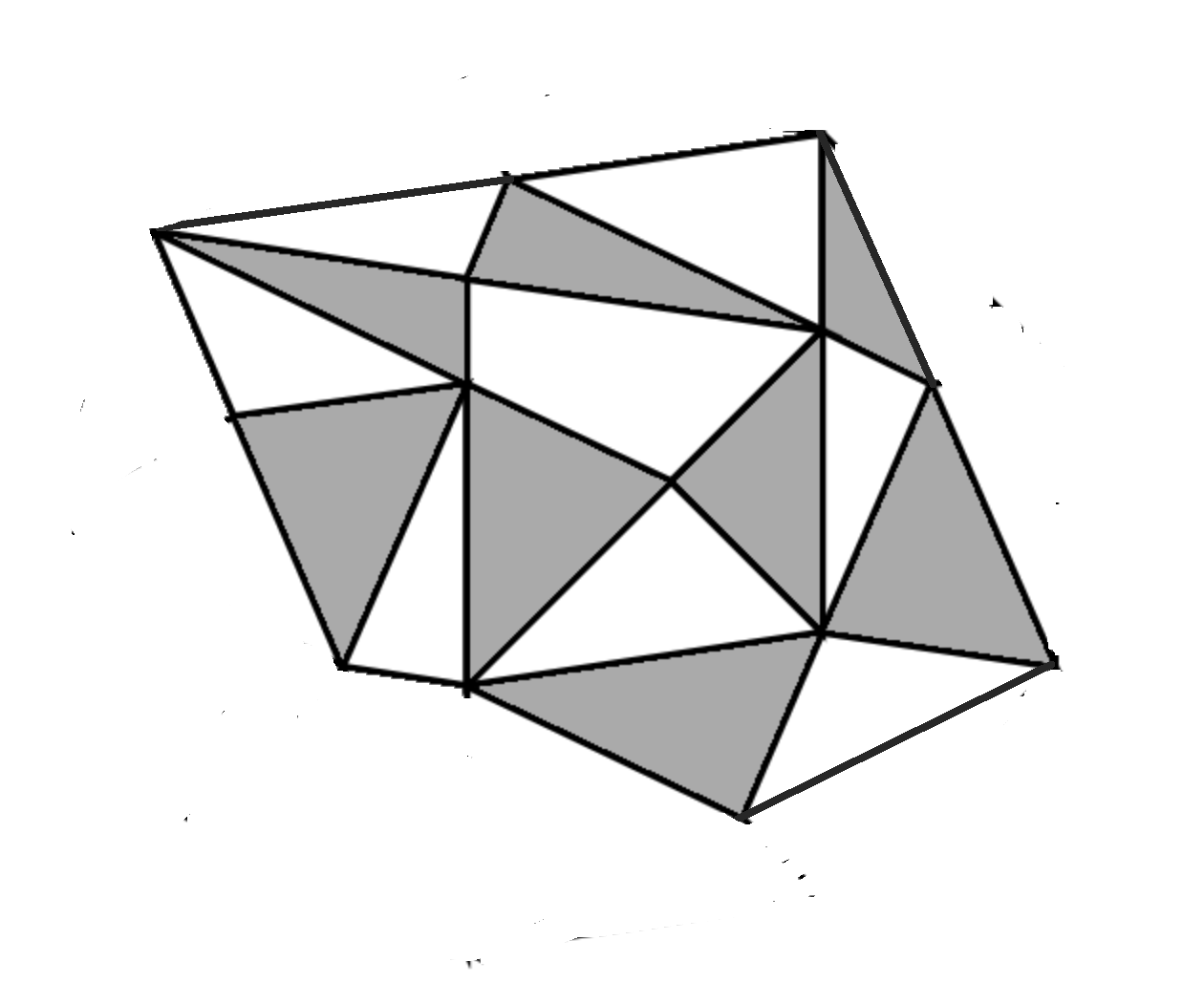}
\includegraphics[scale=0.4]{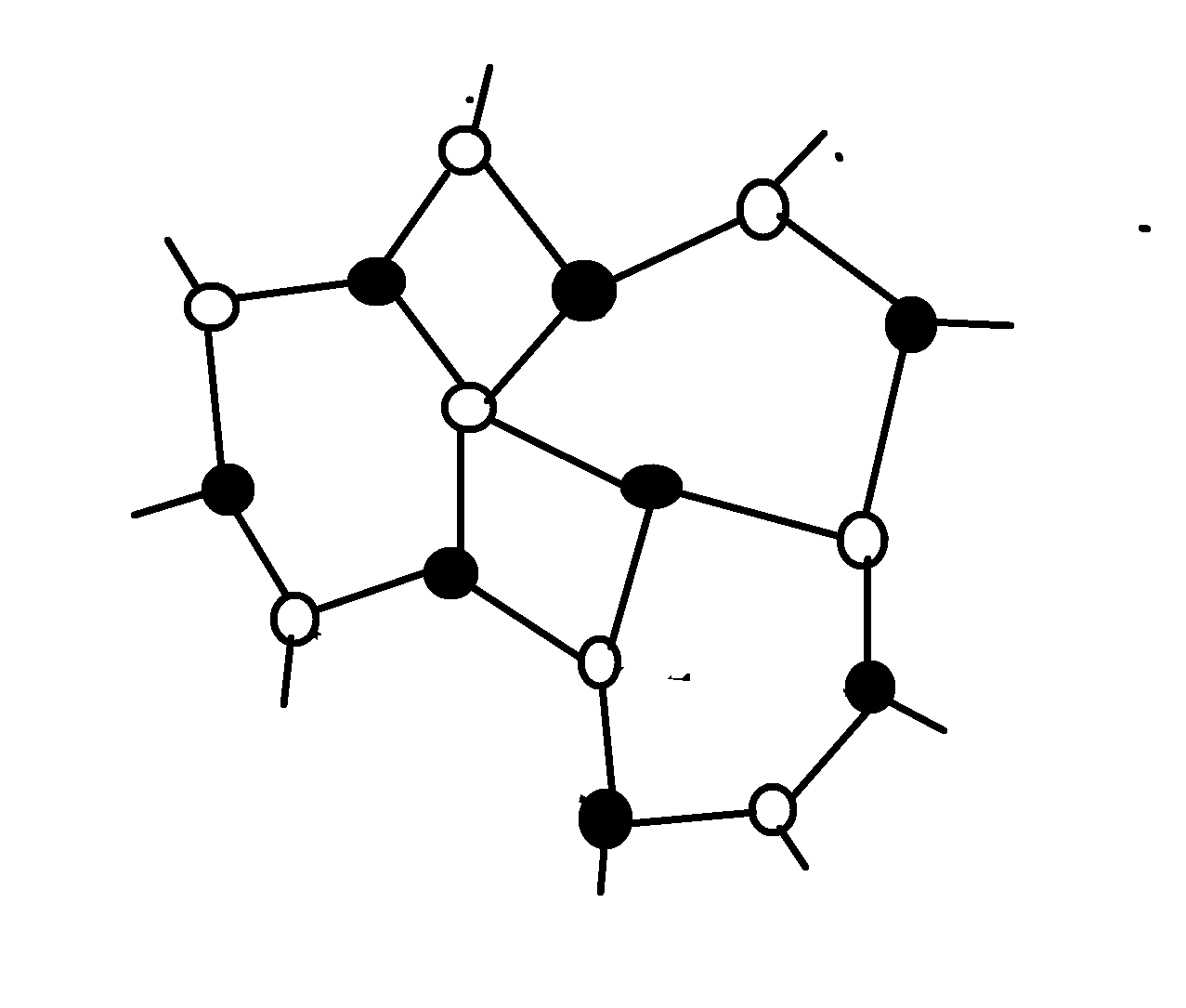}
\end{figure}
\begin{figure}
\caption{}
\label{PT2}
\includegraphics[scale=0.6]{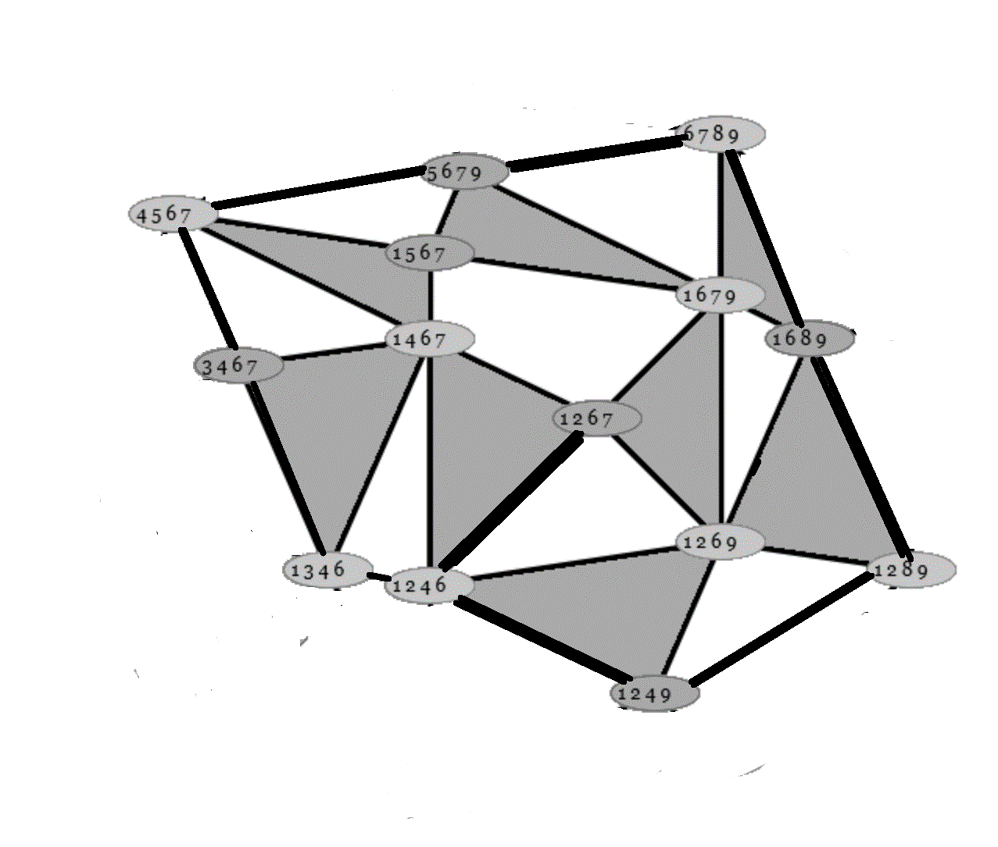}
\includegraphics[scale=0.4]{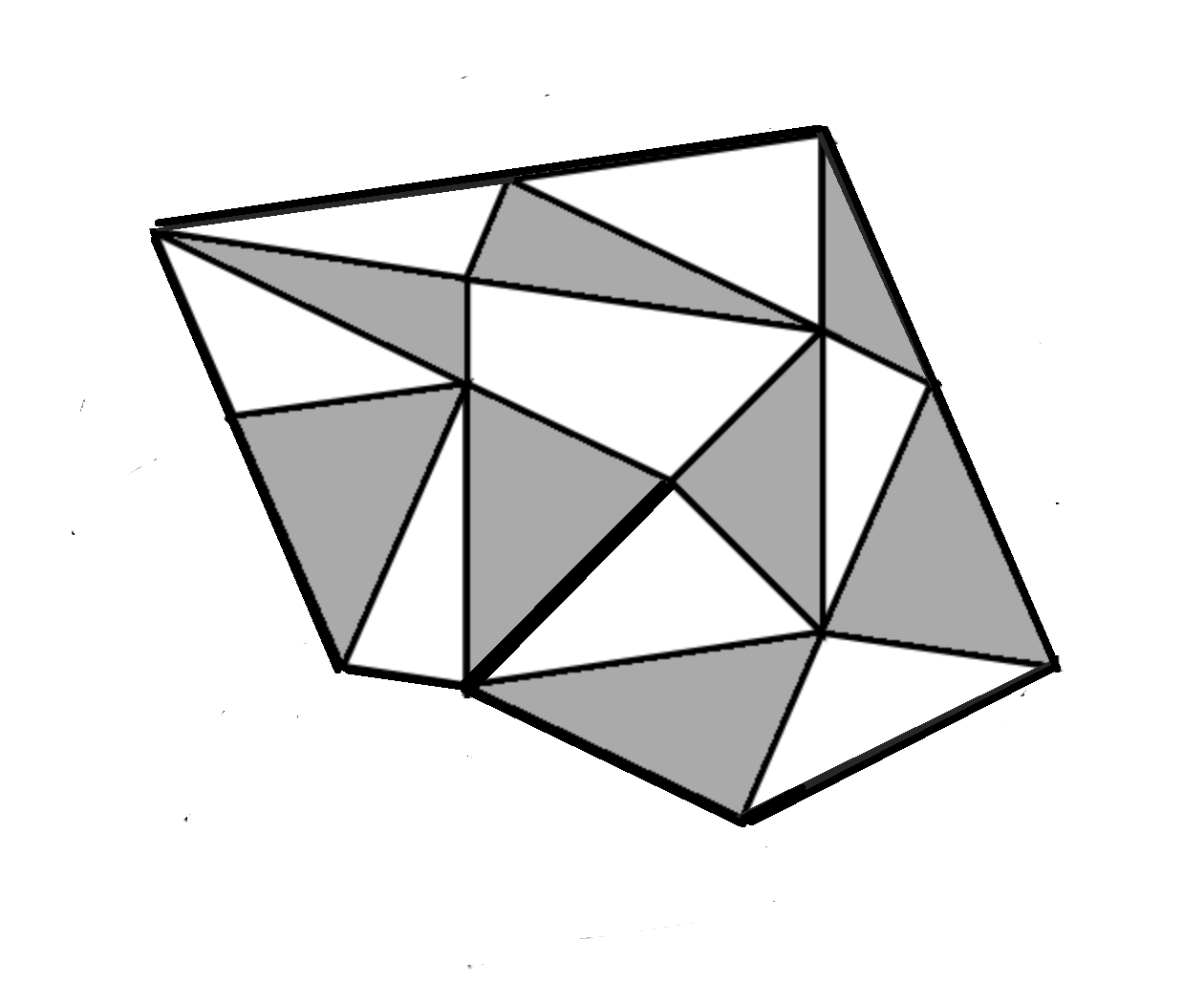}
\includegraphics[scale=0.4]{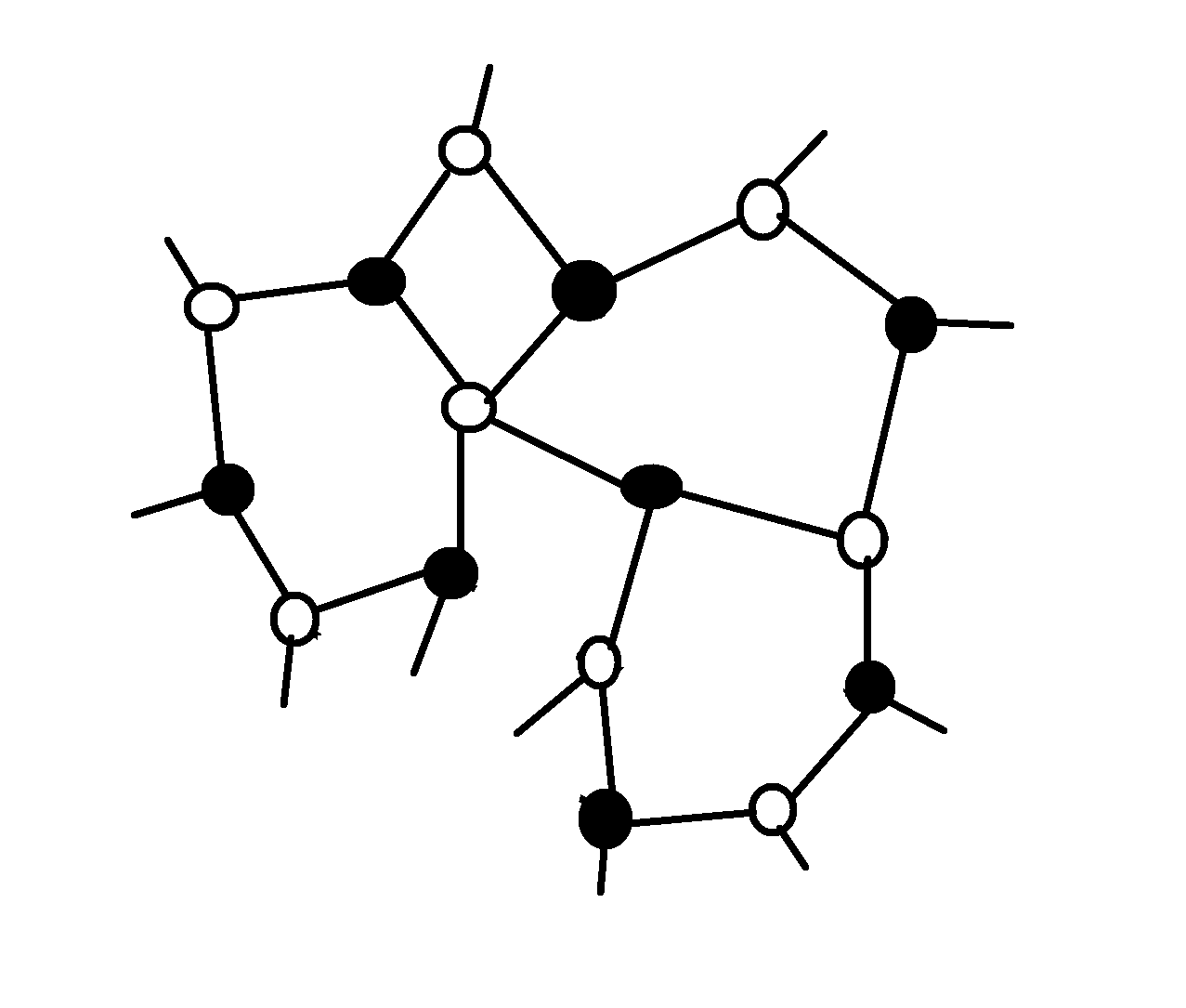}
\end{figure}
We consider the following examples:
\begin{ex}
\label{intred1}
Consider $\mathcal{V}$ as described in Example~$\ref{mainexth}$.
Consider the weakly separated collection $\mathcal{W} \subset \mathcal{V}$ that contains $\left\{1, 3, 4, 6\right\}$, $\left\{3, 4, 6, 7 \right\}$, $\left\{4, 5, 6, 7 \right\}$, $\left\{5, 6, 7, 9 \right\}$, $\left\{6, 7, 8 ,9 \right\}$, $\left\{1, 6, 8 ,9 \right\}$, $\left\{1, 2, 8 ,9 \right\}$, $\left\{1, 2, 4, 9 \right\}$, $\left\{1, 2, 4, 6 \right\}$, $\left\{1, 5, 6, 7 \right\}$, $\left\{1, 4, 6, 7 \right\}$, $\left\{1, 6, 7 ,9 \right\}$, $\left\{1, 2, 6, 7 \right\}$, and $\left\{1, 2, 6, 9\right\}$.
Notice that $\mathpzc{PT}_{\mathcal{V}}(\mathcal{W})$ is an interior-reduced plabic tiling. Figure~$\ref{PT1}$ shows ${\mathpzc{PT}}_{\mathcal{V}}(\mathcal{W})$, $\mathpzc{PT}^{*}_{\mathcal{V}}(\mathcal{W})$, and $\mathpzc{P}_{\mathcal{V}}^{*}(\mathcal{W})$.
\end{ex}
\begin{ex}
\label{intred2}
Consider $\mathcal{V}$ as described in Example~$\ref{mainexth}$.
Consider the weakly separated (multi-)collection $\mathcal{W} \subset \mathcal{V}$ that contains $\left\{1, 3, 4, 6\right\}$, $\left\{3, 4, 6, 7 \right\}$, $\left\{4, 5, 6, 7 \right\}$, $\left\{5, 6, 7, 9 \right\}$, $\left\{6, 7, 8 ,9 \right\}$, $\left\{1, 6, 8 ,9 \right\}$, $\left\{1, 2, 8 ,9 \right\}$, $\left\{1, 2, 4, 9 \right\}$, $\left\{1, 2, 4, 6 \right\}$, $\left\{1, 2, 6, 7 \right\}$, $\left\{1, 2, 4, 6 \right\}$, $\left\{1, 5, 6, 7 \right\}$, $\left\{1, 4, 6, 7 \right\}$, $\left\{1, 6, 7 ,9 \right\}$, and $\left\{1, 2, 6, 9\right\}$. The difference between this $\mathcal{W}$ and the $\mathcal{W}$ in Example~$\ref{intred1}$ in the additional copy of $\left\{1, 2, 4, 6 \right\}$. Notice that $\mathpzc{PT}_{\mathcal{V}}(\mathcal{W})$ is an interior-reduced plabic tiling. Figure~$\ref{PT1}$ shows $\mathpzc{PT}_{\mathcal{V}}(\mathcal{W})$, $\mathpzc{PT}^{*}_{\mathcal{V}}(\mathcal{W})$, and $\mathpzc{P}_{\mathcal{V}}^{*}(\mathcal{W})$.
\end{ex}

We prove the following property of interior-reduced plabic graphs: 
\begin{prop}
Every interior-reduced plabic graph is a reduced plabic graph without strand labels.
\end{prop}
\begin{proof}
We consider the interior-reduced plabic graph $\mathpzc{P}_{\mathcal{V}}^{*}(\mathcal{W})$. First, notice that the condition (2) in the definition of an interior-reduced plabic tiling guarantees that the each boundary vertex of $\mathpzc{P}_{\mathcal{V}}^{*}(\mathcal{W})$ will only have one notch coming out of it on the outside. This means that $\mathpzc{P}^{*}_V$ is a plabic graph. Suppose that we label the strands of the $\mathpzc{P}_{\mathcal{V}}^{*}(\mathcal{W})$ in clockwise order. It suffices to prove that these strands satisfy the properties of a reduced plabic graph. We know that this is true, because the strands of $\mathpzc{P}_{\mathcal{V}}^{*}(\mathcal{W})$ are contained in the strands of the reduced plabic graph of $\mathcal{V}$ (up to breaking and relabeling). 
\end{proof}

Notice that the sequence of sets $\mathcal{I}^{*}$ on the boundary curve $\zeta_{\mathcal{I}^{*}}$ of an interior-reduced plabic tiling form a modified Grassmann-like necklace which is defined as follows:

\begin{definition}
\label{modgras}
Let $\mathcal{I}^{*} = (I^{*}_1, I^{*}_2,\ldots,I^{*}_n, I^{*}_{n+1} = I^{*}_1)$ be a quasi-generalized cyclic pattern. We call $\mathcal{I}^{*}$ a $\textbf{modified Grassmann-like necklace}$ if the following properties are satisfied:
\begin{enumerate}
\item{The curve $\zeta_{\mathcal{I}^{*}}$ is prong-closed.}
\item{There is no $i$ such that $|I_{i+2} \setminus I_{i}| = 1$ (where indices are modulo $n$) and the integers $I_{i} \setminus I_{i+1}, I_{i+1} \setminus I_{i+2}, I_{i+2} \setminus I_i$ are cyclically ordered.}
\end{enumerate}
\end{definition}
\begin{remark}
It follows from the definition that
\[|\cup_{i=1}^{n} I^{*}_i| \geq n.\]
\end{remark}

We associate modified Grassmann-like necklaces with weakly separated multi-collections.
\begin{definition}
Let $\mathcal{I}^{*}$ be a modified Grassmann-like necklace contained in a maximal weakly separated collection $\mathcal{V}$. Then, we define the $\mathbf{\mathcal{I}^{*}}$\textbf{-enclosed collection} of $\mathcal{V}$ to be
$$\displaystyle {\mathcal{V}}^{\mathcal{I}^{*}} = (\mathcal{D}_{\mathcal{I}^{*}}^{in} \cap {\mathcal{V}}) \cup \mathcal{I}^{*},$$
where repeated subsets in $\mathcal{I}^{*}$ continue to be repeated subsets in ${\mathcal{V}}^{\mathcal{I}^{*}}$.
\end{definition}
\begin{remark}
Notice that $\mathpzc{PT}_{\mathcal{V}}({\mathcal{V}}^{\mathcal{I}^{*}})$ is always an interior-reduced plabic tiling. We thus know that for a fixed maximal weakly separated collection $\mathcal{V}$, the $\mathcal{I}^{*}$-enclosed collections $\mathcal{V}^{\mathcal{I}^{*}}$ are in bijection with the interior-reduced plabic tilings/graphs of $V$.
\end{remark}

We show two examples of $\mathcal{I}^{*}$-enclosed collections.
\begin{ex}
\label{de1}
Let $\mathcal{S}$ be as defined in Example~$\ref{firstex}$. Then $\mathcal{I}^{*} = \mathcal{S}$ is a modified-like Grassmann necklace. Consider $\mathcal{V}$ and $\mathcal{W}$ as defined in Example~$\ref{intred1}$. Then ${\mathcal{V}}^{\mathcal{I}^{*}} = \mathcal{W}$.
\end{ex}
\begin{ex}
\label{de2}
Let $\mathcal{S}$ be as defined in Example~$\ref{secondex}$. Then $\mathcal{I}^{*} = \mathcal{S}$ is a modified-like Grassmann necklace. Consider $\mathcal{V}$ and $\mathcal{W}$ as defined in Example~$\ref{intred2}$. Then ${\mathcal{V}}^{\mathcal{I}^{*}} = \mathcal{W}$.
\end{ex}

Now, consider a modified Grassmann-like necklace $\mathcal{I}^{*}$. Suppose that $\mathcal{I}^{*}$ is contained in a maximal weakly separated collection $\mathcal{V}$. We map $\mathcal{I}^{*}$ to a Grassmann necklace $\mathcal{I}$ in such a way that is independent of choice of $\mathcal{V}$. We call $\mathcal{I}$ the $\mathbf{relabeled}$ $\mathbf{Grassmann}$ $\mathbf{necklace}$ of $\mathcal{I}^{*}$. We define the decorated permutation $\pi$ associated to $\mathcal{I}$ as follows. Suppose that $\mathcal{I}^{*} = \left\{I^{*}_1, I^{*}_2,..., I^{*}_{n+1} = I^{*}_1 \right\}$. Suppose that $\mathpzc{PT}_{\mathcal{V}}(V^{\mathcal{I}^{*}})$ has $n$-strands $\mathscr{S} = (S^1, S^2,...S^n)$ such that $S^i$ starts at the face directly clockwise from the subset $I^{*}_i$. If $S^i$ ends at the face directly clockwise from the subset $I^{*}_j$, then we set $\pi(i) = j$.
Notice that $\mathscr{S}$ (and thus the values $\pi(i)$) are independent of the choice of maximal weakly separated collection $\mathcal{V}$. 

We show two examples:
\begin{ex}
\label{grasss1}
Suppose that $\mathcal{I}^{*}$ is defined as in Example~$\ref{de1}$. Then, the relabeled Grassmann necklace is in the equivalence class containing the decorated permutation $365982147$.
\end{ex}
\begin{ex}
\label{grasss2}
If $\mathcal{I}^{*}$ is defined as in Example~$\ref{de2}$, then the relabeled Grassmann necklace is in the equivalence class containing the decorated permutation $5416(10)932(11)87$.
\end{ex}

Now, we define the bijective $\textbf{mapping function}$ $f$ that sends each $\mathcal{I}^{*}$-enclosed collection $\mathcal{V}^{\mathcal{I}^{*}}$ to a maximal weakly separated collection $\mathcal{W}$ in $\mathpzc{G}^{\mathcal{I}}$. In $\mathpzc{P}^{*}_\mathcal{V}(\mathcal{V}^{\mathcal{I}^{*}})$, for each $1 \le i \le n$, we label the strand $S^i$ in $\mathpzc{PT}_{\mathcal{V}}(V^{\mathcal{I}^{*}})$ with the number $\pi(i)$. Now, we label the sets using these new strand labels. Let this maximal weakly separated collection in $\mathpzc{G}^{\mathcal{I}}$ be $\mathcal{W}$. The function $f$ takes $\mathcal{V}^{\mathcal{I}^{*}}$ to $\mathcal{W}$.

We show how to obtain $\mathcal{V}^{\mathcal{I}^{*}}$ from $W$. We relabel the strands of $\mathpzc{P}^{*}(W)$, replacing the label of the strand that is labeled $i$ in $\mathpzc{PT}(W)$ with the label of the strand $S^{\pi^{-1}(i)}$ in $\mathpzc{PT}_{\mathcal{V}}(V^{\mathcal{I}^{*}})$ for $1 \le i \le n$. We label the sets and to each set, add the elements of $\cap^{n}_{j = 1} I^{*}_j$ (the elements common to all of the sets in the modified Grassmann-like necklace). This recovers the $\mathcal{I}^{*}$-enclosed collection $V^{\mathcal{I}^{*}}$.

We say that $\mathpzc{P}_{\mathcal{V}_1}^{*}(\mathcal{W}_1) \cong \mathpzc{P}_{\mathcal{V}_2}^{*}(\mathcal{W}_2)$ and $\mathpzc{PT}_{\mathcal{V}_1}^{*}(\mathcal{W}_1) \cong \mathpzc{PT}_{\mathcal{V}_2}^{*}(\mathcal{W}_2)$ if the relabeled Grassmann necklace of the modified-like Grassmann necklace that forms the boundary of $\mathpzc{PT}_{\mathcal{V}_1}(\mathcal{W}_1)$ is in the same equivalence class as the relabeled Grassmann necklace of the modified-like Grassmann necklace that forms the boundary of $\mathpzc{PT}_{\mathcal{V}_2}(\mathcal{W}_2)$. Notice that $\mathpzc{P}_{\mathcal{V}_1}^{*}(\mathcal{W}_1) \cong \mathpzc{P}_{\mathcal{V}_2}^{*}(\mathcal{W}_2)$ and $\mathpzc{PT}_{\mathcal{V}_1}^{*}(\mathcal{W}_1) \cong \mathpzc{PT}_{\mathcal{V}_2}^{*}(\mathcal{W}_2)$ if $\mathpzc{P}_{\mathcal{V}_1}^{*}(\mathcal{W}_1)$ can be transformed into $\mathpzc{P}_{\mathcal{V}_2}^{*}(\mathcal{W}_2)$ through a sequence of rotations, reflections, and color flipping operations.

We now relate the plabic tilings of $\mathcal{I}^{*}$-enclosed collections with plabic tilings of maximal weakly separated collections over $\mathcal{I}$ (the relabeled Grassmann necklace of $\mathcal{I}^{*}$). This gives us an isomorphism between certain $\mathcal{C}$-constant graphs and certain exchange graphs.
\begin{lemma}
\label{quasiexch}
Consider a maximal weakly separated collection $\mathcal{V}$ over a Grassmann necklace $\mathcal{J}$. Consider a modified-like Grassmann necklace $\mathcal{I}^{*} \subset V$ over the relabeled Grassmann necklace $\mathcal{I}$ and the $\mathcal{I}^{*}$-enclosed collection $\mathcal{V}^{\mathcal{I}^{*}}$. Then, we have that
\[
\left\{\mathpzc{PT}^{*}_{\mathcal{W}}(\mathcal{W}^{\mathcal{I}^{*}}) \mid \mathcal{W} \in \mathpzc{G}^{\mathcal{J}}((\mathcal{V} \setminus \mathcal{V}^{\mathcal{I}^{*}}) \cup \mathcal{I}^{*}) \right\} \cong \left\{\mathpzc{PT}^{*}(\mathcal{W}) \mid \mathcal{W} \in \mathpzc{G}^{\mathcal{I}} \right\} \]
and
\[
\mathpzc{G}^{\mathcal{J}}((\mathcal{V} \setminus \mathcal{V}^{\mathcal{I}^{*}}) \cup \mathcal{I}^{*}) \cong \mathpzc{G}^{\mathcal{I}}\]
where repeated subsets in $(\mathcal{V} \setminus \mathcal{V}^{\mathcal{I}^{*}}) \cup \mathcal{I}^{*}$ are deleted so that $(\mathcal{V} \setminus \mathcal{V}^{\mathcal{I}^{*}}) \cup \mathcal{I}^{*}$ contains at most one of each subset.
\end{lemma}

\begin{proof}
Notice that first equation follows from the bijective mapping function between $\mathcal{I}^{*}$-enclosed collections and maximal weakly separated collections over a Grassmann necklace $\mathcal{I}$ and the definition of interior-reduced plabic tilings. Since square moves will have the same effect on the two plabic tilings, up to relabeling, it also follows that the second equation holds.
\end{proof}
\subsection{Direct Product}
We define and prove some results regarding the decomposition set and direct product of Grassmann necklaces.

We define the following notion of decomposition. Let $\mathcal{I} = (I_1, I_2, \ldots I_n)$ be a Grassmann necklace and $\zeta_{\mathcal{I}}$ be the associated prong-closed curve. We construct the following complex: begin with $\zeta_{\mathcal{I}}$ and add an edge between the embeddings of $I_i$ and $I_j$ if and only if $I_i$ and $I_j$ are quasi-adjacent. Then, this complex will look like a graph with many cycles (some complete graphs) glued together and each cycle (or complete graph) will itself be a modified-like Grassmann necklace. Suppose there are $i$ such cycles (or complete graphs). We define the  \textbf{decomposition set} to be 
\[\mathcal{D*}^{\mathcal{I}} = \left\{\mathcal{D*}^{\mathcal{I}}(j) \mid 1 \le j \le i\right\},\] where each $\mathcal{D*}^{\mathcal{I}}(j)$ is one of the aforementioned modified-like Grassmann necklaces. We define the \textbf{relabeled decomposition set} to be 
\[\mathcal{D}^{\mathcal{I}} = \left\{\mathcal{D}^{\mathcal{I}}(j) \mid 1 \le j \le i\right\}\] where $\mathcal{D}^{\mathcal{I}}(j)$ is the relabeled Grassmann necklace of $\mathcal{D*}^{\mathcal{I}}(j)$.

\begin{remark}
Notice that
$$\displaystyle \sum_{j=1}^{|\mathcal{D}^{\mathcal{I}}|} i(\mathcal{D}^{\mathcal{I}}(j)) = i(\mathcal{I}).$$
\end{remark}

\begin{figure}
\caption{}
\label{dec}
\includegraphics[scale=0.75]{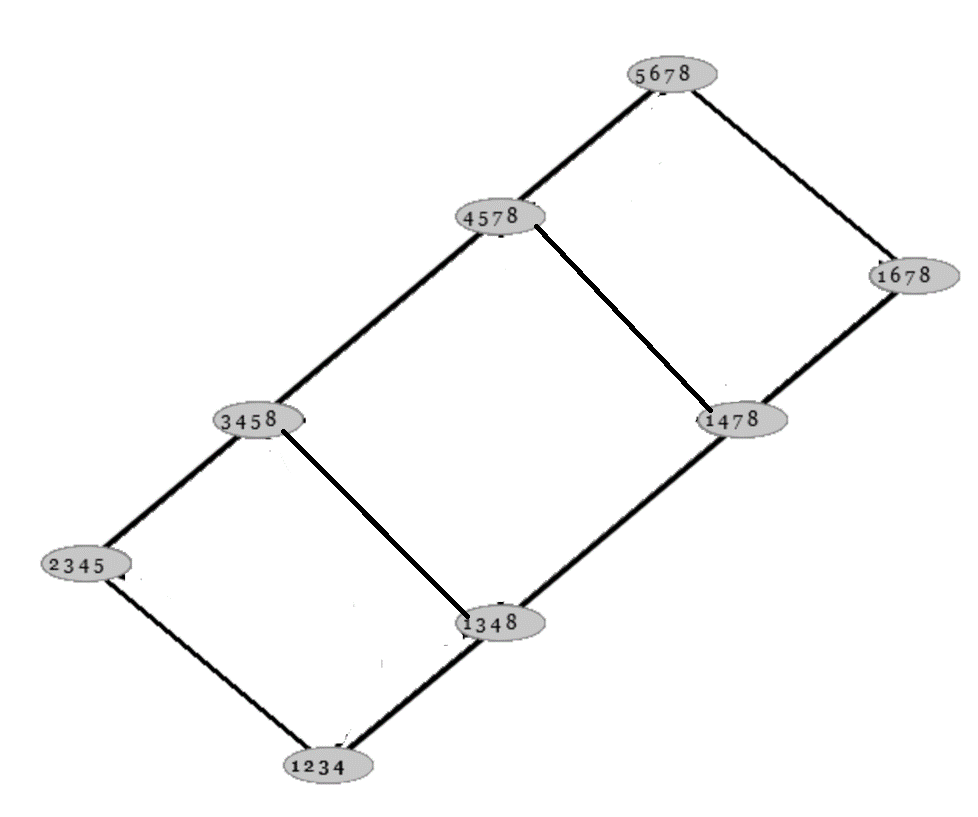}
\end{figure}
\begin{ex}
Consider the Grassmann necklace
$$\mathcal{I}= \left\{ \left\{1, 2, 3, 4 \right\},  \left\{2, 3, 4, 5 \right\},  \left\{3, 4, 5, 8 \right\},  \left\{4, 5, 7, 8 \right\}, \left\{5, 6, 7, 8 \right\},  \left\{1, 6, 7, 8 \right\}\right\} \cup$$
$$\left\{\left\{1, 4, 7, 8 \right\},  \left\{1, 3, 4, 8 \right\} \right\}.$$.
The decomposition complex of $\mathcal{I}$ is shown in Figure~$\ref{dec}$. Notice that $\mathcal{D*}^{\mathcal{I}}$ contains the following three modified-like Grassmann necklaces:
\begin{align*}
\mathcal{D*}^{\mathcal{I}}(1) &= (\left\{1, 2, 3, 4 \right\},  \left\{2, 3, 4, 5 \right\},  \left\{3, 4, 5, 8 \right\},  \left\{1, 3, 4, 8 \right\}),\\
\mathcal{D*}^{\mathcal{I}}(2) &= (\left\{1, 3, 4, 8  \right\},  \left\{3, 4, 5, 8 \right\},  \left\{4, 5, 7 ,8 \right\},  \left\{1, 4, 7, 8 \right\}), \\
\mathcal{D*}^{\mathcal{I}}(3) &= (\left\{1, 4, 7, 8 \right\},  \left\{4, 5, 7, 8 \right\},  \left\{5, 6, 7, 8 \right\},  \left\{1, 6, 7, 8 \right\}).
\end{align*}
\end{ex}

We now define notions related to the direct product of Grassmann necklaces. 
\begin{definition}
Given a Grassmann necklace $\mathcal{I}$ with associated permutation $\pi_1$ on $[n]$ and $\mathcal{J}$ with associated permutation $\pi_2$ on $[m]$, we let the \textbf{glued Grassmann necklace} $\mathcal{I} \square \mathcal{J}$ have decorated permutation $\pi_3$ on $[n+m-2]$ defined as follows:
\[\begin{cases}
\pi_3(a) = \pi_1(a) & 1 \leq a < n, \pi_1(a) \neq n \\
\pi_3(a) = \pi_2(1)+n-2 & \pi_1(a) = n \\
\pi_3(n+a-2) = \pi_2(a)+ n-2 & 2<a\leq m, \pi_2(a) \neq 1\\
\pi_3(n+a-2) = \pi_1(n) & \pi_2(a) = 1.
\end{cases}\]
\end{definition} 
\begin{ex}
\label{dpex}
Let $\mathcal{I}$ have decorated permutation $34512$ and $\mathcal{J}$ have decorated permutation $3412$. Then $\mathcal{I}\square\mathcal{J}$ has decorated permutation $3461725$.
\end{ex}

The glued Grassmann necklace $\mathcal{I} \square \mathcal{J}$ can be obtained by gluing together the plabic tilings of $\mathcal{I}$ and $\mathcal{J}$ along the edge between $I_n$ and $I_1$, and $J_1$ and $J_2$. In fact, this idea can be extended to the maximal weakly separated collections over these Grassmann necklaces. Let $\mathcal{V}_1$ be in $\mathpzc{G}^{\mathcal{I}}$ and let $\mathcal{V}_2$ be in $\mathpzc{G}^{\mathcal{J}}$. Then let $\square(\mathcal{V}_1, \mathcal{V}_2)$ be the maximal weakly separated collection in  $\mathpzc{G}^{\mathcal{I} \square \mathcal{J}}$ obtained by gluing $\mathcal{V}_1$ and $\mathcal{V}_2$ along that same edge. 

\begin{ex}
Let $\mathcal{I}$ and $\mathcal{J}$ be defined as in Example~$\ref{dpex}$. Let $\mathcal{V}_1$ be a maximal weakly separated collection over a Grassmann necklace $\mathpzc{I}$ and let $\mathcal{V}_2$ be a maximal weakly separated collection over a Grassmann necklace $\mathpzc{J}$. Figure~$\ref{glu1}$ shows $\mathpzc{PT}^{*}(\mathcal{V}_1)$,$\mathpzc{PT}^{*}(\mathcal{V}_2)$, and $\mathpzc{PT}^{*}(\square(\mathcal{V}_1, \mathcal{V}_2))$. 
\end{ex}
\begin{figure}
\caption{}
\label{glu1}
\includegraphics[scale=0.3]{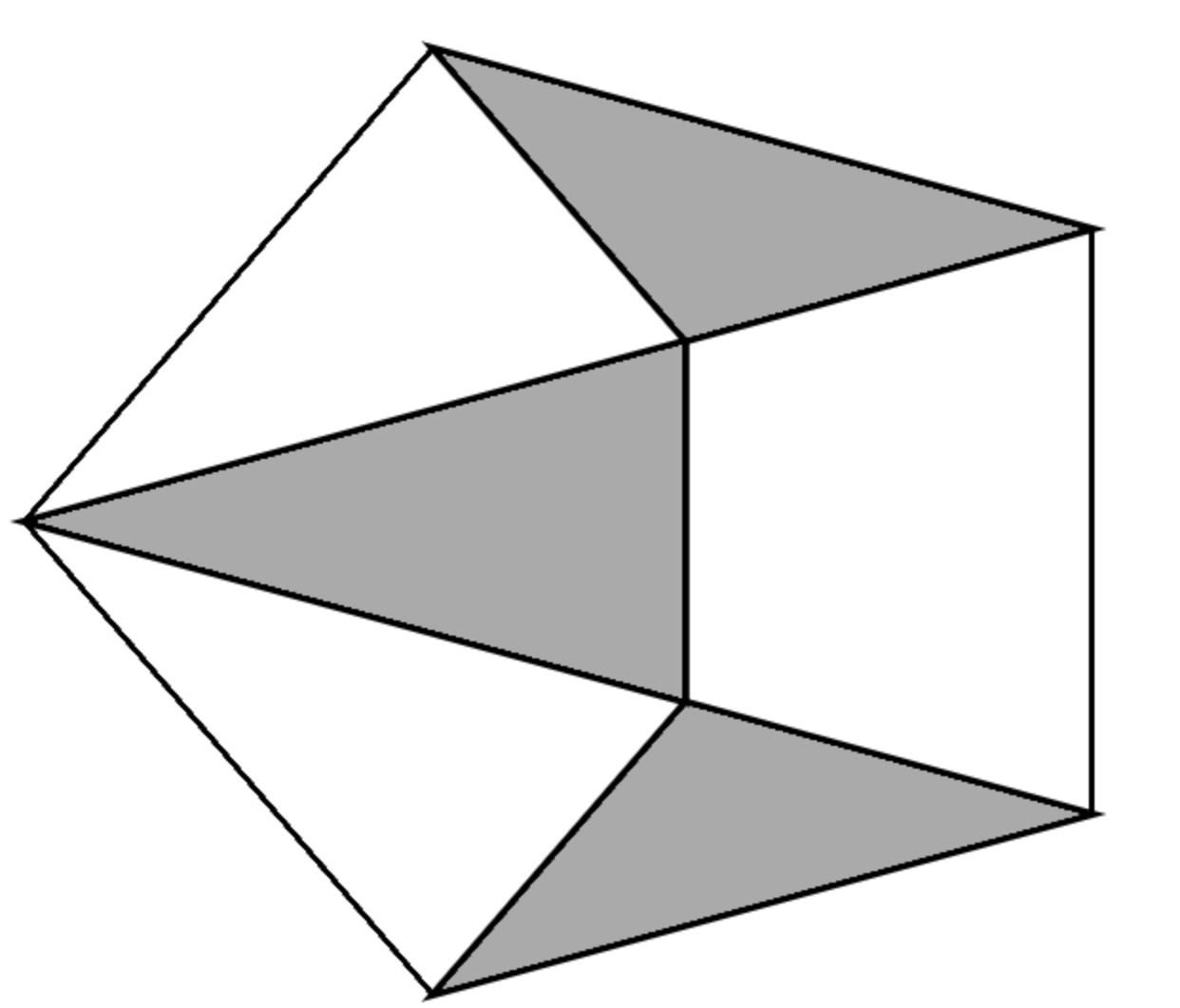}
\includegraphics[scale=0.3]{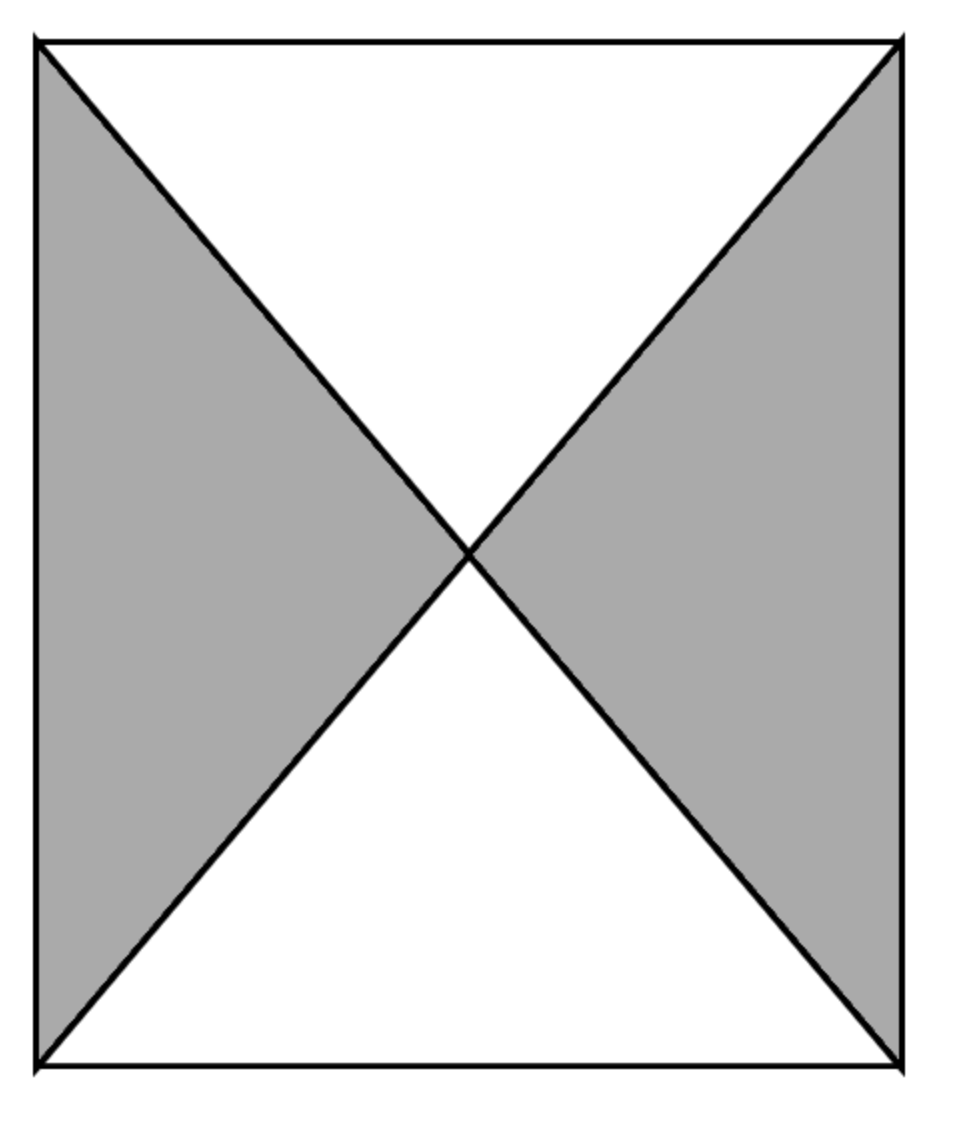}
\includegraphics[scale=0.5]{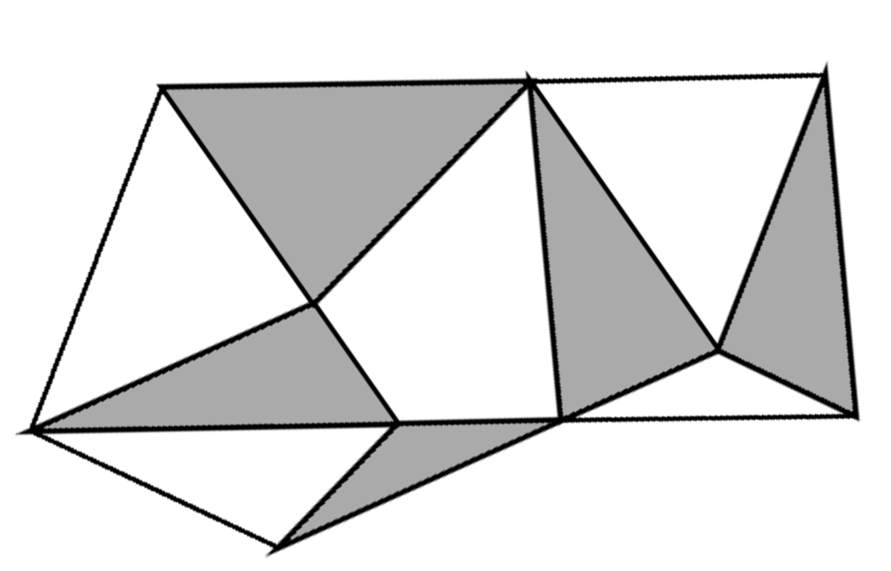}
\end{figure}
We now prove the following proposition about glued Grassmann necklaces:
\begin{prop}
\label{dpp}
Given a Grassmann necklace $\mathcal{I}$ with associated permutation $\pi_1$ on $[n]$ and a Grassmann necklace $\mathcal{J}$ with associated permutation $\pi_2$, we have that
$$\displaystyle \mathpzc{G}^{\mathcal{I} \square \mathcal{J}} \cong \mathpzc{G}^{\mathcal{I}} \square \mathpzc{G}^{\mathcal{J}}$$
and 
\[i(\mathcal{I} \square \mathcal{J}) = i(\mathcal{I}) + i(\mathcal{J}).\]
\end{prop}
\begin{proof}
Notice that
\[\left\{\square(\mathcal{V}, \mathcal{W}) \mid \mathcal{V}\in \mathpzc{G}^{\mathcal{I}}, \mathcal{W} \in \mathpzc{G}^{\mathcal{J}} \right\} = \left\{\mathcal{V} \mid \mathcal{V} \in \mathpzc{G}^{\mathcal{I} \square \mathcal{J}}\right\}. \]
It follows that 
\[i(\mathcal{I} \square \mathcal{J}) = i(\mathcal{I}) + i(\mathcal{J}).\]
Furthermore, the maximal weakly separated collection $\square(\mathcal{V}_1,\mathcal{W}_1)$ can be mutated into $\square(\mathcal{V}_2,\mathcal{W}_2)$ if and only if $\mathcal{V}_1$ can be mutated into $\mathcal{V}_2$ in $\mathpzc{G}^{\mathcal{I}}$ or $\mathcal{W}_1$ can be mutated into $\mathcal{W}_2$ $\mathpzc{G}^{\mathcal{J}}$. This means that $\mathpzc{G}^{\mathcal{I} \square \mathcal{J}} \cong \mathpzc{G}^{\mathcal{I}} \square \mathpzc{G}^{\mathcal{J}}.$
\end{proof}

We prove the following lemma involving direct products and decomposition sets.
\begin{lemma}
\label{dirproduct}
For a Grassmann necklace $\mathcal{I}$, the following is true:
$$\displaystyle \mathpzc{G}^{\mathcal{I}} \cong \square_{j=1}^{|\mathcal{D}^{\mathcal{I}}|} \mathpzc{G}^{\mathcal{D}^{\mathcal{I}}(j)}.$$ If $\mathpzc{G}^{\mathcal{I}}$ is mutation-friendly, then each $\mathpzc{G}^{\mathcal{D}^{\mathcal{I}}(j)}$ is mutation-friendly for $1 \le j \le |\mathcal{D}^{\mathcal{I}}|$.
For any collection of exchange graphs $\mathpzc{G}^{\mathcal{I}_1}$, $\mathpzc{G}^{\mathcal{I}_2}$,$\ldots$,$\mathpzc{G}^{\mathcal{I}_{|\mathcal{D}^{\mathcal{I}}|}}$, there exists a Grassmann necklace $\mathcal{K}$ such that
$$\displaystyle \mathpzc{G}^{\mathcal{K}} \cong \square_{j=0}^{|\mathcal{D}^{\mathcal{I}}|} \mathpzc{G}^{\mathcal{I}_j}$$
and
$$\displaystyle i(\mathpzc{G}^{\mathcal{K}}) = \sum_{j=0}^{|\mathcal{D}^{\mathcal{I}}|} i(\mathpzc{G}^{\mathcal{I}_j}).$$
If all of the $\mathpzc{G}^{\mathcal{I}_j}$ are mutation-friendly, then there exists such a $\mathcal{K}$ that is mutation-friendly.
\end{lemma}
\begin{proof}
For $1 \le j \le |\mathcal{D}^{\mathcal{I}}|$, we let $f^j$ be the mapping function for $\mathcal{D*}^{\mathcal{I}}(j)$ and $\mathcal{V}_j$ be in $\mathpzc{G}^{\mathcal{D}^{\mathcal{I}(j)}}$. Then we define $\mathcal{M}(\mathcal{V}_1, \mathcal{V}_2,\ldots,\mathcal{V}_{|\mathcal{D}^{\mathcal{I}}|})$ to be the maximal weakly separated collection in $\mathpzc{G}^{\mathcal{I}}$ such that $f^j(\mathcal{V}_j) = \mathcal{V}^{\mathcal{D*}^{\mathcal{I}}(j)}$ for $1 \le j \le |\mathcal{D}^{\mathcal{I}}|$.
Notice that 
\[\left\{\mathcal{M}(\mathcal{V}_1, \mathcal{V}_2,\ldots,\mathcal{V}_{|\mathcal{D}^{\mathcal{I}}|}) \mid \mathcal{V}_j \in \mathpzc{G}^{\mathcal{D}^{\mathcal{I}(j)}} \text{ for } 1 \le j \le |\mathcal{D}^{\mathcal{I}}| \right\} = \left\{\mathcal{V} \in \mathpzc{G}^{\mathcal{I}} \right\} .\]
It is clear that $\mathcal{M}(\mathcal{V}_1, \mathcal{V}_2,\ldots,\mathcal{V}_{|\mathcal{D}^{\mathcal{I}}|})$ can be mutated into $\mathcal{M}(\mathcal{W}_1, \mathcal{W}_2,\ldots,\mathcal{W}_{|\mathcal{D}^{\mathcal{I}}|})$ in one square move if and only if ${\mathcal{V}}_{j}$ can be mutated into ${\mathcal{W}}_j$ for some $1 \le j \le |\mathcal{D}^{\mathcal{I}}|$ and $\mathcal{V}_j = \mathcal{W}_j$ for all $j \neq |\mathcal{D}^{\mathcal{I}}|.$ This shows that $\mathpzc{G}^{\mathcal{I}} \cong \square_{j=1}^{|\mathcal{D}^{\mathcal{I}}|} \mathpzc{G}^{\mathcal{D}^{\mathcal{I}}(j)}.$

Given a collection of Grassmann necklaces $\mathcal{I}_1,\ldots,\mathcal{I}_i$, a repeated application of Proposition~$\ref{dpp}$ shows that there exists an isomorphic exchange graph with the appropriate interior size.

Notice that the mutation-friendly conditions hold since
\begin{enumerate}
\item{The decomposition set of a mutation-friendly Grassmann necklace consists of modified-like Grassmann necklaces over relabeled Grassmann necklaces that are mutation-friendly.}
\item{The glued Grassmann necklace of two mutation-friendly Grassmann necklaces is also mutation-friendly.}
\end{enumerate}
\end{proof}

This motivates the definition of a prime Grassmann necklace:
\begin{definition}
We call a Grassmann necklace \textbf{prime} if its decomposition set has only one element. We also call the corresponding exchange graph prime.
\end{definition}
The definition of the $\square$ operation on decorated permutations allows us to characterize the decorated permutations of prime Grassmann necklaces:
\begin{cor}
If a Grassmann necklace has decorated permutation $\pi$, it is not prime if there exists a cyclically considered interval $[i, j]$ such that $i \neq j$ and $|\pi([i,j]) \setminus [i,j]| = 1$.
\end{cor}

\subsection{Adjacency Graphs and Clusters}
We define and prove results relating to adjacency.

We define the adjacency graph as follows:
\begin{definition}
Let $\mathcal{V}$ be a maximal weakly separated collection over a Grassmann necklace $\mathcal{I}$. Let $\mathcal{W}$ be a weakly separated collection contained in $\mathcal{V}$ such that $|\mathcal{W} \cap \mathcal{I}| = 0$. Then, we define the adjacency graph $\mathpzc{AG}^{\mathcal{V}}(\mathcal{W})$ as follows:
\begin{itemize}
\item{The vertex set of $\mathpzc{AG}^{\mathcal{V}}(\mathcal{W})$ is $\mathcal{W}$.}
\item{The vertices $V_1$ and $V_2$ are connected by an edge if and only if they are adjacent in $\mathcal{V}$.}
\end{itemize}
\end{definition}
\begin{figure}
\caption{}
\label{agex}
\includegraphics[scale=1]{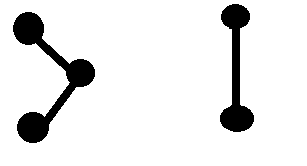}
\end{figure}
\begin{ex}
\label{agextext}
Let $\mathcal{V}$ be defined as in Example~$\ref{mainexth}$. We define the weakly separated collection $\mathcal{W} \subset \mathcal{V}$ to be
\[\mathcal{W} = \left\{\left\{1, 3, 4, 5\right\}, \left\{1, 3, 4, 6\right\}, \left\{3, 4, 6, 7\right\}, \left\{1, 6, 7, 9\right\},\left\{1, 6, 8, 9\right\} \right\}. \]
Then $\mathpzc{AG}^{\mathcal{V}}(\mathcal{W})$ is shown in Figure~$\ref{agex}$.
\end{ex}
In the case that $\mathcal{W} = \mathcal{V} \setminus \mathcal{I}$, we denote the adjacency graph as $\mathpzc{AG}^{\mathcal{V}}$.

We prove the following proposition regarding the adjacency graphs of prime Grassmann necklaces.
\begin{prop}
For a given Grassmann necklace $\mathcal{I}$, the adjacency graph $\mathpzc{AG}^{\mathcal{V}}$ is either connected for all $\mathcal{V} \in \mathpzc{G}^{\mathcal{I}}$ or not connected for any $\mathcal{V} \in \mathpzc{G}^{\mathcal{I}}$. In particular, $\mathpzc{AG}^{\mathcal{V}}$ is connected if and only if $\mathcal{I}$ is a prime Grassmann necklace.
\end{prop}
\begin{proof}
For $1 \le j \le |\mathcal{D}^{\mathcal{I}}|$, we let 
\[\mathpzc{G}_j = \mathpzc{AG}^{\mathcal{V}}(\mathcal{V}^{\mathcal{D*}^{\mathcal{I}(j)}} \setminus \mathcal{D*}^{\mathcal{I}}(j)).\]

We will show that 
\[\mathpzc{AG}^{\mathpzc{V}} = \cup_{j=1}^{|\mathcal{D}^{\mathcal{I}}|} \mathpzc{G}_j.\]
First, notice that for $j_1 \neq j_2$, we have that 
\[(\mathcal{V}^{\mathcal{D*}^{\mathcal{I}(j_1)}} \setminus \mathcal{D*}^{\mathcal{I}}(j_1)) \cap(\mathcal{V}^{\mathcal{D*}^{\mathcal{I}(j_2)}} \setminus \mathcal{D*}^{\mathcal{I}}(j_2)) = \emptyset,\] so $\mathpzc{G}_{j_1}$ and $\mathpzc{G}_{j_2}$ do not share any of the same vertices. Also, notice that $\mathpzc{AG}^V$ and $\cup_{j=1}^{|\mathcal{D}^{\mathcal{I}}|} \mathpzc{G}_j$ both consist of the sets in the weakly separated collection $\mathpzc{V} \setminus \mathcal{I}$. Furthermore, notice that given sets $V_1 \in {\mathcal{V}}^{\mathcal{D*}^{\mathcal{I}}(j_1)} \setminus \mathcal{D*}^{\mathcal{I}}(j_1)$ and $V_2 \in {\mathcal{V}}^{\mathcal{D*}^{\mathcal{I}}(j_2)} \setminus \mathcal{D*}^{\mathcal{I}}(j_2)$ where $j_1 \neq j_2$, we have that $V_1$ is not adjacent to $V_2$ in $V$, so there are no edges between vertices in $\mathpzc{G}_{j_1}$ and $\mathpzc{G}_{j_2}$ for $j_1 \neq j_2$.

Also, we know that $\mathpzc{G}_j$ is connected for all $1 \le j \le |\mathcal{D}^{\mathcal{I}}|.$ This means that $\mathpzc{AG}^V$ is connected if and only if $|\mathcal{D}^{\mathcal{I}}| = 1$, a condition independent of $\mathcal{V}$ and equivalent to the prime condition on $\mathpzc{I}$.
\end{proof}

We define an adjacency grouping using the notion of an adjacency graph.
\begin{definition}
Let $\mathcal{V}$ be a maximal weakly separated collection over a Grassmann necklace $\mathcal{I}$. Let $\mathcal{W}$ be a weakly separated collection contained in $\mathcal{V}$ such that $|\mathcal{W} \cap \mathcal{I}| = 0$. Then, we define the $\mathbf{adjacency}$ $\mathbf{grouping}$ ${\mathpzc{AC}}_{\mathcal{V}}(W)$ to be the partition of $W$ generated by the connected components of $\mathpzc{AG}^V(W)$.
\end{definition}
Suppose that ${\mathpzc{AC}}_{\mathcal{V}}(W)$ is a set of $i$ weakly separated collections. We denote these collections by ${\mathpzc{AC}}_{\mathcal{V}}(\mathcal{W})(j)$ for $1 \le j \le i$.
\begin{ex}
We define $\mathcal{V}$ and $\mathcal{W}$ as in Example~$\ref{agextext}.$ Then we have the following:
\begin{align*}
{\mathpzc{AC}}_{\mathcal{V}}(\mathcal{W})(1) &= \left\{\left\{2, 3, 4, 6\right\}, \left\{2, 4, 5, 6\right\}, \left\{2, 3, 4, 7\right\}\right\} \\
{\mathpzc{AC}}_{\mathcal{V}}(\mathcal{W})(2) &= \left\{\left\{2, 5, 6, 7\right\},\left\{1, 5, 6, 7\right\} \right\}
\end{align*}
\end{ex}

We define the notion of a weakly separated collection being connected in a maximal weakly separated collection.
\begin{definition}
Let $\mathcal{V}$ be a maximal weakly separated collection over a Grassmann necklace $\mathcal{I}$. Let $\mathcal{W}$ be a weakly separated collection contained in $\mathcal{V}$ such that $|\mathcal{W} \cap \mathcal{I}| = 0$. Suppose that ${\mathpzc{AC}}_{\mathcal{V}}(\mathcal{W})$ has only one element. Then, we call the weakly separated collection $\mathcal{W}$ $\mathbf{connected}$ $\mathbf{in}$ $\mathcal{V}$. Suppose also that $\mathpzc{G}^{\mathcal{I}}(\mathcal{V} \setminus \mathcal{W})$ is applicable. Then, we call the weakly separated collection $\mathcal{W}$ $\mathbf{very-connected}$ $\mathbf{in}$ $\mathcal{V}$.
\end{definition}
Let $\mathcal{W}$ be a weakly separated collection over $\mathcal{I}$ and let $\mathcal{V}_1, \mathcal{V}_2$ be in $\mathpzc{G}^{\mathcal{I}}(\mathcal{W})$. Notice that if $\mathcal{V}_1 \setminus \mathcal{W}$ is very-connected in $\mathcal{V}_1$, then $\mathcal{V}_2 \setminus \mathcal{W}$ must be very-connected in $\mathcal{V}_2$. We say that $\mathcal{W}$ is \textbf{reverse-very-connected} over $\mathcal{I}$ if $\mathcal{V} \setminus {W}$ is very-connected for any $\mathcal{V} \in \mathpzc{G}^{\mathcal{I}}(\mathcal{W})$.
 
We define an adjacency collection:
\begin{definition}
Let $\mathcal{V}$ be a maximal weakly separated collection over a Grassmann necklace $\mathcal{I}$. Let $V_1$ be a set in $\mathcal{V}$.
We call the \textbf{adjacency collection} ${\mathcal{A}}^{\mathcal{V}}(V_1)$ be the collection of all elements in $\mathcal{V}$ adjacent to $V_1$.
\end{definition}
We use the notion of an adjacency collection to define an adjacency cluster:
\begin{definition}
Let $V$ be a maximal weakly separated collection over a Grassmann necklace $\mathcal{I}$. Let $\mathcal{W}$ be a weakly separated collection contained in $\mathcal{V}$ such that $|\mathcal{W} \cap \mathcal{I}| = 0$. Then, we define the $\mathbf{adjacency}$ $\mathbf{cluster}$ of $\mathcal{W}$ to be
\[\mathcal{A}^{\mathcal{V}} (\mathcal{W}) = \mathcal{W} \cup_{W \in \mathcal{W}} \mathcal{A}^{\mathcal{V}}(W). \]
\end{definition}
\begin{ex}
\label{adc}
We define $\mathcal{V}$ as in Example~$\ref{mainexth}$. We consider the following two weakly separated collections contained in $\mathcal{V}$: 
\begin{align*}
\mathcal{W}_1 &= \left\{\left\{1, 5, 6, 7 \right\}, \left\{1, 6, 7, 9 \right\}, \left\{1, 2, 6, 9 \right\}, \left\{1, 4, 6, 7 \right\} \right\} \\
\mathcal{W}_2 &= \left\{\left\{1, 5, 6, 7 \right\}, \left\{1, 6, 7, 9 \right\}, \left\{1, 2, 6, 9 \right\}, \left\{1, 4, 6, 7 \right\}, \left\{1, 2, 6, 7 \right\} \right\}. \\
\end{align*} 
Then $\mathcal{A}^{\mathcal{V}}(\mathcal{W}_1)$ is equivalent to the weakly separated collection $\mathcal{W}$ in Example~$\ref{intred2}$ (with repeated subsets deleted so that $\mathcal{A}^{\mathcal{V}}(\mathcal{W}_1)$ contains at most one of each set) and $\mathcal{A}^{\mathcal{V}}(\mathcal{W}_2)$ is equivalent to the weakly separated collection $\mathcal{W}$ in Example~$\ref{intred1}$. 
\end{ex}

\subsection{Isomorphism in the case of Very-Connected Weakly Separated Collections}
We consider Theorem~$\ref{link}$ in the case where $\mathcal{C}$ is reverse-very-connected over $\mathcal{I}$. We start with the following lemma:
\begin{lemma}
\label{adjcluqua}
Let $\mathcal{V}$ be a maximal weakly separated collection over a Grassmann necklace $\mathcal{I}$. Let $\mathcal{W}$ be a very-connected weakly separated collection in $\mathcal{V}$. Then there exists a unique interior-reduced plabic tiling consisting exactly of the sets in $\mathcal{A}^{\mathcal{V}}(\mathcal{W})$ (though there might be repeated sets on the boundary curve) and with interior sets exactly equal to $\mathcal{W}$. In fact, this interior-reduced plabic tiling must be $\mathpzc{PT}_{\mathcal{V}}(\mathcal{A}^{\mathcal{V}}(\mathcal{W})).$
\end{lemma}

We use the following definition in the proof of Lemma~$\ref{adjcluqua}$.
\begin{definition}
Let $\mathcal{V}$ be a maximal weakly separated collection, and let $\mathcal{W} \subset \mathcal{V}$ be very-connected in $\mathcal{V}$. Consider the adjacency graph $\mathpzc{AG}^{\mathcal{V}}(\mathcal{W})$. A \textbf{reduction} $\mathcal{R}(\mathcal{W})$ is a connected weakly separated collection of $\mathcal{V}$ such that $\mathcal{R}(\mathcal{W}) \subset \mathcal{W}$ and $|\mathcal{R}(\mathcal{W})| = |\mathcal{W}| - 1.$
\end{definition}
\begin{remark}
Notice that a reduction always exists.
\end{remark}
\begin{figure}
\caption{}
\label{nicex}
\includegraphics[scale=0.7]{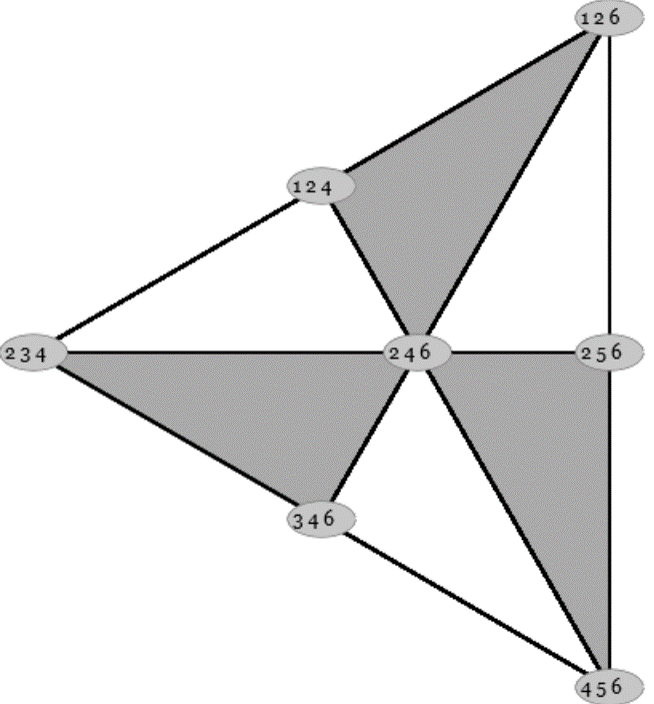}
\end{figure}

We use the notion of a reduction to prove Lemma~$\ref{adjcluqua}$.
\begin{proof}[Proof of Lemma~$\ref{adjcluqua}$]
The uniqueness of such an interior-reduced plabic tiling is easy to see. To prove that $\mathpzc{PT}_{\mathcal{V}}(\mathcal{A}^{\mathcal{V}}(\mathcal{W}))$ is an interior-reduced plabic tiling and has interior sets equal to $\mathcal{W}$, we induct on the size of $|\mathcal{W}|$. The base case is $|\mathcal{W}| = 1$. Notice that $\mathpzc{PT}_{\mathcal{V}}(\mathcal{A}^{\mathcal{V}}(\mathcal{W}))$ looks like a polygon with $2a$ boundary sets and one set in the interior for some $a \ge 2$.
This is an interior-reduced plabic tiling over a relabeled Grassmann necklace with associated permutation in the equivalence class of the permutation below:
$\pi(i) = i+2$ for even $i$ and $\pi(i) = i-2$ for odd $i$, shifted mod $a$ as necessary.

We consider an example.
\begin{ex}
Let $a=3$. Then Figure~$\ref{nicex}$ shows the plabic tiling of the relabeled Grassmann necklace (with associated decorated permutation $\pi = 365412.$
\end{ex} 

For all $\mathcal{A}^{\mathcal{V}}(\mathcal{W})$ such that $|\mathcal{W}| = i$ and $\mathcal{W}$ is any very-connected weakly separated collection in a maximal weakly separated collection $V$, we assume that $\mathpzc{PT}_{\mathcal{V}}(\mathcal{A}^{\mathcal{V}}(\mathcal{W}))$ is an interior-reduced plabic tiling and has interior sets equal to $\mathcal{W}$. We will show that for every such $\mathcal{A}^{\mathcal{V}}(\mathcal{W})$ for $|\mathcal{W}| = i+1$, we have that $\mathpzc{PT}_{\mathcal{V}}(\mathcal{A}^{\mathcal{V}}(\mathcal{W}))$ is an interior-reduced plabic tiling and has interior sets equal to $\mathcal{W}$. By the induction hypothesis, we know that $\mathpzc{PT}_{\mathcal{V}}(\mathcal{A}^{\mathcal{V}}(\mathcal{R}(\mathcal{W})))$ is an interior-reduced plabic tiling and has interior sets equal to $\mathcal{R}(\mathcal{W})$. We know that the set $W = \mathcal{W} \setminus \mathcal{R}(\mathcal{W})$ must be one of the sets on the boundary curve of $\mathpzc{PT}_{\mathcal{V}}(\mathcal{A}^{\mathcal{V}}(\mathcal{R}(\mathcal{W})))$. Notice that
\[\mathcal{A}^{\mathcal{V}}(\mathcal{W}) = \mathcal{A}^{\mathcal{V}}(\mathcal{R}(\mathcal{W})) \cup \mathcal{A}^{\mathcal{V}}(W). \]
We add the sets in $\mathcal{A}^{\mathcal{V}}(W)$ to $\mathpzc{PT}_{\mathcal{V}}(\mathcal{A}^{\mathcal{V}}(\mathcal{R}(W)))$ and add the set $V$ to the interior of $\mathpzc{PT}_{\mathcal{V}}(\mathcal{A}^{\mathcal{V}}(\mathcal{R}(W)))$ to obtain $\mathpzc{PT}_{\mathcal{V}}(\mathcal{A}^{\mathcal{V}}(\mathcal{W}))$. 

\begin{figure}
\caption{}
\label{proof1}
\includegraphics[scale=0.7]{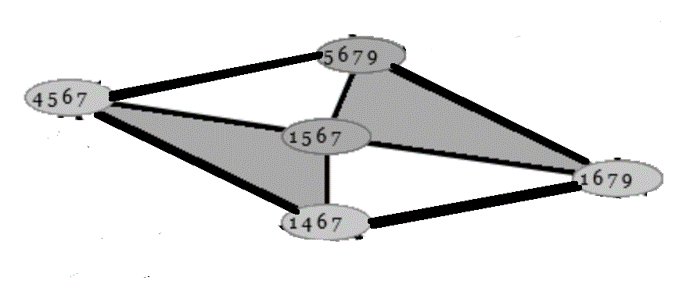}
\end{figure}
\begin{figure}
\caption{}
\label{proof2}
\includegraphics[scale=0.7]{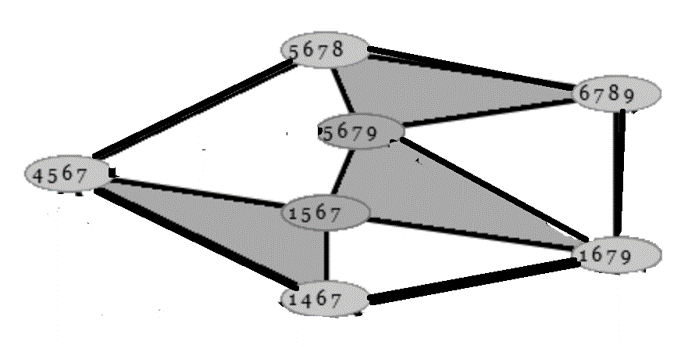}
\end{figure}
We consider two examples of this. 
\begin{ex}
Let $\mathcal{V}$ be defined as in Example~$\ref{adc}$. The plabic tiling of $\mathcal{V}$ is shown in Figure~$\ref{appl}$. Consider $\mathcal{R}(\mathcal{W}) = \left\{\left\{1, 5, 6, 7\right\}\right\}$ and $V = \left\{5, 6, 7, 9\right\}$. Notice that $$\mathcal{A}^{\mathcal{V}}(\mathcal{R}(\mathcal{W}) = \left\{\left\{1, 5, 6, 7\right\}, \left\{1, 5, 6, 7\right\}, \left\{1, 5, 6, 7\right\}, \left\{1, 5, 6, 7\right\}, \left\{1, 5, 6, 7\right\} \right\}.$$ Figure~$\ref{proof1}$ shows the interior-reduced plabic tiling $\mathpzc{PT}_{\mathcal{V}}(\mathcal{A}^{\mathcal{V}}(\mathcal{R}(\mathcal{W})))$. Notice that
$$\mathcal{A}^{\mathcal{V}}(V) = \left\{\left\{5, 6, 7, 8\right\}, \left\{1, 5, 6, 7\right\}, \left\{6, 7, 8, 9\right\}, \left\{1, 6, 7, 9\right\}, \left\{1, 5, 6, 7\right\} \right\}.$$ Adding the sets in $\mathcal{A}^{\mathcal{V}}(V)$ results in the desired interior-reduced plabic tiling $\mathpzc{PT}_{\mathcal{V}}(\mathcal{A}^{\mathcal{V}}(\mathcal{W}))$ as shown in Figure~$\ref{proof2}$.
\end{ex}

\begin{ex}
Let $\mathcal{V}$, $\mathcal{W}_1$, $\mathcal{W}_2$, $\mathcal{A}^{\mathcal{V}}(\mathcal{W}_1)$, and $\mathcal{A}^{\mathcal{V}}(\mathcal{W}_2)$ be defined as in Example~$\ref{adc}$. It is easy to see that $\mathpzc{PT}_{\mathcal{V}}(\mathcal{A}^{\mathcal{V}}(\mathcal{W}_1))$ (shown in Example~$\ref{intred2}$) is an interior-reduced plabic tiling and satisfies the desired properties. Notice that $\mathcal{W}_1 = \mathcal{R}(\mathcal{W}_2)$.  the associated interior-reduced plabic tiling is shown there. Notice that $\mathcal{W}_2 \setminus \mathcal{W}_1 = \left\{\left\{1, 2, 6, 7 \right\}\right\}.$ The adjacency cluster of $\mathcal{A}^{\mathcal{V}}(\left\{1, 2, 6, 7\right\})$ contains $\left\{1, 4, 6, 7 \right\}$, $\left\{1, 6, 7,9 \right\}$, $\left\{ 1, 2, 6, 9\right\}$, and $\left\{1, 2, 4, 6 \right\}$. All of these sets are contained in $\mathcal{A}^{\mathcal{V}}(\left\{1, 2, 6, 7\right\})$. Thus the set $\left\{1, 2, 6, 7 \right\}$ simply needs to be added to the interior. It is easy to see that $\mathpzc{PT}_{\mathcal{V}}(\mathcal{A}^{\mathcal{V}}(\mathcal{W}_2))$ (shown in Example~$\ref{intred1}$) is an interior-reduced plabic tiling and satisfies the desired properties.  
\end{ex}

The applicable constraint ensures that adding the sets of $\mathcal{A}^{\mathcal{V}}(V)$ to $\mathpzc{PT}_{\mathcal{V}}(\mathcal{A}^{\mathcal{V}}(\mathcal{R}(\mathcal{W})))$ results in a valid interior-reduced plabic tiling. Specifically, the boundary of $\mathcal{A}^{\mathcal{V}}(\mathcal{W})$ continues to be a prong-closed curve with no disconnected sets in the interior after the addition of the sets in $\mathcal{A}^{\mathcal{V}}(V)$ (see Remark~$\ref{appcondreq}$ for discussion). The fact that only sets adjacent to $V$ were added ensures that condition (2) of Definition~$\ref{modgras}$ continues to be satisfied. It is thus clear that the resulting process results in the interior-reduced plabic tiling $\mathpzc{PT}_{\mathcal{V}}(\mathcal{A}^{\mathcal{V}}(\mathcal{W}))$ as desired. 
\end{proof}
\begin{figure}
\caption{}
\label{almostbad}
\includegraphics[scale=0.5]{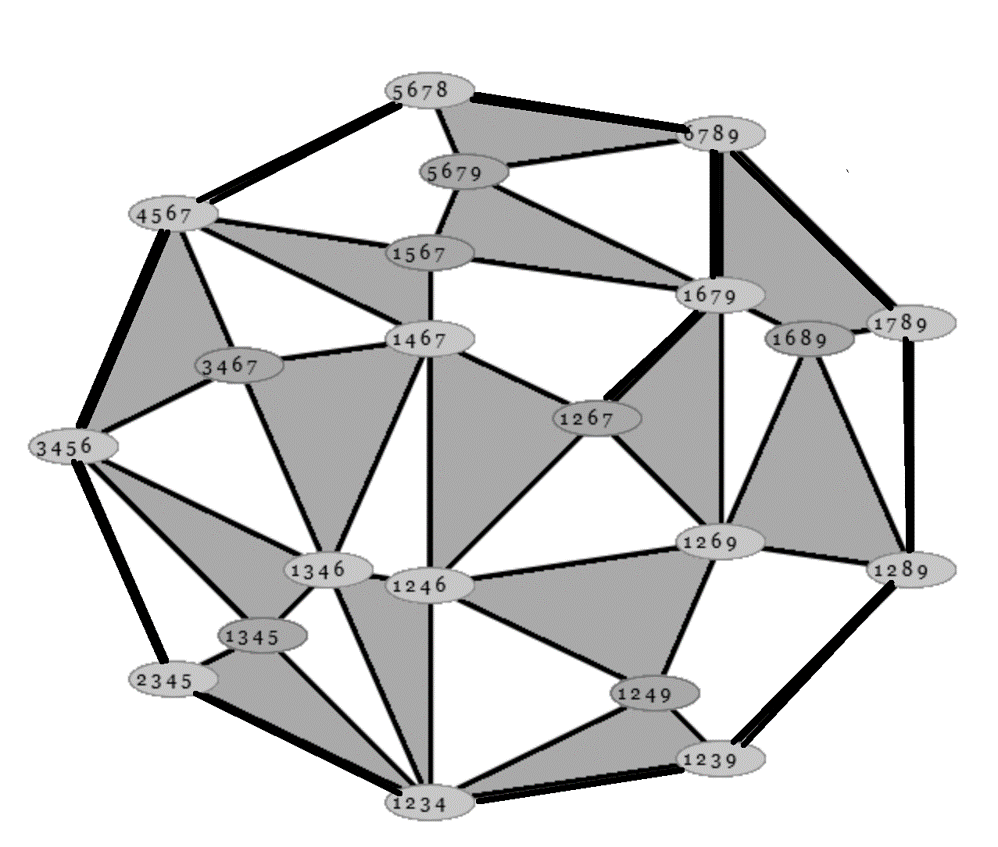}
\end{figure}
\begin{remark}
\label{appcondreq}
Roughly speaking, the applicable constraint prevents the addition of the faces in $\mathcal{A}^{\mathcal{V}}(V)$ from ``entrapping" any additional subsets on the boundary curve of $\mathpzc{PT}_{\mathcal{V}} (\mathcal{A}^{\mathcal{V}}(\mathcal{R}(\mathcal{W}))$. That is, no subsets on the boundary curve of $\mathpzc{PT}_{\mathcal{V}} (\mathcal{A}^{\mathcal{V}}(\mathcal{R}(\mathcal{W}))$ (besides $V$) will become disconnected from the rest of the boundary in the plabic tiling of $\mathcal{A}^{\mathcal{V}}(\mathcal{W})$.
\end{remark}
This is illustrated in the following example:
\begin{ex}
Consider $\mathcal{I}$, $\mathcal{V}$, $\mathcal{C}_1$, and $\mathcal{C}_2$ as in Example~$\ref{appandnot}$.
Consider the following weakly separated collections:
\begin{align*}
\mathcal{W}_1 &= \mathcal{V} \setminus \mathcal{C}_1 \\
\mathcal{W}_2 &= \mathcal{V} \setminus \mathcal{C}_2 \\
\end{align*}
Notice that $\mathcal{W}_2 = \mathcal{R}(\mathcal{W}_1)$. Figure~$\ref{almostbad}$ shows that $\mathcal{A}^{\mathcal{V}}(\mathcal{W}_2)$ is an interior-reduced plabic graph with boundary that is a prong-closed curve $\zeta$ as expected. Notice that the addition of $\mathcal{A}^{\mathcal{V}}(\left\{1, 6, 7, 9\right\})$ is not possible, because it would disconnect the set $\left\{1, 2, 6, 7 \right\}$ from the rest of the boundary. As a result, the boundary would no longer be a prong-closed curve as required.
\end{ex}
 
For a reverse-very-connected weakly separated collection $\mathcal{W}$ over a Grassmann necklace $\mathcal{I}$, we define the following:
\[\mathscr{PT}^{*}(\mathpzc{G}^{\mathcal{I}}(\mathcal{W})) = \left\{\mathpzc{PT}_{\mathcal{V}}^{*}(\mathcal{A}^{\mathcal{V}}(\mathcal{V} \setminus \mathcal{W})) \mid \mathcal{V} \in \mathpzc{G}^{\mathcal{I}}(\mathcal{W}) \right\}\]
In the case $\mathcal{W} = \mathcal{I}$, we call this set $\mathscr{PT}^{*}(\mathpzc{G}^{\mathcal{I}})$.

\begin{lemma}
\label{isokey}
Consider a reverse-very-connected weakly separated collection $\mathcal{C}$ over a Grassmann necklace $\mathcal{I}$. Let $\mathcal{V}$ be in $\mathpzc{G}^{\mathcal{I}}(\mathcal{C})$. Let $\mathcal{J}$ be the relabeled Grassmann necklace of the interior-reduced plabic graph of $\mathcal{A}^{\mathcal{V}}(\mathcal{V} \setminus \mathcal{C})$. Then, we have the following:
\[
\mathscr{PT}^{*}(\mathpzc{G}^{\mathcal{I}}(\mathcal{C})) \cong \mathscr{PT}^{*}(\mathpzc{G}^{\mathcal{J}}),
\]
\[\mathpzc{G}^{\mathcal{I}}(\mathcal{C}) \cong \mathpzc{G}^{\mathcal{J}}.\]
If $\mathpzc{G}^{\mathcal{I}}(\mathcal{C})$ is mutation-friendly, then $\mathpzc{G}^{\mathcal{J}}$ is mutation-friendly.
\end{lemma}
\begin{proof}
We first prove the first equation. By Lemma~$\ref{adjcluqua}$, we know that $\mathcal{A}^{\mathcal{\mathcal{V}}}(\mathcal{V} \setminus \mathcal{C})$ forms an interior-reduced plabic graph with interior sets of $\mathpzc{PT}_{\mathcal{V}}(\mathcal{A}^{\mathcal{V}}(\mathcal{V} \setminus \mathcal{C}))$ exactly equal to the sets in $\mathcal{V} \setminus \mathcal{C}$. Notice that $$\left\{\mathpzc{PT}^{*}_{\mathcal{V}}(\mathcal{A}^{\mathcal{V}}(\mathcal{V} \setminus \mathcal{C})) \mid \mathcal{V} \in \mathpzc{G}^{\mathcal{I}}(\mathcal{C}) \right\}$$
is equal to $\mathscr{PT}^{*}(\mathpzc{G}^{\mathcal{I}}(\mathcal{C}))$ and thus congruent to $\mathscr{PT}^{*}(\mathpzc{G}^{\mathcal{J}})$ by Lemma~$\ref{quasiexch}$.

The second equation follows from the fact that square moves have the same effect on corresponding plabic tilings. This also implies that the mutation-friendly condition holds.
\end{proof}

\subsection{Proof of Theorem~$\ref{link}$}
We first prove the following lemma:
\begin{lemma}
Given an applicable $\mathcal{C}$-constant graph $\mathpzc{G}^{\mathcal{I}}(\mathcal{C})$ and maximal weakly separated collection $\mathcal{V} \in \mathpzc{G}^{\mathcal{I}}(\mathcal{C})$, we have
\label{diir}
$$\displaystyle \mathpzc{G}^{\mathcal{I}}(\mathcal{C}) \cong \square_{\mathcal{W} \in {\mathpzc{AC}}_{\mathcal{V}}(\mathcal{V} \setminus C)}\mathpzc{G}^{\mathcal{I}}(\mathcal{V} \setminus \mathcal{W}).$$
If $\mathpzc{G}^{\mathcal{I}}(\mathcal{C})$ is mutation-friendly, then each $\mathpzc{G}^{\mathcal{I}}(\mathcal{V} \setminus \mathcal{W})$ is mutation-friendly.
\end{lemma} 
\begin{proof}
Let $i$ be $|{\mathpzc{AC}}_{\mathcal{V}}(\mathcal{V} \setminus \mathcal{C})|$. Let $\mathcal{W}_j$ be in $\mathpzc{G}^{\mathcal{I}}(\mathcal{V} \setminus {\mathpzc{AC}}_{\mathcal{V}}(\mathcal{V} \setminus C)(j))$ for $1 \le j \le i$. We define
\[\mathcal{M}(\mathcal{W}_1,\ldots,\mathcal{W}_i) = \mathcal{C} \cup_{j=1}^{i} (\mathcal{W}_j \cap {\mathpzc{AC}}_{\mathcal{V}}(\mathcal{V} \setminus C)(j))\]
Notice that 
\[\left\{\mathcal{M}(\mathcal{W}_1,\ldots,\mathcal{W}_i) \mid\mathcal{W}_j \in \mathpzc{G}^{\mathcal{I}}(\mathcal{V} \setminus {\mathpzc{AC}}_{\mathcal{V}}(\mathcal{V} \setminus C)(i)) \text{ for } 1 \le j \le i  \right\} = \left\{\mathcal{V} \mid \mathcal{V} \in \mathpzc{G}^{\mathcal{I}}(\mathcal{V} \setminus \mathcal{C}) \right\}.\]
Furthermore, the weakly separated collection $\mathcal{M}(\mathcal{V}_1,\ldots,\mathcal{V}_i)$ can be mutated into $\mathcal{M}(\mathcal{W}_1,\ldots,\mathcal{W}_i)$ if and only if $\mathcal{V}_i$ can be mutated into $\mathcal{W}_i$ for some $i$ and $\mathcal{V}_j = \mathcal{W}_j$ for $j \neq i$. This means that $$\displaystyle \mathpzc{G}^{\mathcal{I}}(\mathcal{C}) \cong \square_{\mathcal{W} \in {\mathpzc{AC}}_{\mathcal{V}}(\mathcal{V} \setminus C)}\mathpzc{G}^{\mathcal{I}}(\mathcal{V} \setminus \mathcal{W})$$ and that the mutation-friendly condition holds.
\end{proof}

\begin{proof}[Proof of Theorem~$\ref{link}$]
Consider an applicable $\mathcal{C}$-constant graph $\mathpzc{G}^{\mathcal{I}}(\mathcal{C})$ and a maximal weakly separated collection $\mathcal{V} \in \mathpzc{G}^{\mathcal{I}}(\mathcal{C})$. We prove that $\mathpzc{G}^{\mathcal{I}}(\mathcal{C})$ is isomorphic to an exchange graph with interior size $|\mathcal{V} \setminus \mathcal{C}|.$ By Lemma~$\ref{diir}$, we know that $\mathpzc{G}^{\mathcal{I}}(\mathcal{C})$ is isomorphic to the direct product of the $\mathcal{C}$-constant graphs $\mathpzc{G}^{\mathcal{I}}(\mathcal{V} \setminus \mathcal{W})$ for $\mathcal{W} \in {\mathpzc{AC}}_{\mathcal{V}}(\mathcal{V} \setminus C)$ (where the mutation-friendly condition holds). Notice that ${\mathpzc{AC}}_{\mathcal{V}}(\mathcal{V} \setminus C)(i)$ is a very-connected weakly separated collection in $\mathcal{V}$ for $1 \le j \le |\mathcal{AC}_{\mathcal{V}}(\mathcal{V} \setminus \mathcal{C})$. By Lemma~$\ref{isokey}$, we know that each $\mathpzc{G}^{\mathcal{I}}(\mathcal{V} \setminus \mathcal{W})$ is isomorphic to an exchange graph with interior size $|\mathcal{W}|$ (where the mutation-friendly condition holds). Let this exchange graph be $\mathpzc{G}_{\mathcal{W}}$. Then, we know that
$$\displaystyle \mathpzc{G}^{\mathcal{I}}(\mathcal{C}) \cong \square_{\mathcal{W}\in {\mathpzc{AC}}_{\mathcal{V}}(\mathcal{V} \setminus C)} \mathpzc{G}_{\mathcal{W}}.$$
By Lemma~$\ref{dirproduct}$, we know that $\square_{W \in {\mathpzc{AC}}_{\mathcal{V}}(\mathcal{V} \setminus \mathcal{C})} \mathpzc{G}_{W}$ is isomorphic to an exchange graph with interior size
$$\displaystyle \sum_{\mathcal{W} \in {\mathpzc{AC}}_{\mathcal{V}}(\mathcal{V} \setminus C)} i(\mathpzc{G}_{\mathcal{W}}) = \sum_{\mathcal{W} \in {\mathpzc{AC}}_{\mathcal{V}}(\mathcal{V} \setminus \mathcal{C})} |\mathcal{W}| = |\mathcal{V} \setminus \mathcal{C}|$$ where the mutation-friendly condition also holds.
\end{proof}
\subsection{Proof of Corollary~$\ref{stronglink}$}
We show show Corollary~$\ref{stronglink}$ follows from Theorem~$\ref{link}$. In order to do so, we first prove that the following for $\mathcal{C}$-constant graphs with small co-dimension:
\begin{lemma}
\label{app}
Every $\mathcal{C}$-constant graph of co-dimension $c <4$ is applicable.
\end{lemma}
\begin{proof}
Consider a $\mathcal{C}$-constant graph $\mathpzc{G}^{\mathcal{I}}(\mathcal{C})$ that is not applicable. This means there exists $V \in \mathcal{C}$ such that no set in $\mathcal{C}$ or $\mathcal{I}$ is quasi-adjacent to $V$. First, this means that $V$ is not in $\mathcal{I}$. Now, consider a maximal weakly separated collection $\mathcal{V} \in \mathpzc{G}^{\mathcal{I}}$. Notice that there must be at least $4$ sets in $\mathcal{V} \setminus \mathcal{C}$ that are quasi-adjacent to $V$. Hence, we know that the co-dimension of $\mathpzc{G}^{\mathcal{I}}(\mathcal{C}) \ge 4$.
\end{proof}
\begin{proof}[Proof of Corollary~$\ref{stronglink}$]
This follows from Theorem~$\ref{link}$ and Lemma~$\ref{app}.$
\end{proof}

\section{Proof of Theorem~$\ref{mainchare}$ and Theorem~$\ref{maincharc}$}
We now characterize exchange graphs with interior size $0, 1, 2, 3, 4$ and their associated decorated permutations. We also provide a characterization of all $\mathcal{C}$-constant graphs of co-dimension $0, 1, 2, 3$. In Section 5.1, we prove properties of mutation-friendly Grassmann necklaces that we require to prove Theorem~$\ref{mainchare}$ and Theorem~$\ref{maincharc}$. In Section 5.2, we prove Theorem~$\ref{mainchare}$. In Section 5.3, we prove Theorem~$\ref{maincharc}$.

\subsection{Mutation-Friendly Grassmann Necklaces}
We first prove the following fact of $\mathcal{C}$-constant graphs of very-mutation-friendly graphs which we use to prove our later properties. 
\begin{prop}
\label{vmfproperty}
Let $\mathcal{I}$ be a very-mutation-friendly Grassmann necklace with interior size $i$. Consider the weakly separated collections
\[\mathcal{W}_1 \supset \mathcal{W}_2 \supset \ldots \supset \mathcal{W}_{i-1} \supset W_{i} = \mathcal{I}  \]
using the notation in Definition~$\ref{vmf}$ applied to $\mathcal{I}$ and adding the weakly separated collection $\mathcal{W}_i$. Let $\mathcal{V}$ be a maximal weakly separated collection in $\mathpzc{G}^{\mathcal{I}}(\mathcal{W}_1)$. Let $j$ be an integer such that $1 \le j \le i$. Let $\mathcal{W}$ be a weakly separated collection in ${\mathpzc{AC}}_{\mathcal{V}} (\mathcal{V} \setminus W_j)$. Suppose that $\mathpzc{PT}_{\mathcal{V}}(\mathcal{A}^{\mathcal{V}}(\mathcal{W}))$ has boundary curve $\mathcal{J}^{*}$ with relabeled Grassmann necklace $\mathcal{J}$. Then $\mathcal{J}$ is very-mutation-friendly. 
\end{prop}
\begin{proof}
 We split the proof into two cases: $\zeta^{P}_{\mathpzc{J}^{*}} = \emptyset$ and $\zeta^{P}_{\mathpzc{J}^{*}} \neq \emptyset$.

\textbf{Case 1:} $\zeta^{P}_{\mathpzc{J}^{*}} = \emptyset$.
Consider the weakly separated collections 
\begin{equation}
\label{setone}
(f(\mathcal{W}_1 \cap \mathcal{V}^{\mathcal{J}^{*}}), f(\mathcal{W}_2 \cap \mathcal{V}^{\mathcal{J}^{*}}), \ldots, f(\mathcal{W}_{j} \cap \mathcal{V}^{\mathcal{J}^{*}})).
\end{equation}
First, notice that $f(\mathcal{W}_{j} \cap \mathcal{V}^{\mathcal{J}^{*}}) = \mathcal{J}$. For $1 \le l \le j$, we want to show that $\mathpzc{G}^{\mathcal{J}}(f(\mathcal{W}_{l} \cap \mathcal{V}^{\mathcal{J}^{*}}))$ is mutation-friendly and applicable. This follows directly from the fact that $\mathpzc{G}^{\mathcal{I}}(\mathcal{W}_l)$ is applicable and mutation-friendly. We know that deleting repeated weakly separated collections in $(\ref{setone})$ gives us the desired requirement for Definition~$\ref{vmf}$. 

\textbf{Case 2:} $\zeta^{P}_{\mathpzc{J}^{*}} \neq \emptyset$.
Let $V_1$ be the set at the point $P$ in $\zeta^{P}_{\mathcal{J}^{*}} $ using the notation in Definition~$\ref{prongclosed}.$ For $1 \le l \le j$, let $\mathcal{X}_l$ be the weakly separated multi-collection consisting of $\mathcal{W}_l$ and an additional copy of the set $\left\{V_1 \right\}$. Consider the resulting weakly separated collections
\begin{equation}
\label{settwo}
f(\mathcal{X}_1 \cap \mathcal{V}^{\mathcal{J}^{*}}), f(\mathcal{X}_2 \cap \mathcal{V}^{\mathcal{J}^{*}}), \ldots, f(\mathcal{X}_j \cap \mathcal{V}^{\mathcal{J}^{*}})).
\end{equation}
First, notice that $f(\mathcal{X}_{j} \cap \mathcal{V}^{\mathcal{J}^{*}}) = \mathcal{J}$. For $1 \le l \le j$, we want to show that $\mathpzc{G}^{\mathcal{J}}(f(\mathcal{X}_{l} \cap \mathcal{V}^{\mathcal{J}^{*}}))$ is mutation-friendly and applicable.  This follows directly from the fact that $\mathpzc{G}^{\mathcal{I}}(\mathcal{W}_l)$ is applicable and mutation-friendly. We know that deleting repeated weakly separated collections in $(\ref{settwo})$ gives us the desired requirement for Definition~$\ref{vmf}$. 
\end{proof}
We prove the following bound on the number of sets in a very-mutation-friendly Grassmann necklaces as a function with interior size:
\begin{lemma}
\label{nnn}
Let $\mathcal{I}$ be a very-mutation-friendly, prime Grassmann necklace. If $i(\mathcal{I}) > 0$, then $|\mathcal{I}| \le 2 \cdot i(\mathcal{I}) +2$.
\end{lemma}
\begin{remark}
This lemma gives a necessary but not sufficient bound on the size of a very-mutation-friendly Grassmann necklace.
\end{remark}
\begin{ex}
Consider the permutation $\pi = 38762145$. Let $\mathpzc{I}$ be the Grassmann necklace which corresponds to $\pi$. It can verified that $\mathcal{I}$ is prime. Notice that $i(\mathcal{I}) = 3$ and $|\mathcal{I}| = 8$. This means that $\mathcal{I}$ satisfies the inequality in Lemma~$\ref{nnn}$. However, every maximal weakly separated collection in $\mathpzc{G}^{\mathcal{I}}$ contains the set $\left\{2, 4, 5, 8 \right\}$ which is not in $\mathcal{I}$. Thus, $\mathcal{I}$ is not mutation-friendly.
\end{ex}

\begin{proof}[Proof of Lemma~$\ref{nnn}$]
Let $i = i(\mathcal{I})$. We proceed by induction on $i$.

For the base case, we consider very-mutation-friendly Grassmann necklaces $\mathcal{I}$ such that $i=1$. Notice that $\mathcal{I}$ must be part of the equivalence class of $\left\{\{1,2\}, \{2, 3\}, \{3, 4\}, \{4, 1\} \right\}.$ This means that $|\mathcal{I}| = 4 \le 2 \cdot 1 +2$ as desired.

Assume that for all very-mutation-friendly, prime Grassmann necklaces $\mathcal{J}$ with interior size $j$ such that $j \le i$, we have that $|\mathcal{J}| \le 2 \cdot i(\mathcal{J}) +2$. Now, consider a very-mutation-friendly, prime Grassmann necklace $\mathcal{I}$ with interior size $i+1$. Consider the weakly separated collections: 
\[\mathcal{W}_1 \supset \mathcal{W}_2 \supset \ldots \supset \mathcal{W}_{i} \supset \mathcal{I} \] using the notation in Definition~$\ref{vmf}$ applied to $\mathcal{I}$.
We define the set $V$ as follows: 
\[\left\{V \right\} = \mathcal{W}_i \setminus \mathcal{I}.\]
By Definition~$\ref{vmf}$, we know that $\mathpzc{G}^{\mathpzc{I}}(\mathcal{W}_i)$ is applicable and mutation-friendly. Consider $\mathcal{V} \in \mathpzc{G}^{I}(\mathcal{W}_1)$ such that $V$ is mutatable in $\mathcal{V}$. Suppose that the adjacency grouping ${\mathpzc{AC}}_{\mathcal{V}}(\mathcal{V} \setminus \mathcal{W}_i)$  is a partition containing $l$ weakly separated collections that we call $(\mathcal{V}_1, \mathcal{V}_2,\ldots, \mathcal{V}_l).$ 
For $1 \le m \le l$, suppose that $\mathpzc{PT}_{\mathcal{V}}(\mathcal{A}^{\mathcal{V}}({\mathcal{W}}_m))$ has boundary curve $\mathcal{I}^{*m}$ with relabeled Grassmann necklace $\mathcal{I}^m$.

By Proposition~$\ref{vmfproperty}$, we have that $\mathcal{I}^m$ is very-mutation-friendly. Thus by the induction hypothesis, we have that 
\begin{equation}
\label{indhyp}
|\mathcal{I}^{m}| \le 2 \cdot i(\mathcal{I}^m)+2.
\end{equation} 
Also, notice that for all $1 \le m \le l$, we have that 
\begin{equation}
\label{inside}
V \in \mathcal{I}^{*j}. 
\end{equation}
We also have that
\begin{equation}
\label{sumfact}
\sum_{m=1}^{l} i(\mathcal{I}^m) = i. 
\end{equation}
For any $1 \le j \le i$, we have that 
\begin{equation}
\label{sense} 
\mathcal{I}^{*m} \subset \mathcal{W}_j \subset \mathcal{I} \cup \left\{V \right\}
\end{equation}
We can obtain $\mathcal{I}$ from the set $V$ and from $\mathcal{I}^{*m}$ and $\mathcal{V}^{\mathcal{I}^{*m}}$ for $1 \le m \le l$ as follows: 
\begin{equation}
\label{union}
\mathcal{I} = \cup_{m=1}^{l} (\mathcal{I}^{*m} \setminus \left\{V \right\}) \cup ((\mathcal{A}^{\mathcal{V}}(V) \setminus \cup_{m=1}^{l} ({\mathcal{V}}^{\mathcal{I}^{*m}} \setminus \mathcal{I}^{*m}))
\end{equation}
where in the case of multi-collections, duplicate copies of all sets are deleted so that $\mathcal{I}$ does not have any repeated sets. 

Consider the curve $\zeta_{\mathcal{I}^{*m}}$ corresponding to $\mathpzc{I}^{*m}$. We split the remainder our proof into two cases: $\zeta^{P}_{\mathpzc{I}^{*m}} \neq \emptyset$ for some $1 \le m \le l$ and $\zeta^{P}_{\mathpzc{I}^{*m}} = \emptyset$ for all $1 \le m \le l$.

\textbf{Case 1:} $\zeta^{P}_{\mathpzc{I}^{*m}} \neq \emptyset$ for some $m$.
By $(\ref{sense})$, since no set in $\mathcal{I}$ can be on $\zeta^{P}_{\mathcal{I}^{*m}} \setminus P$, we know that $V$ must be on $\zeta^{P}_{\mathcal{I}^{*m}} \setminus P$. This means that $V$ is in the interior of the closed curve $\zeta^{F1}_{\mathcal{I}^{*m}} \cup \zeta^{F2}_{\mathcal{I}^{*m}}$. This implies that $l=1$. By the $(\ref{sumfact})$, we have that $i(\mathcal{I}^{1}) = i$. We also know that 
\[\mathcal{A}^{\mathcal{V}}(V) \subset \mathcal{V}^{\mathcal{I}^{*1}}.\]
Hence we have that 
\[\mathcal{A}^{\mathcal{V}}(V) \setminus (\mathcal{V}^{\mathcal{I}^{*1}} \setminus \mathcal{I}^{*1}) \subset \mathcal{I}^{*1}.\]
Applying ($\ref{union}$) and $(\ref{indhyp})$ thus gives us that
\[|\mathcal{I}| \le 2i+2 - 1\le 2i + 1\] 
as desired.

\textbf{Case 2:} $\zeta^{P}_{\mathpzc{I}^{*m}} = \emptyset$ for all $1 \le m \le l$.
Notice that $|\mathcal{A}^{\mathcal{V}}(V) = 4|$ since $V$ is mutatable.
By ($\ref{inside}$) and since $\mathcal{I}^{*m}$ is a modified-like Grassmann necklace, we know that $V$ must be adjacent to one of the sets in $V^{\mathcal{I}^{*m}} \setminus \mathcal{I}^{*m}$ for each $1 \le m \le l$. Thus $l \le 4$. We have that  
\[|\mathcal{A}^{\mathcal{V}}(V) \setminus \cup_{m=1}^{l} (V^{\mathcal{I}^{*1}} \setminus \mathcal{I}^{*})| \le 4-l.\] By ($\ref{union}$), we obtain the following
\[\mathcal{I} \le \sum_{m=1}^{l} |\mathcal{I}^{*1}| - l + 4 - l.\]
By $(\ref{indhyp})$, we have that 
\[\mathcal{I} \le \sum_{m=1}^{l} 2\cdot i(\mathcal{I}) + l + 4 - l.\]
By $(\ref{sumfact})$, we have that 
\[\mathcal{I} \le 2i + 4 = 2(i+1) + 2\]
as desired.
\end{proof}
We prove the following lemma involving the decomposition set of a very-mutation-friendly Grassmann necklace.
\begin{lemma}
\label{vmfdirproduct}
Given a very-mutation-friendly Grassmann necklace $\mathcal{I}$, the relabeled Grassmann necklace of every Grassmann necklace in the decomposition set $\mathcal{D}^{\mathcal{I}}$ is very-mutation-friendly.
\end{lemma}
\begin{proof}
We use the notation from Proposition~$\ref{vmfproperty}.$ Let $\mathcal{V}$ be a maximal weakly separated collection in $\mathpzc{G}^{\mathcal{I}}(\mathcal{W}_1)$. Notice that
\[\left\{\mathcal{V}^{\mathcal{D*}^{\mathcal{I}}(i)} \mid  1 \le l \le |\mathcal{D*}^{\mathcal{I}}|\right\} = \left\{\mathcal{AC}_{\mathcal{V}}(\mathcal{V} \setminus \mathcal{W}_i)(l) \mid 1 \le l \le |\mathcal{AC}_{\mathcal{V}}(\mathcal{V} \setminus \mathcal{W}_i)| \right\}. \] Thus, the desired result follows from Proposition~$\ref{vmfproperty}$.
\end{proof}
We also prove the following lemma involving very-mutation-friendly exchange graphs of small interior size. 
\begin{lemma}
\label{iso}
Any exchange graph with interior size $i \le 4$ is isomorphic to a very-mutation-friendly exchange graph with interior size $j \le i$.
\end{lemma}
We use the following propositions in our proof:
\begin{prop}
\label{mutfri}
For every mutation-friendly $\mathcal{C}$-constant graph $\mathpzc{G}^{\mathcal{I}}(\mathcal{C})$ of co-dimension $c > 1$, there exists a weakly separated collection $\mathcal{D}$ such that $\mathcal{I} \subset \mathcal{D} \subset \mathcal{C}$ and $\mathpzc{G}^{\mathcal{I}}(\mathcal{D})$ is mutation friendly with co-dimension $c-1$.
\end{prop}
\begin{proof}
We obtain $\mathcal{D}$ as follows: Fix a maximal weakly separated collection $\mathcal{V} \in \mathpzc{G}^{\mathcal{I}}(\mathcal{C})$. For each set $V \in \mathcal{V} \setminus \mathcal{C}$, let $m(V)$ be the minimum path length in $\mathpzc{G}^{\mathcal{I}}(\mathcal{C})$ from $\mathcal{V}$ to reach a $\mathcal{C} \in \mathpzc{G}^{\mathcal{I}}(\mathcal{C})$ that does not contain $V$. Consider $V \in \mathcal{V}$ such that $m(V)$ is maximal. Then, let $\mathcal{D} = \mathcal{I} \cup \left\{V \right\}$.
\end{proof}
\begin{prop}
\label{sm}
All mutation-friendly exchange graphs with interior size $\le 4$ are very-mutation-friendly.
\end{prop}
\begin{proof}
Consider an exchange graph $\mathpzc{G}^{\mathcal{I}}$ with interior size $i \le 4$. Notice that $\mathpzc{G}^{\mathcal{I}}$ is isomorphic to the $\mathcal{C}$-constant graph $\mathpzc{G}^{\mathcal{I}}(\mathcal{I})$ of co-dimension $i$. A repeated application of Proposition~$\ref{mutfri}$ results in a chain of mutation-friendly $\mathcal{C}$-constant graphs which we know are applicable by Lemma~$\ref{app}$.
\end{proof}
We now prove Lemma~$\ref{iso}$.
\begin{proof}[Proof of Lemma~$\ref{iso}$]
Suppose that $\mathpzc{G}^{\mathpzc{I}}$ is mutation-friendly with interior size $i \le 4$. Then by Proposition~$\ref{sm}$, we know that $\mathpzc{G}^{\mathpzc{I}}$ is very-mutation-friendly as desired.

Suppose that $\mathpzc{G}^{\mathpzc{I}}$ with interior size $i \le 4$ is not mutation-friendly and $\cup_{\mathcal{V} \in \mathpzc{G}^{\mathpzc{I}}} = \mathcal{C}$. By definition, we know that $|\mathcal{C}| \ge |\mathcal{I}|$ and $\mathpzc{G}^{\mathpzc{I}} \cong \mathpzc{G}^{\mathpzc{I}}(\mathcal{C})$. We also know that $\mathpzc{G}^{\mathpzc{I}}(\mathcal{C})$ is mutation-friendly. By Theorem~$\ref{iso}$, we know that $\mathpzc{G}^{\mathpzc{I}}(\mathcal{C})$ is isomorphic to a mutation-friendly exchange graph $\mathpzc{G}^{\mathpzc{J}}$ with interior size $\le i-1$. By Proposition~$\ref{sm}$, we know that $\mathpzc{G}^{\mathcal{J}}$ is very-mutation-friendly with interior size $\le i-1$ as desired.
\end{proof}
\subsection{Proof of Theorem~$\ref{chare}$}
We can decompose the set of exchange graphs for interior size $i$ for $i \le 4$ into sets of prime, mutation-friendly exchange graphs as follows:
\begin{lemma}
\label{yayy}
For any nonnegative integer $i \le 4$, the set of exchange graphs with interior size $i$ is exactly:
$$\displaystyle \left\{\square_{j=0}^{k} \mathpzc{G}^{\mathcal{I}_j} \mid \sum_{j=0}^{k} i(\mathpzc{G}^{\mathcal{I}_j}) \le i \right\},$$
where each $\mathpzc{G}^{\mathcal{I}_j}$ is a prime, very-mutation-friendly exchange graph.
\end{lemma}
\begin{proof}
Consider an exchange graph $\mathpzc{G}^{\mathcal{I}}$ with interior size $i$ where $i \le 4$.
By Lemma~$\ref{iso}$, we know that $\mathpzc{G}^{\mathcal{I}}$ is isomorphic to a very-mutation-friendly exchange graph $\mathpzc{G}^{\mathcal{J}}$ where $\mathcal{J}$ is a Grassmann necklace with interior size $j \le i$. By Lemma~$\ref{dirproduct}$ and Lemma~$\ref{vmfdirproduct}$, we know that $\mathpzc{G}^{\mathcal{J}}$ is isomorphic to a direct product in the above form.

Conversely, suppose that we have a direct product in the above form. By Lemma~$\ref{dirproduct}$, we know that there exists a Grassmann necklace $\mathcal{I}$ with exchange graph $\mathpzc{G}^{\mathcal{I}}$ with the desired interior size isomorphic to the desired direct product.
\end{proof}

Lemma~$\ref{yayy}$ allows us to prove Theorem~$\ref{chare}$.
\begin{proof}[Proof of Theorem~$\ref{chare}$]
By Theorem~$\ref{yayy}$, the set of exchange graphs with interior size $i$ for $i \le 4$ is reduced to computing all prime mutation-friendly exchange graphs with interior size $\le 4$. Let $\mathcal{I}$ be the Grassmann necklace of any prime, very-mutation-friendly exchange graph with interior size $j \le 4$. If $j > 0$, by Theorem~$\ref{nnn}$, we know that $|\mathcal{I}| \le 2j + 2$. If $j = 0$, it is easy to see that $\mathcal{I}$ corresponds to the decorated permutation $312$. This reduces proving the results in Table~$\ref{mainchare}$ to a finite computation, which we executed using a computer program. The results of this computer program are shown in Table~$\ref{fullchar}$, which shows all equivalence classes of prime, mutation-friendly Grassmann necklaces and their corresponding exchange graphs.
\end{proof}

\subsection{Proof of Theorem~$\ref{charc}$}
We now use Theorem~$\ref{chare}$ to prove Theorem~$\ref{charc}$.
\begin{proof}[Proof of Theorem~$\ref{charc}$]
By Corollary~$\ref{stronglink}$, we know that the set of $\mathcal{C}$-constant graphs of co-dimension $c$ for $0 < c \le 3$ is equivalent to the set of exchange graphs with interior size $c$. In the case that $c =0$, it is easy to see that the only possible $\mathcal{C}$-constant graph is a path with one vertex. The results in Table~$\ref{maincharc}$ follow from Theorem~$\ref{chare}$.
\end{proof}

\section{Explanation of Conjecture~$\ref{cat}$}
Conjecture~$\ref{cat}$ involves the maximum order of an exchange graph of a given interior size. 
We denote by $M(i)$ the maximum possible order of an exchange graph with interior size $i$. Conjecture~$\ref{cat}$ is partially motivated by the following:
\begin{prop}
For any $i \ge 0$, a lower bound on $M(i)$ is the Catalan number $C_{i+1}$ (where $C_1 = 1, C_2 = 2, C_3 = 5$).
\end{prop}
\begin{proof}
We can construct an exchange graph with interior size $\mathcal{C}$ that has order $C_{i+1}$ by taking $\mathcal{I} = (\left\{1, 2 \right\}, \left\{2, 3 \right\},\ldots,\left\{i+2, i+3 \right\}, \left\{i+3, 1 \right\}).$ This correspond to triangulations of a $(i+3)$-gon.
\end{proof}
We now discuss further motivation for Conjecture~$\ref{cat}$:
\begin{remark}
Theorem~$\ref{chare}$ proves Conjecture~$\ref{cat}$ for $i \le 4$. The results in Table~$\ref{fullchar}$ motivates us to hypothesize that the only Grassmann necklaces with interior size $c$ and order $C_{i+1}$ are Grassmann necklaces in the equivalence class of $$\mathcal{I} = (\left\{1, 2 \right\}, \left\{2, 3 \right\},\ldots,\left\{i+2, i+3 \right\}, \left\{i+3, 1 \right\}).$$

In the case of triangulations, the interior size is equal to the number of mutatable subsets at each step. For other Grassmann necklaces, the interior size is greater than the number of mutatable sets at some steps, since some interior sets are immutatable in some maximal weakly separated collections.
\end{remark}
We also note some consequences of Conjecture~$\ref{cat}$:
\begin{remark}
If Conjecture~$\ref{cat}$ is true, the following result would follow. Consider the following Grassmann necklace: $$\displaystyle \mathcal{I} = (\left\{1,\ldots,k\right\}, \left\{2,\ldots,k+1\right\},\ldots,\left\{n,\ldots,k-1\right\}).$$ It would follow that
\[|\mathpzc{G}^{\mathcal{I}}| \le C_{k(n-k)+2-n}.\]
This bound would provide a partial answer to Problem 36 in \cite{suhoh}.
\end{remark}
\section{Proof of Theorem~$\ref{cycle}$ and Theorem~$\ref{tree}$}
Theorem~$\ref{cycle}$ and Theorem~$\ref{tree}$ provide a full characterization of very-mutation-friendly exchange graphs in the special cases of cycles and trees. In Section 7.1, we prove relevant properties of mutation-friendly exchange graphs as well as an important property of Grassmann necklaces in $\mathscr{C}_i$. In Section 7.2, we prove Theorem~$\ref{tree}$. In Section 7.3, we prove Theorem~$\ref{cycle}$.

\subsection{Important Properties of Mutation-Friendly Exchange Graphs and $\mathscr{C}_i$}
In Section 7.1.1, we discuss relevant properties involving cycles and $C$-constant graphs of mutation-friendly exchange graphs. In Section 7.1.2, we discuss an important property that relates the exchange graphs of Grassmann necklaces in $\mathscr{C}_i$ with the exchange graphs of Grassmann necklaces in $\mathscr{C}_{i+1}$.
\subsection{Mutation-Friendly Exchange Graphs}
We prove the following lemma involving the existence of a cycle in mutation-friendly exchange graphs:
\begin{lemma}
\label{cyc4}
Given a mutation-friendly, not prime exchange graph $\mathpzc{G}^{\mathcal{I}}$ with interior size $\ge 2$, we know that there 
exists a $\mathcal{C}$-constant graph $\mathpzc{G}^{\mathcal{I}}(\mathcal{C})$ of co-dimension $2$ that is a cycle with $4$ vertices.
\end{lemma}
\begin{proof}
By Lemma~$\ref{dirproduct}$, we know that a not prime exchange graph with interior size $\ge 2$ is either isomorphic to a prime exchange graph and is not mutation-friendly or has a $\mathcal{C}$-constant graph of co-dimension $2$ that is a cycle with $4$ vertices. The result follows. 
\end{proof}
We also prove the existence of a cycle in certain applicable, mutation-friendly $\mathcal{C}$-constant graphs.
\begin{lemma}
\label{nrvc}
Let $\mathpzc{G}^{\mathpzc{I}}$ be an exchange graph and let $\mathcal{W}$ be not reverse-very-connected over $\mathcal{I}$. If $\mathpzc{G}^{\mathcal{I}}(\mathcal{W})$ is applicable and mutation-friendly, then $\mathpzc{G}^{\mathcal{I}}(\mathcal{W})$ contains a cycle of length $4$.
\end{lemma} 
\begin{proof}
Let $\mathcal{W}$ be not reverse-very-connected over $\mathcal{I}$ and suppose that $\mathpzc{G}^{\mathcal{I}}(\mathcal{W})$ is applicable. Let $\mathcal{V}$ be in $\mathpzc{G}^{\mathcal{I}}(\mathcal{W})$. By definition, we know that 
\[|{\mathpzc{AC}}_{\mathcal{V}}(\mathcal{V} \setminus \mathcal{W})| > 1.\] Since $\mathpzc{G}^{\mathcal{I}}(\mathcal{W})$ is mutation-friendly, the desired result follows from Lemma~$\ref{diir}$. 
\end{proof}
We also prove the following lemma involving the $C$-constant graphs of very-mutation-friendly exchange graphs:
\begin{lemma}
\label{whee}
Given a very-mutation-friendly exchange graph $\mathpzc{G}^{\mathcal{I}}$ with interior size $i$, there exists a weakly separated collection $\mathcal{C}$ over $\mathcal{I}$ such that the $\mathcal{C}$-constant graph $\mathpzc{G}^{\mathcal{I}}(\mathcal{C})$ is mutation-friendly, has co-dimension $i-1$, and is isomorphic to a very-mutation-friendly exchange graph of co-dimension $i-1$.
\end{lemma}
We require the following proposition in the proof: 
\begin{prop}
\label{sillydp}
If the Grassmann necklaces $\mathcal{I}^1$, $\mathcal{I}^2$, \ldots, $\mathcal{I}^i$ are very-mutation-friendly, then $\square_{j=1}^i \mathcal{I}^j$ is very-mutation-friendly.
\end{prop}
\begin{proof}
Let $\mathcal{I} = \square_{j=1}^i \mathcal{I}^j$. WLOG, for each $1 \le j \le i$, assume that $\mathcal{D}^{\mathcal{I}}(j) = \mathcal{I}^j$. Let $f^j$ be the mapping function that takes $\mathcal{D}^{\mathcal{I}}(j)$ to $\mathcal{D*}^{\mathcal{I}}(j)$. 
Suppose $\mathcal{I}^j$ has interior size $i^j$. that since $\mathcal{I}_j$ is very-mutation-friendly, there exists a maximal weakly separated collection $\mathcal{W}^j \in \mathpzc{G}^{\mathcal{I}^j}$ and weakly separated  that satisfy:  
\[\mathcal{W}^j = \mathcal{W}^j_{0} \supsetneq \mathcal{W}^j_1 \supsetneq \mathcal{W}^j_2 \supsetneq \ldots \supsetneq \mathcal{W}^j_{i^j-1} \supsetneq \mathcal{I}^j \]
such that $\mathpzc{G}^{\mathcal{I}^j}(\mathcal{W}^j_l)$ is applicable and mutation-friendly for $1 \le l \le i^j$. 

We construct a sequence $(\mathcal{W}_1, \mathcal{W}_2, \ldots, \mathcal{W}_{\sum_{j=1}^i i_j})$ of weakly separated collections over $\mathcal{I}$ as follows. For $1 \le j \le i$, if $\sum_{j=1}^i i_{j-1} \le p < \sum_{j=1}^i i_j$, then 
\[\mathcal{W}_p = \cup_{l=1}^{j-1} (f^l(\mathcal{I}^l)) \cup f^j(\mathcal{W}^j_{p - \sum_{j=1}^i i_{j-1}}) \cup_{l=j+1}^i f^l(\mathcal{W}^l). \]
This satisfies the requirements of Definition~$\ref{vmf}$.
\end{proof}
\begin{proof}[Proof of Lemma~$\ref{whee}$]
Let $\mathcal{C} = \mathcal{W}_{i-1}$ using the notation in Definition~$\ref{vmf}$. By definition, we know that $\mathpzc{G}^{\mathcal{I}}(\mathcal{C})$ is mutation-friendly and applicable. Suppose that $\mathpzc{AC}_{\mathcal{V}}(V \setminus \mathcal{C})$ has $j$ weakly separated collections. Let $\mathcal{J}_l$ be the relabeled Grassmann necklace of the boundary curve of
\[\mathpzc{PT}_{\mathcal{V}}(\mathcal{A}^{\mathcal{V}}(\mathpzc{AC}_{\mathcal{V}}(V \setminus \mathcal{C})(l)))\] for $1 \le l \le j$. Then by Proposition~$\ref{vmfproperty}$, we know that $\mathcal{J}_l$ is very-mutation-friendly. By Lemma~$\ref{dirproduct}$, we know that 
\[\square_{l=1}^j \mathpzc{G}^{\mathcal{J}_l} \cong \mathpzc{G}^{\square_{l=1}^j \mathcal{J}_l}. \]
By Proposition~$\ref{sillydp}$, we know that $\square_{l=1}^j \mathcal{J}_l$ is also very-mutation-friendly. By Lemma~$\ref{isokey}$, we know that
\[\mathpzc{G}^{\square_{l=1}^j \mathcal{J}_l} \cong \mathpzc{G}^{\mathcal{I}} \] as desired.
\end{proof}
\subsection{Equivalence Classes $\mathscr{C}_i$}
We use the following notation to discuss $\mathscr{C}_i$.

Consider $\mathcal{I} \in \mathscr{C}_i$. Figure~$\ref{ptree}$ shows an element of $\mathscr{PT}^{*}(\mathpzc{G}^{\mathscr{C}_i})$ for $i=2,3,4$. Note that the adjacency graph $\mathpzc{AG}^V$ is a path with $i-1$ vertices for all $V \in \mathpzc{G}^{\mathcal{I}}$. Denote the sets in the adjacency graph, $\mathpzc{PT}_{\mathcal{I}}[1]$, $\mathpzc{PT}_{\mathcal{I}}[2]$,..,$\mathpzc{PT}_{\mathcal{I}}[i-1]$ (in order from left to right along $\mathpzc{AG}^V$). The exchange graph $\mathpzc{G}^{\mathcal{I}}$ is a path with $i$ vertices. We label the maximal weakly separated collections in $\mathpzc{G}^{\mathcal{I}})$ as $\mathpzc{PT}_{\mathcal{I}}(1)$, $\mathpzc{PT}_{\mathcal{I}}(2)$,$\ldots$, $\mathpzc{PT}_{\mathcal{I}}(k)$ so that
\begin{enumerate}
\item{In $\mathpzc{PT}_{\mathcal{I}}(1)$, the set $\mathpzc{PT}_{\mathcal{I}}[1]$ is the only mutatable set.}
\item{In $\mathpzc{PT}_{\mathcal{I}}(i)$, the set $\mathpzc{PT}_{\mathcal{I}}[i-1]$ is the only mutatable set.}
\item{In $\mathpzc{PT}_{\mathcal{I}}(j)$ for $1 < j < i$, the sets $\mathpzc{PT}_{\mathcal{I}}[j-1]$ and $\mathpzc{PT}_{\mathcal{I}}[j]$ are the only mutatable sets.}
\end{enumerate}
For $j = 1, i$, we let $\mathpzc{PT}_{\mathcal{I}}[j, END]$ be the set in the Grassmann necklace that is adjacent to $\mathpzc{PT}_{\mathcal{I}}[j]$ and not adjacent to $\mathpzc{PT}_{\mathcal{I}}[l]$ for any $l \neq j$.

Since $\mathscr{PT}^{*}(\mathcal{I}) \cong \mathscr{PT}^{*}(\mathcal{J})$ for all $\mathcal{I}, \mathcal{J} \in \mathscr{C}_i$, we denote the set of plabic tilings (without set labels) as $\mathscr{PT}^{*}(\mathscr{C}_i).$

We prove the following lemma that relates $\mathscr{C}_i$ to $\mathscr{C}_{i+1}$:
\begin{lemma}
\label{cyc45}
Let $\mathpzc{G}^{\mathcal{I}}$ be a very-mutation-friendly exchange graph with interior size $i$. Suppose there is an applicable $\mathcal{C}$-constant graph $\mathpzc{G}^{\mathcal{I}}(\mathcal{C})$ of co-dimension $i-1$ such that $\mathcal{C}$ is a reverse-very-connected weakly separated collection over $\mathcal{I}$ and $\mathscr{PT}^{*}(\mathpzc{G}^{\mathcal{I}}(\mathcal{C})) \cong \mathscr{PT}^{*}(\mathscr{C}_i)$ for some $i \ge 2$. If $\mathpzc{G}^{\mathcal{I}}$ does not have a $\mathcal{C}$-constant graph of co-dimension $2$ that is a cycle of length $4$ or $5$, then $\mathcal{I}$ is part of the equivalence class $\mathscr{C}_{i+1}$.
\end{lemma}

\begin{figure}
\caption{}
\label{ptree}
\includegraphics[scale=0.3]{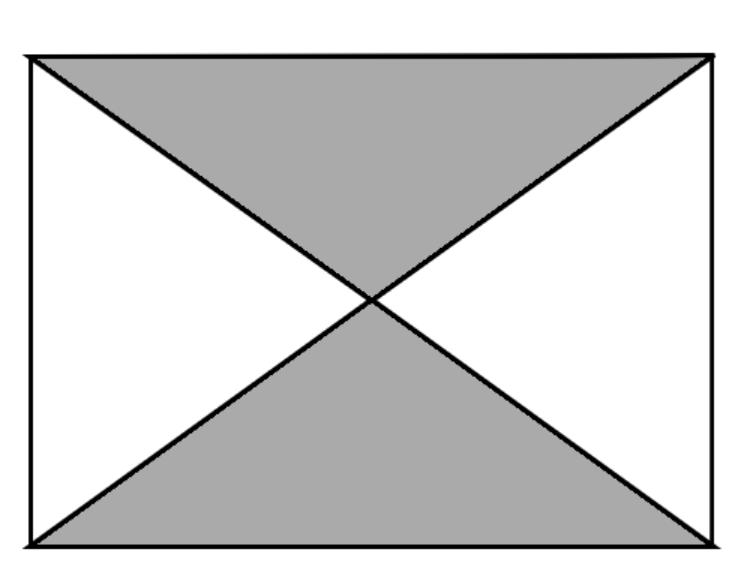}
\includegraphics[scale=0.3]{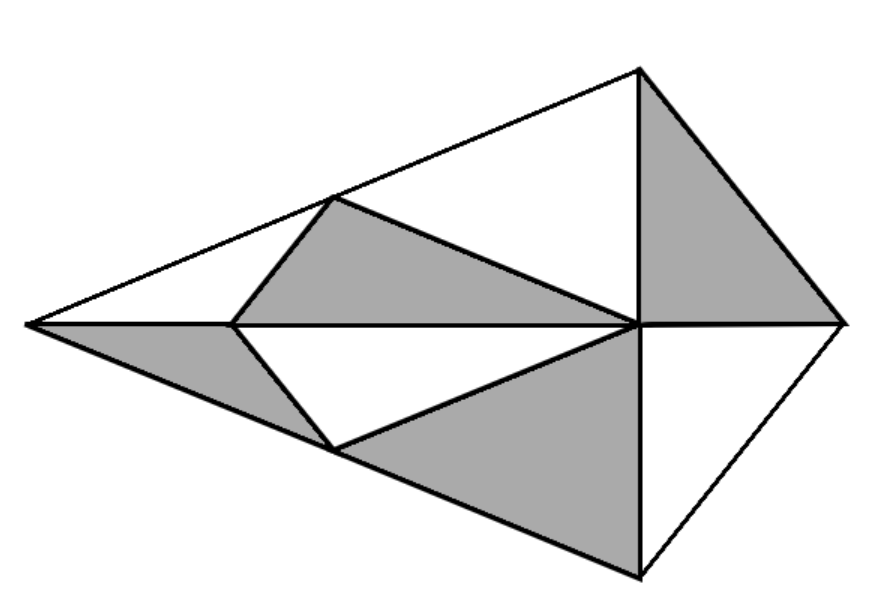}
\includegraphics[scale=0.3]{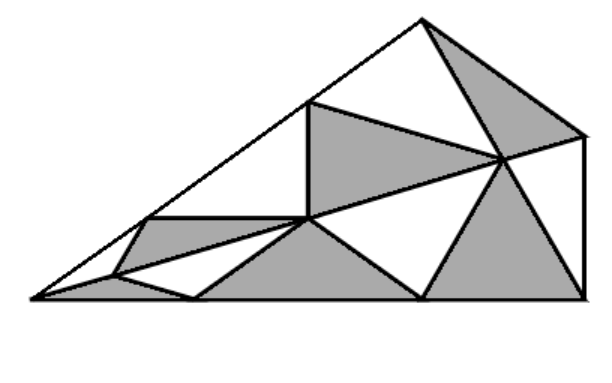}
\end{figure}

\begin{figure}
\caption{}
\label{onepl}
\includegraphics[scale=0.7]{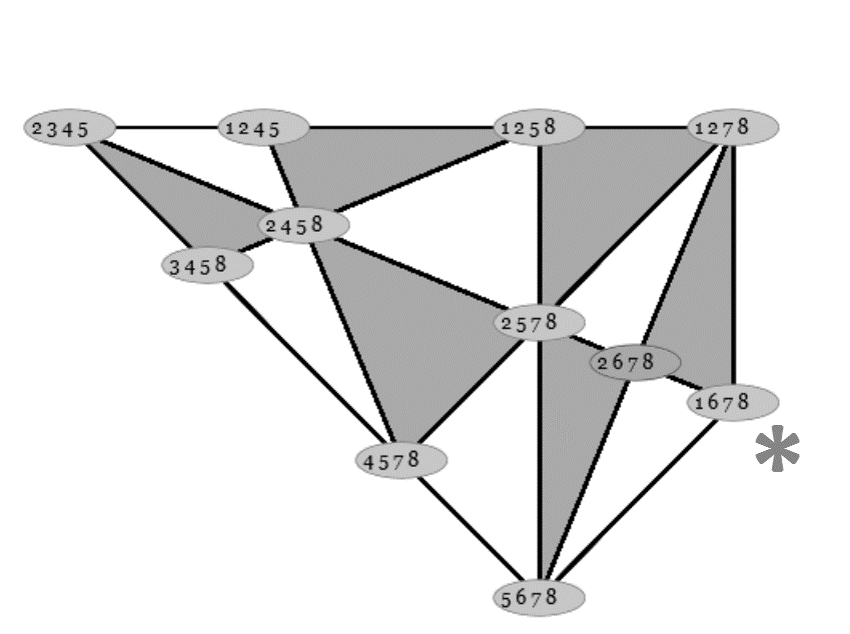}
\end{figure}

\begin{figure}
\caption{}
\label{twopl}
\includegraphics[scale=0.7]{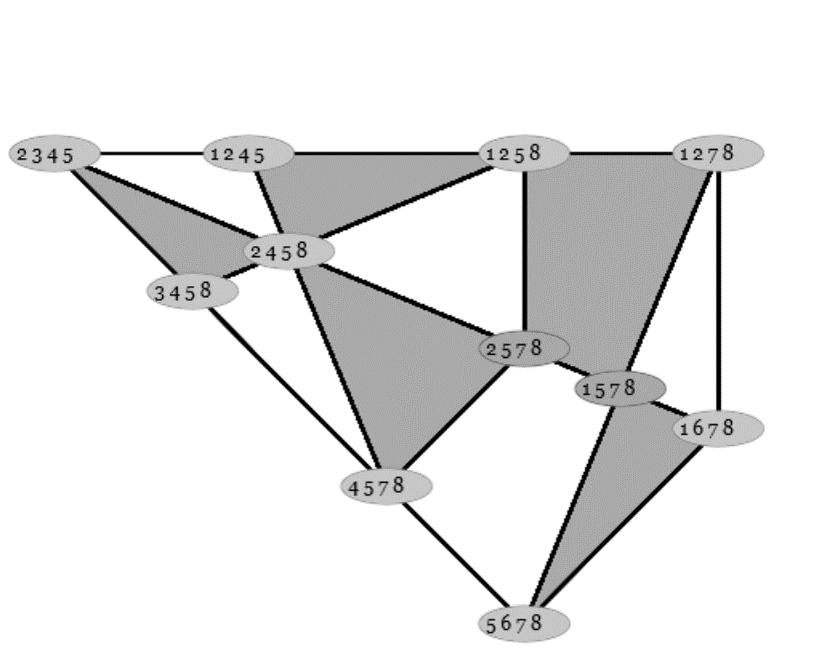}
\end{figure}
\begin{figure}
\caption{}
\label{pltree3}
\includegraphics[scale = 0.7]{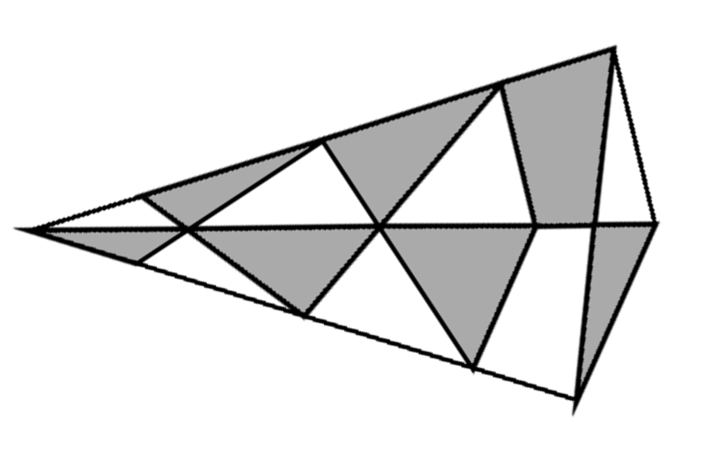}
\end{figure}
\begin{proof}
Assume that very-mutation-friendly exchange graph $\mathpzc{G}^{\mathcal{I}}$ does not have a $\mathcal{C}$-constant graph of co-dimension $2$ that is a cycle of length $4$ or $5$. By Lemma~$\ref{cyc4}$, this means that $\mathpzc{G}^{\mathcal{I}}$ must be prime. Suppose that $\mathpzc{G}^{\mathcal{I}}$ has an applicable, mutation-friendly $\mathcal{C}$-constant graph $\mathpzc{G}^\mathcal{I}(\mathcal{I} \cup \left\{r\right\})$ such that $\mathscr{PT}^{*}(\mathpzc{G}^\mathcal{I}(\mathcal{I} \cup \left\{r\right\}) \cong \mathscr{PT}^{*}(\mathpzc{G}^{\mathscr{C}_i})$ for some $i$. Let $\mathcal{J}$ be a Grassmann necklace in $\mathcal{C}_i$. Consider $\mathcal{V} \in \mathpzc{G}^{\mathcal{I}}(\mathcal{I} \cup \left\{ V \right\})$ such that $V$ is mutatable, and let $\mathcal{W} = \mathcal{A}^{\mathcal{V}}(\mathcal{V} \setminus (\mathcal{I} \cup \left\{V \right\}))$. Then $\mathpzc{PT}^{*}_{\mathcal{V}}(\mathcal{W}) \cong \mathpzc{PT}^{*}_{\mathcal{J}}(j)$ for some $1 \le j \le i$.

As we present the general proof, for ease of readability, we also consider the running example where $\mathcal{J}$ has decorated permutation $38761254$.

We will show that $j = 1$ or $k$. In the example, we consider $\mathpzc{PT}^{*}_{\mathcal{J}}(j)$ to be the plabic tiling shown in Figure~$\ref{onepl}$. Assume for sake of contradiction that $1 < j < i$. In the example, we consider $\mathpzc{PT}^{*}_{\mathcal{J}}(j)$ to be the plabic tiling in Figure~$\ref{twopl}$. The set $V$ must be adjacent to the sets corresponding to $\mathpzc{PT}_{\mathcal{J}}[j-1]$ and $\mathpzc{PT}_{\mathcal{J}}[j]$ in $\mathpzc{PT}^{*}_{\mathcal{J}}(j)$. In the example, these are the sets $\left\{2, 5, 7, 8 \right\}$ and $\left\{1, 5, 7, 8 \right\}$, the two mutatable faces. (If $V$ is not adjacent to a set $W$ corresponding to one of them, then the $\mathcal{C}$-constant graph $\mathpzc{G}^{\mathcal{I}}(\mathcal{V} \setminus \left\{V, W\right\})$ will be a cycle of length $4$.) This leaves no spots for $V$.

Hence, $j = 1$ or $i$. We know that $V$ must be adjacent to the set $W$ corresponding to $\mathpzc{PT}_{\mathcal{J}}[1]$ if $j=1$ or corresponding to $\mathpzc{PT}_{\mathcal{J}}[k]$ if $j= i$. In the example, the set $W$ corresponds to $\left\{2, 6, 7, 8 \right\}$ as shown in Figure~$\ref{onepl}$. (If $V$ is not adjacent to $W$, then the $\mathcal{C}$-constant graph $\mathpzc{G}^{\mathcal{I}}(\mathcal{V} \setminus \left\{V, W\right\})$ will be a cycle of length $4$.) Notice that $|\mathcal{A}^{\mathcal{V}}(\left\{V, W \right\})| \in \{7,8\}$. However, if $|\mathcal{A}^{\mathcal{V}}(\left\{V, W \right\})| = 7$, then $\mathpzc{G}^{\mathcal{I}}(\mathcal{V} \setminus \left\{V, W\right\})$ would be isomorphic to a very-mutation-friendly, prime exchange graph with $n=5$ and interior size $2$, which by Table~$\ref{fullchar}$ is a cycle with $5$ vertices. Hence, we know that $|\mathcal{A}^{\mathcal{V}}(\left\{V, W \right\})| = 8$.

We will now show that $V$ must be $\mathpzc{PT}_{\mathcal{J}}[j, END]$, that is, it must be in the starred location in Figure~$\ref{onepl}$. Assume for sake of contradiction that $V$ corresponds to some other set in the Grassmann necklace besides $\mathpzc{PT}_{\mathcal{J}}[l, END]$. Since $V$ is mutatable, this means that $|\mathcal{A}^{\mathcal{V}}(V)| = 7$, which is a contradiction. In the example, suppose that $V$ corresponds to the set $\left\{1, 2, 7, 8 \right\}$ in Figure~$\ref{onepl}$ and $W$ corresponds to the set $\left\{2, 6, 7, 8 \right\}$. Notice that $\mathcal{A}^{\mathcal{V}}(\left\{V, W \right\})$ would consist of $\left\{2, 5, 7, 8 \right\}$, $\left\{1, 2, 7, 8 \right\}$, $\left\{1,6,7,8\right\}$, $\left\{5,6,7,8\right\}$, $\left\{2,6,7,8\right\}$, and two more sets adjacent to $\left\{1,2,7,8\right\}$.

Hence, $V$ is in the starred location with $|\mathcal{A}^{\mathcal{V}}(\left\{V, W \right\}| = 8$. This means that $\mathpzc{PT}^{*}(\mathcal{V}) \cong \mathpzc{PT}^{*}_{\mathcal{I}}(m)$ for $m=1,2,i,i+1$, where $\mathcal{I}$ is a Grassmann necklace in $\mathcal{C}_{i+1}$. In the example, this is shown in Figure~$\ref{pltree3}$. This means that $\mathcal{I}$ is part of $\mathscr{C}_{i+1}$.
\end{proof}

\subsection{Proof of Theorem~$\ref{tree}$}
We prove Theorem~$\ref{tree}$.
\begin{proof}[Proof of Theorem~$\ref{tree}$]
We notice right away that non prime, very-mutation-friendly exchange graphs are not trees by Lemma~$\ref{cyc4}$. We prove that prime, very-mutation-friendly exchange graphs with interior size $i$ are trees if and only if their Grassmann necklaces are a part of the equivalence class $\mathscr{C}_i$ (which yield paths). We proceed by induction on the interior size of $\mathpzc{G}^{\mathcal{I}}$.

The base case is when $i(\mathpzc{G}^{\mathcal{I}})=0$ or $1$. By Table~$\ref{fullchar}$, the only possible associated decorated permutations corresponding to $\mathcal{I}$ are $321$ and $3412$ as desired.

We assume that prime, very-mutation-friendly exchange graphs with interior size $i$ are trees if and only if their Grassmann necklaces are a part of the equivalence class $\mathscr{C}_i$ (which yield paths). Now, we consider a prime, very-mutation-friendly exchange graph $\mathpzc{G}^{\mathcal{I}}$ with interior size $i+1$ that is a tree. By Lemma~$\ref{whee}$, we know that there exists a mutation-friendly $\mathcal{C}$-constant graph $\mathpzc{G}^{\mathcal{I}}(\mathcal{C})$ of co-dimension $i$ that is isomorphic to a very-mutation-friendly exchange graph $\mathpzc{G}^{\mathcal{J}}$ with interior size $i$. This graph must be a tree too, so by the induction hypothesis, $\mathcal{J}$ is part of the equivalence class $\mathscr{C}_i$. By Lemma~$\ref{nrvc}$, we know that $\mathcal{C}$ must be reverse-very-connected. Since a tree does not have any cycles, by Lemma~$\ref{cyc45}$, $\mathcal{I}$ is part of the equivalence class $\mathscr{C}_{i+1}$.
\end{proof}
\subsection{Proof of Theorem~$\ref{cycle}$}
We will use the following lemmas in our proof of Theorem~$\ref{cycle}$:
\begin{lemma}
\label{cycprime}
Let $\mathpzc{G}^{\mathcal{I}}$ be a very-mutation-friendly and prime exchange graph. If $i(\mathpzc{G}^{\mathcal{I}}) > 2$, then $\mathpzc{G}^{\mathcal{I}}$ is not a single cycle.
\end{lemma}
\begin{proof}
Assume for sake of contradiction that there exists a prime, very-mutation-friendly exchange graph $\mathpzc{G}^{\mathcal{I}}$ with  $i(\mathpzc{G}^{\mathcal{I}}) > 2$ that is a single cycle. By Lemma~$\ref{whee}$, we know that there exists a mutation-friendly $\mathcal{C}$-constant graph $\mathpzc{G}^{\mathcal{I}}(\mathcal{C})$ of co-dimension $i(\mathpzc{G}^{\mathcal{I}}) - 1$ that is isomorphic to a very-mutation-friendly exchange graph $\mathpzc{G}^{\mathcal{J}}$ with interior size $i(\mathpzc{G}^{\mathcal{I}}) -1$. Since $\mathpzc{G}^{\mathcal{I}}$ is a cycle, $\mathpzc{G}^{\mathcal{J}}$ must be a path. By Lemma~$\ref{nrvc}$, we know that $\mathcal{C}$ must be reverse-very-connected. 

By Theorem~$\ref{tree}$, we know that $\mathcal{J}$ must be in the equivalence class $\mathscr{C}_{i}$. By Lemma~$\ref{cyc45}$, we know that $\mathcal{I}$ must be in the equivalence class $\mathscr{C}_{i+1}$, which means that $\mathpzc{G}^{\mathcal{I}}$ is a path. This is a contradiction.
\end{proof}
\begin{lemma}
\label{cycnotprime}
If $\mathpzc{G}^{\mathcal{I}}$ be a very-mutation-friendly, non-prime exchange graph that is a single cycle, then $\mathpzc{G}^{\mathcal{I}}$ has $4$ vertices and $\mathcal{I}$ is in the equivalence class with $351624$.
\end{lemma}
\begin{proof}
Consider a very-mutation-friendly, non-prime exchange graph $\mathpzc{G}^{\mathcal{I}}$ that is a single cycle. By Lemma~$\ref{cyc4}$, we know that the cycle must have $4$ vertices. This is only the case when the graph is a direct product of two paths each with $2$ vertices so that the $\mathcal{D}^{\mathcal{I}}$ has two Grassmann necklaces. By Lemma~$\ref{vmfdirproduct}$, both $\mathcal{D}^{\mathcal{I}}(1)$ and $\mathcal{D}^{\mathcal{I}}(2)$ are very-mutation-friendly. By Theorem~$\ref{tree}$, we know that $\mathcal{D}^{\mathcal{I}}(1)$ and $\mathcal{D}^{\mathcal{I}}(2)$ both must have interior size $1$. Thus, $\mathcal{I}$ must be in the equivalence class containing the permutation $351624$.
\end{proof}

We now prove Theorem~$\ref{cycle}$.
\begin{proof}[Proof of Theorem~$\ref{cycle}$]
Suppose that $\mathpzc{G}^{\mathcal{I}}$ is a cycle. First, we consider the case where $\mathpzc{G}^{\mathcal{I}}$ is prime and very-mutation-friendly. By Lemma~$\ref{cycprime}$, it suffices to consider $i(\mathpzc{G}^{\mathcal{I}}) \le 2$. By Table~$\ref{fullchar}$, we know that $\mathcal{I}$ must be part of the equivalence class $312$ (cycle with $1$ vertex), $3412$ (cycle with $2$ vertices), or $34512$ (cycle with $5$ vertices).

Now, we consider the case in which $\mathpzc{G}^{\mathcal{I}}$ is very-mutation-friendly, but not prime. Then by Lemma~$\ref{cycnotprime}$, $\mathcal{I}$ is in the equivalence class $351624$ (cycle with $4$ vertices).
\end{proof}
\section{Acknowledgments}
I would like to thank Miriam Farber (MIT) for her incredibly helpful guidance and insight. I would also like to thank Professor Alexander Postnikov and the MIT-PRIMES program for suggesting this project, and Professor David Speyer for correcting an error in an earlier version of the paper. In addition, I would like to thank Pavel Galashin (MIT) for his helpful software available at \url{http://math.mit.edu/~galashin/plabic.html} which was used to generate the plabic tilings and plabic graphs in this paper.

\clearpage

\bibliographystyle{plain}
\bibliography{bibliography}
\appendix
\section{Full Characterization}
Table~$\ref{fullchar}$ is a full characterization of all possible exchange graphs of very-mutation-friendly, prime Grassmann necklaces with interior size $\le$ 4. There is at least one decorated permutation from each equivalence class, and it is possible that multiple decorated permutations from the same equivalence class are listed.
\begin{table}
\caption{}
\label{fullchar}
\begin{tabular}{||c | c | c  | c||}
\hline
Interior Size & Equivalence Class & Exchange Graph Order & Exchange Graph\\ [0.5 ex]
\hline
\hline
0 & 312 & 1 & A\\
\hline \hline
1 & 3412 & 2 & B\\
\hline \hline
2 & 365124 & 3 & C\\
\hline
2 & 34512 & 5 & D\\
\hline
\hline
3 & 38761254 & 4 & E\\
\hline
3 & 38517246 & 5 & F\\
\hline
3 & 38571426 & 5 & F\\
\hline
3 & 38617425 & 5 & F\\
\hline
3 & 3576214 & 5 & F\\
\hline
3 & 3756124 & 7 & G\\
\hline
3 & 356124 & 10 & H\\
\hline
3 & 345612 & 14 & I\\
\hline
\hline
4 & 3(10)98712654 & 5 & J\\
\hline
 4 & 37682154 & 6 & K\\ \hline
4 & 3(10)96182574 & 7 & L\\
\hline
4 & 3(10)96815274 & 7 & L\\
\hline
4 & 3(10)97185264 & 7 & L\\
\hline
 4 & 397618254 & 7 & L\\ \hline
 4 & 36587214 & 7 & M\\ \hline
4 & 3(10)57941628 & 8 & N \\
\hline
4 & 3(10)51829647 & 8 & N\\
\hline
4 & 3(10)51729468 & 8 & N\\
\hline
4 & 3(10)51792648 & 8 & N \\
\hline
4 & 3(10)81794625 & 8 & N\\
\hline
4 & 3(10)69741528 & 8 & N \\
\hline
 4 & 395871426 & 8 & N \\ \hline
 4 & 397618425 & 8 & N\\ \hline
 4 & 379628415 & 8 & N\\ \hline
 4 & 36872154 & 8 & N\\ \hline
4 & 3(10)98714625 & 9 & O \\
\hline
4 & 3(10)98741526 & 9 & O \\
\hline
4 & 3(10)96184527 & 9 & O \\
\hline
4 & 3(10)96815427 & 9 & O\\
\hline
4 & 3(10)97185426 & 9 & O\\
\hline
4 & 3(10)97841625 & 9 & O\\
\hline
4 & 3(10)61982574 & 9 & O\\
\hline
4 & 3(10)61895274 & 9 & O\\
\hline
4 & 3(10)81729564 & 9 & O\\
\hline
 4 & 375129(10)648 & 9 & O\\ \hline
 4 & 39(10)6284517 & 9 & O\\ \hline
 4 & 4(10)62895173 & 9 & O\\ \hline
\end{tabular}
\end{table}
\begin{table}
\begin{tabular}{||c | c | c | c ||}
\hline
Interior Size & Equivalence Class & Exchange Graph Order & Exchange Graph\\ [0.5 ex]
\hline \hline
 4 & 398671254 & 9 & O\\ \hline
 4 & 395871264 & 9 & O\\ \hline
4 & 3(10)71985264 & 9 & O\\
\hline
4 & 3(10)58149627 & 9 & P\\
\hline
 4 & 395718426 & 9 & P\\ \hline
 4 & 35872146 & 9 & P\\ \hline
 4 & 395784126 & 10 & Q \\ \hline
 4 & 396874125 & 10 & Q\\ \hline
 4 & 397681524 & 10 & Q\\ \hline
 4 & 48672153 & 10 & Q\\ \hline
 4 & 37862145 & 10 & Q\\ \hline
 4 & 395618247 & 11 & R\\ \hline
 4 & 34768215 & 11 & R\\ \hline
 4 & 395681427 & 12 & S\\ \hline
 4 & 396178425 & 12 & S\\ \hline
 4 & 395178246 & 12 & S\\ \hline
 4 & 34761825 & 12 & S\\ \hline
 4 & 35871426 & 12 & S\\ \hline
 4 & 36872415 & 12 & S \\ \hline
 4 & 38761245 & 13 & T\\ \hline
 4 & 38671254 & 13 & T\\ \hline
 4 & 38571246 & 14 & U\\ \hline
 4 & 34(10)1895276 & 15 & V\\ \hline
4 & 3(10)91782564 & 15 & V\\
\hline
4 & 3(10)91784526 & 15 & V\\
\hline
4 & 3(10)96714528 & 15 & V\\
\hline
4 & 3(10)96781524 & 15 & V\\
\hline
4 & 3(10)96784125 & 15 & V\\
\hline
4 & 3(10)97815624 & 15 & V\\
\hline
4 & 3(10)97845126 & 15 & V \\
\hline
4 & 3(10)51789264 & 15 & V\\
\hline
4 & 3(10)51289467 & 15 & V\\
\hline
4 & 3(10)51892674 & 15 & V\\
\hline
4 & 3(10)56794128 & 15 & V\\
\hline
4 & 3(10)61789524 & 15 & V\\
\hline
4 & 3(10)67945128 & 15 &V \\
\hline
4 & 3(10)71895624 & 15 & V\\
\hline
4 & 3(10)71289564 & 15 & V\\
\hline
 4 & 365129(10)478 & 15 & V\\ \hline
 4 & 34(10)8719562 & 15 & V\\ \hline
 4 & 398671425 & 15 & V \\ \hline
 \end{tabular}
\end{table}
\begin{table}
\begin{tabular}{||c | c | c | c ||}
\hline
Interior Size & Equivalence Class & Exchange Graph Order & Exchange Graph\\ [0.5 ex]
\hline \hline
 4 & 348172956 & 15 & V\\ \hline
 4 & 346189527 & 15 & V \\ \hline
 4 & 396178254 & 15 & V\\ \hline
 4 & 38617245 & 15 & W\\ \hline
 4 & 38671425 & 15 & W \\ \hline
 4 & 34867125 & 16 & X\\ \hline
 4 & 34871256 & 16 & X\\ \hline
 4 & 3567214 & 17 & Y\\ \hline
 4 & 38567124 & 19 & Z1\\ \hline
 4 & 3576124 & 20 & Z2\\ \hline
 4 & 34(10)1789265 & 25 & Z3\\ \hline
 4 & 34(10)1789526 & 25 & Z3\\ \hline
 4 & 34(10)1289567 & 25 & Z3\\ \hline
 4 & 34(10)1892675 & 25 & Z3\\ \hline
 4 & 34(10)1895627 & 25 & Z3\\ \hline
4 & 34(10)8195672 & 25 & Z3\\ \hline
 4 & 34(10)9815672 & 25 & Z3\\ \hline
 4 & 34(10)6789512 & 25 & Z3\\ \hline
4 & 34(10)7895612 & 25 & Z3\\ \hline
 4 & 4(10)17895623 & 25 & Z3\\ \hline
 4 & 349178256 & 25 & Z3 \\ \hline
 4 & 349781562 & 25 & Z3 \\ \hline
4 & 349678152 & 25 & Z3 \\ \hline
 4 & 34718256 & 25 & Z3\\ \hline
 4 & 34617825 & 25 & Z3\\ \hline
 4 & 3467125 & 26 & Z4 \\ \hline
 4 & 456123 &  34 & Z5\\ \hline
 4 & 3456712 & 42 & Z6\\ \hline
 \end{tabular}
\end{table}

Table~$\ref{allfigures}$ defines the undirected graphs $A, B, C \ldots Y, Z_1, Z_2 \ldots Z_6$. The table either displays an embedding of the graph or the adjacency information for each vertex.
\begin{table}
\caption{}
\label{allfigures}
\begin{tabular}{||c | c ||}
\hline
Name & Graph [Figure or Adjacency Information] \\ [0.5 ex]
\parbox[c]{1em}{A} & \parbox[c]{1em}{\includegraphics[scale = 0.4]{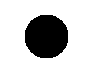}}\\ 
\parbox[c]{1em}{B} & \parbox[c]{1em}{\includegraphics[scale = 0.4]{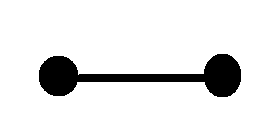}}\\ 
\parbox[c]{1em}{C} & \parbox[c]{1em}{\includegraphics[scale = 0.4]{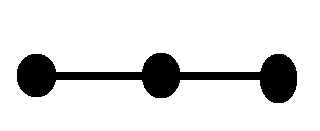}}\\
\parbox[c]{1em}{D} & \parbox[c]{1em}{\includegraphics[scale = 0.4]{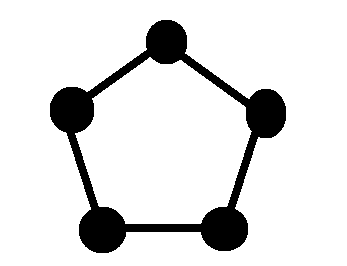}}\\
\parbox[c]{1em}{E} & \parbox[c]{1em}{\includegraphics[scale = 0.4]{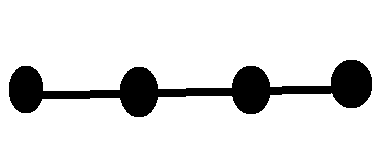}}\\
\parbox[c]{1em}{F} & \parbox[c]{1em}{\includegraphics[scale = 0.4]{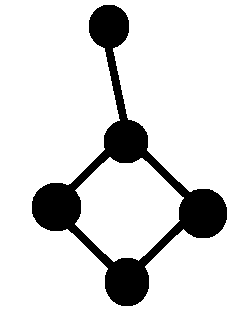}}\\
\parbox[c]{1em}{G} & \parbox[c]{1em}{\includegraphics[scale = 0.4]{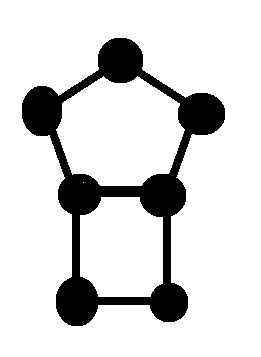}}\\
\parbox[c]{1em}{H} & \parbox[c]{1em}{\includegraphics[scale = 0.4]{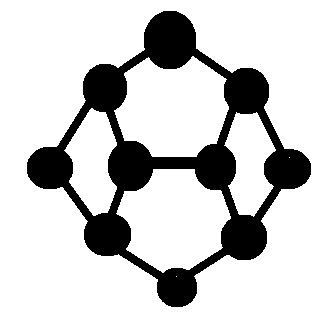}}\\
\parbox[c]{1em}{I} & \parbox[c]{1em}{\begin{align*} 
1 & \rightarrow 2,3,4 \\ 
2 & \rightarrow 5,6,1 \\ 
3 & \rightarrow 7,8,1 \\ 
4 & \rightarrow 6,9,1 \\ 
5 & \rightarrow 10,7,2 \\ 
6 & \rightarrow 11,4,2 \\ 
7 & \rightarrow 12,3,5 \\ 
8 & \rightarrow 13,9,3 \\ 
9 & \rightarrow 14,4,8 \\ 
10 & \rightarrow 12,11,5 \\ 
11 & \rightarrow 14,6,10 \\ 
12 & \rightarrow 13,7,10 \\ 
13 & \rightarrow 8,14,12 \\ 
14 & \rightarrow 9,11,13 \\
 \end{align*}} \\ \hline
 \end{tabular}
 \end{table}
\begin{table}
\begin{tabular}{||c | c ||}
\hline
Name (cont.) & Graph [Figure or Adjacency Information] (cont.) \\ [0.5 ex]
\parbox[c]{1em}{J} & \parbox[c]{1em}{\includegraphics[scale = 0.4]{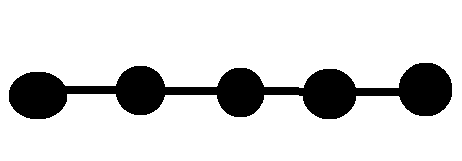}}\\
\parbox[c]{1em}{K} & \parbox[c]{1em}{\includegraphics[scale = 0.4]{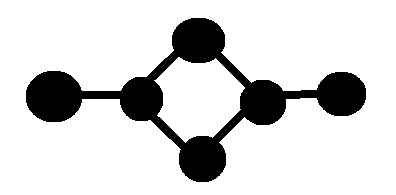}}\\
\parbox[c]{1em}{L} & \parbox[c]{1em}{\includegraphics[scale = 0.4]{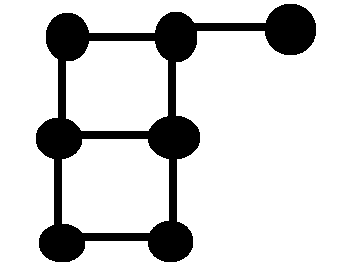}}\\
\parbox[c]{1em}{M} & \parbox[c]{1em}{\includegraphics[scale = 0.4]{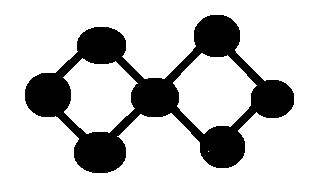}}\\
\parbox[c]{1em}{N} & \parbox[c]{1em}{\includegraphics[scale = 0.4]{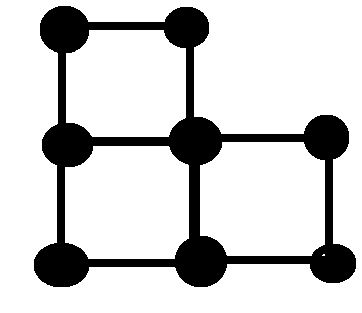}}\\
\parbox[c]{1em}{O} & \parbox[c]{1em}{\includegraphics[scale = 0.4]{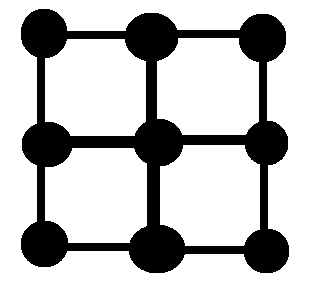}}\\
\parbox[c]{1em}{P} & \parbox[c]{1em}{\begin{align*} 
1 & \rightarrow 2 \\ 
2 & \rightarrow 3,4,1,5 \\ 
3 & \rightarrow 6,2,7 \\ 
4 & \rightarrow 6,8,2 \\ 
5 & \rightarrow 7,8,2 \\ 
6 & \rightarrow 4,9,3 \\ 
7 & \rightarrow 9,5,3 \\ 
8 & \rightarrow 9,5,4 \\ 
9 & \rightarrow 8,7,6 \\
 \end{align*}} \\
\parbox[c]{1em}{Q} & \parbox[c]{1em}{\includegraphics[scale = 0.4]{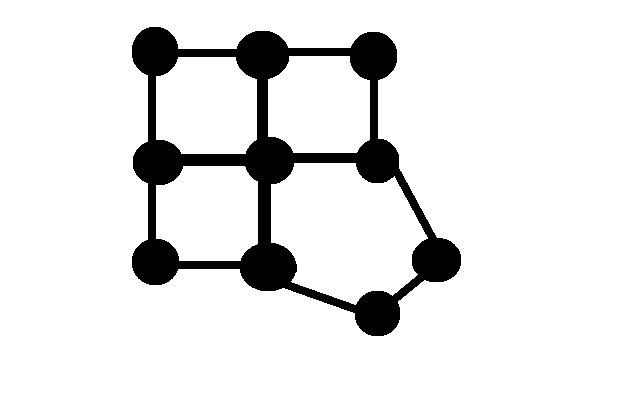}}\\ 
\parbox[c]{1em}{R} & \parbox[c]{1em}{\begin{align*} 
1 & \rightarrow 2,3,4 \\ 
2 & \rightarrow 5,1,6 \\ 
3 & \rightarrow 5,7,8,1 \\ 
4 & \rightarrow 6,8,1 \\
\end{align*}} \\ \hline 
\end{tabular}
\end{table}
\begin{table}
\begin{tabular}{||c | c ||}
\hline
Name (cont.) & Graph [Figure or Adjacency Information] (cont.) \\ [0.5 ex]
\parbox[c]{1em}{R (cont.)} & \parbox[c]{1em}{\begin{align*}
5 & \rightarrow 3,9,2 \\ 
6 & \rightarrow 9,4,2 \\ 
7 & \rightarrow 10,3 \\ 
8 & \rightarrow 9,11,4,3 \\ 
9 & \rightarrow 8,6,5 \\ 
10 & \rightarrow 11,7 \\ 
11 & \rightarrow 8,10 \\
 \end{align*}} \\
\parbox[c]{1em}{S} & \parbox[c]{1em}{\begin{align*} 
1 & \rightarrow 2,3,4 \\ 
2 & \rightarrow 5,6,1 \\ 
3 & \rightarrow 5,7,1 \\ 
4 & \rightarrow 6,8,1 \\ 
5 & \rightarrow 9,3,2 \\ 
6 & \rightarrow 10,4,2 \\ 
7 & \rightarrow 9,8,3 \\ 
8 & \rightarrow 10,4,7 \\ 
9 & \rightarrow 11,7,10,5 \\ 
10 & \rightarrow 12,8,6,9 \\ 
11 & \rightarrow 9,12 \\ 
12 & \rightarrow 10,11 \\
 \end{align*}} \\
\parbox[c]{1em}{T} & \parbox[c]{1em}{\begin{align*} 
1 & \rightarrow 2,3 \\ 
2 & \rightarrow 4,1,5 \\ 
3 & \rightarrow 4,6,1 \\ 
4 & \rightarrow 7,3,8,2 \\ 
5 & \rightarrow 8,2 \\ 
6 & \rightarrow 9,10,3 \\ 
7 & \rightarrow 9,11,4 \\ 
8 & \rightarrow 11,5,4 \\ 
9 & \rightarrow 12,6,7 \\ 
10 & \rightarrow 12,6 \\ 
11 & \rightarrow 13,8,7 \\ 
12 & \rightarrow 10,13,9 \\ 
13 & \rightarrow 11,12 \\
 \end{align*}} \\ \hline
 \end{tabular}
 \end{table}
 \begin{table}
\begin{tabular}{||c | c ||}
\hline
Name (cont.) & Graph [Figure or Adjacency Information] (cont.) \\ [0.5 ex]
\parbox[c]{1em}{U} & \parbox[c]{1em}{\begin{align*} 
1 & \rightarrow 2,3,4 \\ 
2 & \rightarrow 5,1,6 \\ 
3 & \rightarrow 5,7,8,1 \\ 
4 & \rightarrow 6,8,1 \\ 
5 & \rightarrow 3,9,2 \\ 
6 & \rightarrow 9,4,2 \\ 
7 & \rightarrow 10,3 \\ 
8 & \rightarrow 9,11,4,3 \\ 
9 & \rightarrow 12,8,6,5 \\ 
10 & \rightarrow 13,11,7 \\ 
11 & \rightarrow 14,8,10 \\ 
12 & \rightarrow 14,9 \\ 
13 & \rightarrow 10,14 \\ 
14 & \rightarrow 11,12,13 \\
 \end{align*}} \\
\parbox[c]{1em}{V} & \parbox[c]{1em}{\begin{align*} 
1 & \rightarrow 2,3,4 \\ 
2 & \rightarrow 5,6,7,1 \\ 
3 & \rightarrow 5,8,1 \\ 
4 & \rightarrow 6,9,1 \\ 
5 & \rightarrow 10,11,3,2 \\ 
6 & \rightarrow 12,13,4,2 \\ 
7 & \rightarrow 11,13,2 \\ 
8 & \rightarrow 10,9,3 \\ 
9 & \rightarrow 12,4,8 \\ 
10 & \rightarrow 14,8,12,5 \\ 
11 & \rightarrow 14,5,7 \\ 
12 & \rightarrow 15,9,6,10 \\ 
13 & \rightarrow 15,6,7 \\ 
14 & \rightarrow 10,15,11 \\ 
15 & \rightarrow 12,13,14 \\
 \end{align*}} \\
 \parbox[c]{1em}{W} & \parbox[c]{1em}{\begin{align*} 
1 & \rightarrow 2,3,4 \\ 
2 & \rightarrow 5,6,1 \\ 
3 & \rightarrow 7,8,1 \\ 
4 & \rightarrow 6,8,9,1 \\ 
\end{align*}} \\ \hline
\end{tabular}
\end{table}
\begin{table}
\begin{tabular}{||c | c ||}
\hline
Name (cont.) & Graph [Figure or Adjacency Information] (cont.) \\ [0.5 ex]
\parbox[c]{1em}{W (cont.)} & \parbox[c]{1em}{\begin{align*} 
5 & \rightarrow 10,7,2 \\ 
6 & \rightarrow 10,11,4,2 \\ 
7 & \rightarrow 12,3,5 \\ 
8 & \rightarrow 12,13,4,3 \\ 
9 & \rightarrow 14,13,4 \\ 
10 & \rightarrow 15,12,6,5 \\ 
11 & \rightarrow 15,14,6 \\ 
12 & \rightarrow 8,10,7 \\ 
13 & \rightarrow 9,8 \\ 
14 & \rightarrow 9,11 \\ 
15 & \rightarrow 11,10 \\
 \end{align*}} \\
 \parbox[c]{1em}{X} & \parbox[c]{1em}{\begin{align*} 
1 & \rightarrow 2,3,4 \\ 
2 & \rightarrow 5,6,1 \\ 
3 & \rightarrow 7,8,1 \\ 
4 & \rightarrow 6,8,1 \\ 
5 & \rightarrow 9,10,7,2 \\ 
6 & \rightarrow 10,4,2 \\ 
7 & \rightarrow 11,12,3,5 \\ 
8 & \rightarrow 12,4,3 \\ 
9 & \rightarrow 13,11,5 \\ 
10 & \rightarrow 14,12,6,5 \\ 
11 & \rightarrow 15,7,9 \\ 
12 & \rightarrow 16,8,10,7 \\ 
13 & \rightarrow 15,14,9 \\ 
14 & \rightarrow 16,10,13 \\ 
15 & \rightarrow 16,11,13 \\ 
16 & \rightarrow 12,14,15 \\ 
 \end{align*}} \\
 \parbox[c]{1em}{Y} & \parbox[c]{1em}{\begin{align*} 
1 & \rightarrow 2,3 \\ 
2 & \rightarrow 4,5,1,6 \\ 
3 & \rightarrow 7,8,1 \\ 
4 & \rightarrow 9,7,2,10 \\ 
5 & \rightarrow 9,11,2 \\ \end{align*}} \\ \hline
\end{tabular}
\end{table}
\begin{table}
\begin{tabular}{||c | c ||}
\hline
Name (cont.) & Graph [Figure or Adjacency Information] (cont.) \\ [0.5 ex]
 \parbox[c]{1em}{Y (cont.)} & \parbox[c]{1em}{\begin{align*}
6 & \rightarrow 10,11,2 \\ 
7 & \rightarrow 3,12,4 \\ 
8 & \rightarrow 12,3 \\ 
9 & \rightarrow 5,13,4 \\ 
10 & \rightarrow 13,14,6,4 \\ 
11 & \rightarrow 13,6,5 \\ 
12 & \rightarrow 15,8,14,7 \\ 
13 & \rightarrow 16,11,10,9 \\ 
14 & \rightarrow 17,10,12 \\ 
15 & \rightarrow 17,12 \\ 
16 & \rightarrow 17,13 \\ 
17 & \rightarrow 14,16,15 \\
 \end{align*}} \\
 \parbox[c]{1em}{Z1} & \parbox[c]{1em}{\begin{align*} 
1 & \rightarrow 2,3,4 \\ 
2 & \rightarrow 5,6,7,1 \\ 
3 & \rightarrow 6,8,1 \\ 
4 & \rightarrow 7,9,1 \\ 
5 & \rightarrow 10,11,2 \\ 
6 & \rightarrow 12,13,3,2 \\ 
7 & \rightarrow 11,14,4,2 \\ 
8 & \rightarrow 13,9,3 \\ 
9 & \rightarrow 14,4,8 \\ 
10 & \rightarrow 15,12,5 \\ 
11 & \rightarrow 16,7,5 \\ 
12 & \rightarrow 17,6,10 \\ 
13 & \rightarrow 18,8,14,6 \\ 
14 & \rightarrow 19,9,7,13 \\ 
15 & \rightarrow 17,16,10 \\ 
16 & \rightarrow 19,11,15 \\ 
17 & \rightarrow 18,12,15 \\ 
18 & \rightarrow 13,19,17 \\ 
19 & \rightarrow 14,16,18 \\ 
 \end{align*}} \\
 \parbox[c]{1em}{Z2} & \parbox[c]{1em}{\begin{align*} 
1 & \rightarrow 2,3 \\ 
2 & \rightarrow 4,5,1,6 \\
\end{align*}} \\ \hline
\end{tabular}
\end{table}
\begin{table}
\begin{tabular}{||c | c ||}
\hline
Name (cont.) & Graph [Figure or Adjacency Information] (cont.) \\ [0.5 ex]
 \parbox[c]{1em}{Z2 (cont.)} & \parbox[c]{1em}{\begin{align*} 
3 & \rightarrow 4,7,1 \\ 
4 & \rightarrow 8,3,9,2 \\ 
5 & \rightarrow 10,11,2 \\ 
6 & \rightarrow 9,11,2 \\ 
7 & \rightarrow 12,13,3 \\ 
8 & \rightarrow 12,14,10,4 \\ 
9 & \rightarrow 14,6,4 \\ 
10 & \rightarrow 15,16,5,8 \\ 
11 & \rightarrow 16,6,5 \\ 
12 & \rightarrow 17,15,7,8 \\ 
13 & \rightarrow 17,7 \\ 
14 & \rightarrow 18,16,9,8 \\ 
15 & \rightarrow 19,10,12 \\ 
16 & \rightarrow 20,11,14,10 \\ 
17 & \rightarrow 19,13,18,12 \\ 
18 & \rightarrow 20,14,17 \\ 
19 & \rightarrow 20,15,17 \\ 
20 & \rightarrow 16,18,19 \\ 
 \end{align*}} \\
 \parbox[c]{1em}{Z3} & \parbox[c]{1em}{\begin{align*} 
1 & \rightarrow 2,3,4,5 \\ 
2 & \rightarrow 6,7,8,1 \\ 
3 & \rightarrow 9,10,11,1 \\ 
4 & \rightarrow 7,10,12,1 \\ 
5 & \rightarrow 8,11,13,1 \\ 
6 & \rightarrow 14,15,9,2 \\ 
7 & \rightarrow 14,16,4,2 \\ 
8 & \rightarrow 15,17,5,2 \\ 
9 & \rightarrow 18,19,3,6 \\ 
10 & \rightarrow 18,20,4,3 \\ 
11 & \rightarrow 19,21,5,3 \\ 
12 & \rightarrow 16,20,13,4 \\ 
13 & \rightarrow 17,21,5,12 \\ 
14 & \rightarrow 22,18,7,6 \\ 
15 & \rightarrow 23,19,8,6 \\ 
16 & \rightarrow 22,12,17,7 \\ 
17 & \rightarrow 23,13,8,16 \\ 
\end{align*}} \\ \hline
\end{tabular}
\end{table}
\begin{table}
\begin{tabular}{||c | c ||}
\hline
Name (cont.) & Graph [Figure or Adjacency Information] (cont.) \\ [0.5 ex]
 \parbox[c]{1em}{Z3 (cont.)} & \parbox[c]{1em}{\begin{align*} 
18 & \rightarrow 24,10,14,9 \\ 
19 & \rightarrow 25,11,15,9 \\ 
20 & \rightarrow 24,12,21,10 \\ 
21 & \rightarrow 25,13,11,20 \\ 
22 & \rightarrow 24,16,23,14 \\ 
23 & \rightarrow 25,17,15,22 \\ 
24 & \rightarrow 20,22,25,18 \\ 
25 & \rightarrow 21,23,19,24 \\ 
 \end{align*}} \\
 \parbox[c]{1em}{Z4} & \parbox[c]{1em}{\begin{align*} 
1 & \rightarrow 2,3,4 \\ 
2 & \rightarrow 5,6,7,1 \\ 
3 & \rightarrow 5,8,1 \\ 
4 & \rightarrow 6,9,1 \\ 
5 & \rightarrow 10,11,3,2 \\ 
6 & \rightarrow 12,13,4,2 \\ 
7 & \rightarrow 11,14,2 \\ 
8 & \rightarrow 10,15,9,3 \\ 
9 & \rightarrow 12,16,4,8 \\ 
10 & \rightarrow 17,8,12,5 \\ 
11 & \rightarrow 18,7,5 \\ 
12 & \rightarrow 19,9,6,10 \\ 
13 & \rightarrow 20,14,6 \\ 
14 & \rightarrow 21,7,13 \\ 
15 & \rightarrow 22,16,8 \\ 
16 & \rightarrow 23,9,15 \\ 
17 & \rightarrow 22,19,18,10 \\ 
18 & \rightarrow 24,21,11,17 \\ 
19 & \rightarrow 23,20,12,17 \\ 
20 & \rightarrow 25,13,21,19 \\ 
21 & \rightarrow 26,14,18,20 \\ 
22 & \rightarrow 23,24,15,17 \\ 
23 & \rightarrow 25,16,19,22 \\ 
24 & \rightarrow 26,18,22 \\ 
25 & \rightarrow 20,26,23 \\ 
26 & \rightarrow 21,24,25 \\ 
 \end{align*}} \\\hline
 \end{tabular}
\end{table}
\begin{table}
\begin{tabular}{||c | c ||}
\hline
Name (cont.) & Graph [Figure or Adjacency Information] (cont.) \\ [0.5 ex]
 \parbox[c]{1em}{Z5} & \parbox[c]{1em}{\begin{align*} 
1 & \rightarrow 2,3,4 \\ 
2 & \rightarrow 5,6,1,7 \\ 
3 & \rightarrow 8,9,10,1 \\ 
4 & \rightarrow 11,9,12,1 \\ 
5 & \rightarrow 13,8,2,14 \\ 
6 & \rightarrow 13,11,15,2 \\ 
7 & \rightarrow 14,15,2 \\ 
8 & \rightarrow 3,16,5 \\ 
9 & \rightarrow 17,4,3 \\ 
10 & \rightarrow 16,18,3 \\ 
11 & \rightarrow 19,4,6 \\ 
12 & \rightarrow 19,20,4 \\ 
13 & \rightarrow 6,21,5 \\ 
14 & \rightarrow 21,22,7,5 \\ 
15 & \rightarrow 21,23,7,6 \\ 
16 & \rightarrow 24,10,22,8 \\ 
17 & \rightarrow 25,20,18,9 \\ 
18 & \rightarrow 26,27,10,17 \\ 
19 & \rightarrow 28,12,23,11 \\ 
20 & \rightarrow 29,27,12,17 \\ 
21 & \rightarrow 30,15,14,13 \\ 
22 & \rightarrow 31,14,16 \\ 
23 & \rightarrow 32,15,19 \\ 
24 & \rightarrow 26,31,16 \\ 
25 & \rightarrow 29,17,26 \\ 
26 & \rightarrow 33,18,24,25 \\ 
27 & \rightarrow 33,18,20 \\ 
28 & \rightarrow 29,19,32 \\ 
29 & \rightarrow 20,33,28,25 \\ 
30 & \rightarrow 32,31,21 \\ 
31 & \rightarrow 34,22,30,24 \\ 
32 & \rightarrow 34,23,30,28 \\ 
33 & \rightarrow 27,34,26,29 \\ 
34 & \rightarrow 31,32,33 \\ 
 \end{align*}} \\ \hline
 \end{tabular}
\end{table}
\begin{table}
\begin{tabular}{||c | c ||}
\hline
Name (cont.) & Graph [Figure or Adjacency Information] (cont.) \\ [0.5 ex]
 \parbox[c]{1em}{Z6} & \parbox[c]{1em}{\begin{align*} 
1 & \rightarrow 2,3,4,5 \\ 
2 & \rightarrow 6,7,8,1 \\ 
3 & \rightarrow 9,10,11,1 \\ 
4 & \rightarrow 7,12,13,1 \\ 
5 & \rightarrow 8,11,14,1 \\ 
6 & \rightarrow 15,16,9,2 \\ 
7 & \rightarrow 17,18,4,2 \\ 
8 & \rightarrow 16,19,5,2 \\ 
9 & \rightarrow 20,21,3,6 \\ 
10 & \rightarrow 22,23,12,3 \\ 
11 & \rightarrow 21,24,5,3 \\ 
12 & \rightarrow 25,26,4,10 \\ 
13 & \rightarrow 18,27,14,4 \\ 
14 & \rightarrow 19,28,5,13 \\ 
15 & \rightarrow 29,20,17,6 \\ 
16 & \rightarrow 30,21,8,6 \\ 
17 & \rightarrow 31,25,7,15 \\ 
18 & \rightarrow 32,13,19,7 \\ 
19 & \rightarrow 33,14,8,18 \\ 
20 & \rightarrow 34,22,9,15 \\ 
21 & \rightarrow 35,11,16,9 \\ 
22 & \rightarrow 36,10,25,20 \\ 
23 & \rightarrow 37,26,24,10 \\ 
24 & \rightarrow 38,28,11,23 \\ 
25 & \rightarrow 39,12,17,22 \\ 
26 & \rightarrow 40,27,12,23 \\ 
27 & \rightarrow 41,13,28,26 \\ 
28 & \rightarrow 42,14,24,27 \\ 
29 & \rightarrow 34,31,30,15 \\ 
30 & \rightarrow 35,33,16,29 \\ 
31 & \rightarrow 39,32,17,29 \\ 
32 & \rightarrow 41,18,33,31 \\ 
33 & \rightarrow 42,19,30,32 \\ 
34 & \rightarrow 36,35,20,29 \\ 
35 & \rightarrow 38,21,30,34 \\ 
 \end{align*}} \\ \hline
 \end{tabular}
\end{table}
\begin{table}
\begin{tabular}{||c | c ||}
\hline
Name (cont.) & Graph [Figure or Adjacency Information] (cont.) \\ [0.5 ex]
 \parbox[c]{1em}{Z6 (cont.)} & \parbox[c]{1em}{\begin{align*} 
36 & \rightarrow 37,22,39,34 \\ 
37 & \rightarrow 23,40,38,36 \\ 
38 & \rightarrow 24,42,35,37 \\ 
39 & \rightarrow 40,25,31,36 \\ 
40 & \rightarrow 26,41,39,37 \\ 
41 & \rightarrow 27,32,42,40 \\ 
42 & \rightarrow 28,33,38,41 \\ 
 \end{align*}} \\ \hline
\end{tabular}
\end{table}

\end{document}